\documentclass[12pt]{article}

\usepackage{graphicx}
\usepackage{bbm}
\usepackage{amsmath}
\usepackage{amsthm}
\usepackage{amsfonts}
\usepackage{amssymb}
\usepackage{xcolor}
\usepackage{comment}
\usepackage{tcolorbox}
\usepackage{enumerate}
\usepackage{mdframed}
\usepackage[dvipsnames]{xcolor}
\tcbuselibrary{breakable}

\newcommand{\eee}{{\rm e}}

\newcommand{\me}{\mathbb{E}}
\newcommand{\mn}{\mathbb{N}}
\newcommand{\mr}{\mathbb{R}}

\DeclareMathOperator{\1}{\mathbbm{1}}

\newcommand{\var}{{\rm Var \,}}

\newcommand{\dd}{{\rm d}}
\newcommand{\ii}{{\rm i}}

\newcommand{\mmp}{\mathbb{P}}

\newtheorem{thm}{Theorem}[section]
\newtheorem{lemma}[thm]{Lemma}

\newtheorem{example}[thm]{Example}

\theoremstyle{definition}

\theoremstyle{remark}
\newtheorem{rem}[thm]{Remark}

\newmdenv[
linecolor=ForestGreen,
linewidth=1mm,
topline=false,
bottomline=false,
rightline=false,
innerleftmargin=10pt,
innerrightmargin=-5pt
]{greenbar}

\newmdenv[
linecolor=NavyBlue,
linewidth=1mm,
topline=false,
bottomline=false,
rightline=false,
innerleftmargin=10pt,
innerrightmargin=-5pt
]{bluebar}

\newmdenv[
linecolor=RawSienna,
linewidth=1mm,
topline=false,
bottomline=false,
rightline=false,
innerleftmargin=10pt,
innerrightmargin=-5pt
]{orangebar}

\begin{document}
\title{On tail behavior of infinite sums of independent indicators}\date{}
\author{Alexander Iksanov\footnote{Faculty of Computer Science and Cybernetics, Taras Shevchenko National University of Kyiv, Ukraine; e-mail address:
iksan@univ.kiev.ua} \ \ and \ \ Valeriya Kotelnikova\footnote{Faculty of Computer Science and Cybernetics, Taras Shevchenko National University of Kyiv, Ukraine; e-mail address: valeria.kotelnikova@unicyb.kiev.ua}}
\maketitle

\begin{abstract}
Let $Y=\sum_{k\ge 1} \mathbbm{1}_{A_k}$ be an infinite sum of the indicators of independent events.
We investigate a precise (as opposed to logarithmic) first-order asymptotic behavior of the tail probabilities $\mathbb{P}\{Y\ge n\}$ and the point probabilities $\mathbb{P}\{Y=n\}$ as $n\to\infty$.
Our analysis provides a reasonably complete classification of the asymptotic behaviors covering most cases of practical interest. These general results are then applied to specific examples where the success probabilities $r_k:=\mathbb{P}(A_k)$ decay polynomially $r_k\sim ck^{-\beta}$ or (sub-, super-) exponentially $r_k\sim ce^{-k^\beta}$, yielding the asymptotic tail and point probabilities in explicit forms.

As briefly discussed in the paper, infinite sums of independent indicators arise naturally in numerous settings as diverse as the range of Poissonized samples, the infinite Ginibre point processes and decoupled renewal processes, and records in the $F^\alpha$ scheme. We also explore connections between our results and the theory of Hayman-admissible functions, total positivity, and the Laguerre-Pólya class of type~I.
\end{abstract}


\hphantom{2020 Mathematics Subject Classification: } Secondary: 30D15

\section{Introduction}\label{sect:intro}

\subsection{General classification} \label{sec:classif}
Denote by $A_1$, $A_2$, \ldots independent events defined on a common probability space $(\Omega,\cal{F},\mmp)$ with
$r_k:=\mmp(A_k)$ for $k\in\mn:=\{1,2,\ldots\}$. Put $Y:=\sum_{k\ge 1} \1_{A_k}$ and note that, by the Borel-Cantelli lemma, the series converges almost surely
if, and only if, $\sum_{k\ge 1} r_k<\infty$. Further, if $\sum_{k\ge 1}r_k=\infty$, then the series diverges almost surely. Therefore, throughout this paper we assume that $\sum_{k\ge 1} r_k<\infty$.

Our earlier works \cite{Buraczewski+Iksanov+Kotelnikova:2024+} and \cite{Iksanov+Kotelnikova:2024} were concerned with laws of the iterated logarithm for infinite sums of the indicators of independent events parameterized by $t$ as $t\to\infty$ (a central limit theorem in this setting is an immediate consequence of the Lindeberg-Feller theorem). In the present work the parameter $t$ is kept fixed or equivalently there is no parametrization. Our purpose is
to find a precise first-order asymptotic behavior of the distribution tail $\mmp\{Y\ge n\}$ and the point probability $\mmp\{Y=n\}$ as $n\to\infty$.

Investigating an asymptotic decay of probabilities that $Y$ takes extremely large values only makes sense provided there are infinitely many positive $r_k$. On the other hand, it is possible that some $r_k$, or even infinitely many $r_k$, are equal to $0$. Removing from $Y$ the indicators $\1_{A_k}$ with $\mmp(A_k)=0$ does not change $Y$. Hence, in what follows it is tacitly assumed that all $r_k$ are positive. Since the distribution of $Y$ is invariant under permutations of the events, we can and do enumerate them so that the sequence $(r_k)_{k\ge1}$ is nonincreasing.

Put $\psi_k(s):=\log(r_k\eee^s+1-r_k)$ for $s\in\mr$ and $k\in\mn$. Then for $s\in\mr$
\begin{multline}\label{eq:me_e^sY}
	\me [\eee^{sY}]=\sum_{n\ge0}\mmp\{Y=n\}\eee^{sn}=\prod_{k\ge 1}(r_k\eee^s+1-r_k)=\exp\Big(\sum_{k\ge1}\psi_k(s)\Big)=:\exp(\psi(s))<\infty.
\end{multline}
The function $\psi$ is infinitely differentiable on $\mr$ with
\begin{equation}\label{eq:series}
	\psi^\prime(s)=\sum_{k\geq 1}\frac{r_k\eee^s}{r_k\eee^s+1-r_k}\quad\text{and}\quad\psi^{\prime\prime}(s)=\sum_{k\ge1}\frac{(1-r_k)r_k\eee^s}{(r_k\eee^s+1-r_k)^2},\quad s\in\mr.
\end{equation}
To obtain the first equality, we have represented $\psi^\prime(s)$ as the limit and then used the dominated convergence theorem to
interchange the passage to a limit and infinite summation.
The second equality is justified analogously. By Fatou's
lemma, $\lim_{s\to\infty}\psi^\prime(s)=+\infty$. The function $\psi^\prime$ is strictly increasing (and consequently,
$\psi$ is strictly convex).
In view of $\psi^\prime(0)=\sum_{k\geq 1}r_k<\infty$, for each integer $n\geq \lfloor \psi^\prime(0)\rfloor+1$, there exists a unique solution in $(0,\infty)$ to the equation $\psi^\prime(s)=n$ that we denote by $s_n$. Plainly, $n\to\infty$ if, and only if, $s_n\to\infty$.

To understand origins of the classification given below, note that $\psi^\prime(s_n)=n$ may be rewritten as
\begin{equation}\label{eq:core}
	\sum_{k=1}^n \frac{1-r_k}{r_k\eee^{s_n}+1-r_k}=\sum_{k\ge n+1}\frac{r_k\eee^{s_n}}{r_k\eee^{s_n}+1-r_k}.
\end{equation}

Here are our main results.

\begin{tcolorbox}[colback=blue!5,colframe=blue!40!black,title={Regime $\psi^{\prime\prime}(s_n)\to\infty$}]
	\begin{thm}\label{thm:caseB}
		As $n\to\infty$, the following are equivalent:
		\begin{enumerate}[(a)]
			\item $\displaystyle\sum_{k=1}^n \frac{1-r_k}{r_k\eee^{s_n}+1-r_k}=\sum_{k\ge n+1}\frac{r_k\eee^{s_n}}{r_k\eee^{s_n}+1-r_k}\to\infty$;
			
			\item $\psi^{\prime\prime}(s_n)\to\infty$;
			
			\item $\sum_{k=1}^n r_k^{-1}\eee^{-s_n}\to\infty$ and $\sum_{k\ge n+1} r_k\eee^{s_n}\to\infty$.
		\end{enumerate}
		If any of (a)-(c) holds, then
		\begin{equation}\label{eq:caseB}
			\mmp\{Y\ge n\}\sim\mmp\{Y=n\}\sim \frac{\exp(\psi(s_n)-s_n\psi^\prime(s_n))
			}{(2\pi \psi^{\prime\prime}(s_n))^{1/2}},\quad n\to\infty.
		\end{equation}
	\end{thm}
\end{tcolorbox}

\begin{tcolorbox}[colback=green!5,colframe=green!40!black,title={Regime $\psi^{\prime\prime}(s_n)\to0$}]
	\begin{thm}\label{thm:caseA}
		As $n\to\infty$, the following are equivalent:
		\begin{enumerate}[(a)]
			\item $\displaystyle\sum_{k=1}^n \frac{1-r_k}{r_k\eee^{s_n}+1-r_k}=\sum_{k\ge n+1}\frac{r_k\eee^{s_n}}{r_k\eee^{s_n}+1-r_k}\to0$;
			
			\item $\psi^{\prime\prime}(s_n)\to0$;
			
			\item $\sum_{k=1}^n r_k^{-1}\eee^{-s_n}\to0$;
			
			\item $\sum_{k\ge n+1} r_k\eee^{s_n}\to0$.
		\end{enumerate}
		If any of (a)-(d) holds, then
		\[
		\mmp\{Y\ge n\}\sim\mmp\{Y=n\}\sim \exp (\psi(s_n)-s_n\psi^\prime(s_n)),\quad n\to\infty.
		\]
	\end{thm}
\end{tcolorbox}

In the remaining case $\lim_{n\to\infty}\psi^{\prime\prime}(s_n)\in (0,\infty)$ our result is not as clean as the previous ones. Here, we require a number of additional assumptions. A complication arising in this case is briefly discussed in Remark \ref{rem:explanation}.

\begin{tcolorbox}[colback=orange!5,colframe=orange!40!black,title={Regime $\lim_{n\to\infty}\psi^{\prime\prime}(s_n)\in (0,\infty)$},breakable]
	\begin{thm}\label{thm:caseC}
		Assume that
		\begin{enumerate}[(a)]
			\item for each $k\in\mn$, there exists a limit $\lim_{n\to\infty}r_{n+k}\eee^{s_n}=:p_k\in[0,\infty)$. Furthermore, there is a sequence $(\alpha_k)_{k\ge1}$ and an integer $N_1\in\mn$ such that
			$$
			\sup_{n\ge N_1}r_{n+k}\eee^{s_n}\le p_k+\alpha_k \quad \text{and} \quad \sum_{k\ge 1} (p_k+\alpha_k)<\infty;
			$$
			
			\item for each $k\in\mn_0:=\mn\cup\{0\}$, there exists a limit $\lim_{n\to\infty}r_{n-k}^{-1}\eee^{-s_n}=:q_k\in[0,\infty)$. Furthermore, there is a sequence $(\beta_k)_{k\ge0}$ and an integer $N_2\in\mn$ such that
			$$
			\sup_{n\ge N_2}r_{n-k}^{-1}\eee^{-s_n}\le q_k+\beta_k \quad \text{and} \quad \sum_{k\ge 0} (q_k+\beta_k)<\infty;
			$$
			
			\item $q_0>0$.
		\end{enumerate}
		
		Then, as $n\to\infty$, the following limits exist:
		\begin{enumerate}[(i)]
			
			\item $\sum_{k=1}^n r_k^{-1}\eee^{-s_n}\to\sum_{k\ge0}q_k\in(0,\infty);$
			
			\item $\sum_{k\ge n+1} r_k\eee^{s_n}\to\sum_{k\ge1}p_k\in(0,\infty);$
			
			\item $\displaystyle\psi^{\prime\prime}(s_n)\to\sum_{k\ge0}\frac{q_k}{(1+q_k)^2}+\sum_{k\ge1}\frac{p_k}{(1+p_k)^2}\in(0,\infty);$
			
			\item $\displaystyle\sum_{k=1}^n \frac{1-r_k}{r_k\eee^{s_n}+1-r_k}\!=\!\sum_{k\ge n+1}\frac{r_k\eee^{s_n}}{r_k\eee^{s_n}+1-r_k}\to\sum_{k\ge0}\frac{q_k}{1+q_k}\!=\!\sum_{k\ge1}\frac{p_k}{1+p_k}\in(0,\infty).$
		\end{enumerate}
		Furthermore,
		$$
		\mmp\{Y\ge n\}\sim\mmp\{Y=n\}\sim c_0 \exp(\psi(s_n)-s_n\psi^\prime(s_n))
		,\quad n\to\infty,
		$$
		with $c_0:=\mmp\{\sum_{m\ge0}\theta_m=0\}$ satisfying $c_0\in(0,1)$. Here, $\theta_0\in\{-1,0\}$ and $\theta_m\in\{-1,0,1\}$ for $m\in\mn$ are independent random variables with the distributions
		\begin{align*}
			&\mmp\{\theta_0=0\}=\frac{1}{1+q_0}\quad\text{and}\quad \mmp\{\theta_0=-1\}=\frac{q_0}{1+q_0},\\
			&\mmp\{\theta_m=0\}=\frac{1+p_mq_m}{(1+p_m)(1+q_m)},\quad\mmp\{\theta_m=1\}=\frac{p_m}{(1+p_m)(1+q_m)}\\&\text{and}\quad \mmp\{\theta_m=-1\}=\frac{q_m}{(1+p_m)(1+q_m)}.
		\end{align*}
	\end{thm}
\end{tcolorbox}
\begin{rem}
	Equivalently, for $m\in\mn$,
	the random variable $\theta_m$ arising in Theorem \ref{thm:caseC} may be represented as $\theta_m=\theta_m^+-\theta_m^-$, where $\theta^+_m$ and $\theta^-_m$ are independent Bernoulli random variables with success probabilities $p_m/(1+p_m)$ and $q_m/(1+q_m)$, respectively. Setting $p_0:=0$, this representation also extends to $m=0$, because $\theta_0^+$ then has success probability $p_0/(1+p_0)=0$ and is therefore equal to $0$ almost surely.
\end{rem}

\begin{rem}
	Put $I(x):=\sup_{s\geq 0}\,(sx-\psi(s))$ for $x>0$. The function $I$ is the Legendre transform of the distribution of $Y$. Since $I(n)=s_n \psi^\prime(s_n)-\psi(s_n)$ for $n\geq \lfloor \psi^\prime(0)\rfloor+1$, the asymptotic relations in Theorems \ref{thm:caseB}, \ref{thm:caseA} and \ref{thm:caseC} can equivalently be written in terms of the Legendre transform.
\end{rem}

\begin{rem}
	Let $X$ be a non-negative integer-valued random variable such that $\me [\eee^{sX}]<\infty$ for all $s\in\mr$ and $\mmp\{X=n\}>0$ for all sufficiently large $n$. In particular, $Y$ satisfies these assumptions. For any such $X$, $\liminf_{n\to\infty}\mmp\{X\ge n\}/\mmp\{X=n\}=1$. However, $\limsup_{n\to\infty}\mmp\{X\ge n\}/\mmp\{X=n\}$ need not be equal to $1$. 
	For instance, let 
	$\mmp\{X=2k\}=\mmp\{X=2k+1\}=c\eee^{-k^2}$ for $k\in\mn_0$,
	where $c$ is a normalizing constant. For this choice, the latter $\limsup$ is equal to $2$.
\end{rem}

\begin{rem}
	This remark is motivated by the fact that $\psi^{\prime\prime}(s_n)$ may have several limit points rather than a unique limit. In such scenarios, the statements in Theorems \ref{thm:caseB}, \ref{thm:caseA} and \ref{thm:caseC}, with the exception of the equivalence $\mathbb{P}\{Y\ge n_j\}\sim\mathbb{P}\{Y=n_j\}$, remain valid when $n$ is replaced by appropriate subsequences $n_j \to \infty$.
\end{rem}

\subsection{General results in action}
Denote by $\zeta$ the (analytic continuation of) the Riemann zeta function. It is defined by $$\zeta(s):=\frac{1}{1-2^{1-s}}\sum_{n\geq 1}\frac{(-1)^{n-1}}{n^s},\quad s\in \mathbb{C}, {\rm Re}\,s>0$$ and
$$
\zeta(s)=2(2\pi)^{s-1}\Gamma(1-s)\sin(\pi s/2)\zeta(1-s),\quad s\in \mathbb{C}, {\rm Re}\,s<0,
$$
where $\Gamma$ denotes the Euler gamma-function.

Define for $k\in\mn_0$
the $k$th Bernoulli polynomial $B_k:\mr\to\mr$ by a generating function
$$
\frac{t\eee^{xt}}{\eee^t-1}=\sum_{k\ge0}B_k(x)\frac{t^k}{k!}.
$$
The first few Bernoulli polynomials are
\begin{align*}
	&B_0(x)=1;\\
	&B_1(x)=x-\frac{1}{2};\\
	&B_2(x)=x^2-x+\frac{1}{6}.
\end{align*}
As usual, $\lfloor x\rfloor$ and $\{x\}$ will denote the integer and fractional part of $x\in\mr$, respectively.

\begin{thm}\label{thm:main}
	The following asymptotic equivalences hold as $n\to\infty$.
	\begin{enumerate}[(a)]
		\item Let $\beta>1$ and $c>0$. If $r_k=ck^{-\beta}$ for $k\in\mn$, then
		\begin{multline*}
			\mmp\{Y\ge n\}~\sim~\mmp\{Y= n\}~\sim~\exp\Big(-\beta n\log n-\beta\Big(\log\Big(\frac{\beta\sin(\pi/\beta)}{\pi c^{1/\beta}}\Big)-1\Big)n\\-\frac{\beta+1}{2}\log n-\frac{\beta}{2}\log(2\beta\sin(\pi/\beta))-\frac{1}{2}\log(2\pi)+\frac{1}{2}\log\beta\Big).
		\end{multline*}
		
		\item Let $\beta\in (0,1)$, $c>0$ and put $\ell=\lfloor (1+\beta)/(2\beta)\rfloor$.
		If $r_k=c\,\exp(-k^{\beta})$ for $k\in\mn$, then
		\vspace{-0pt}
		\begin{multline*}
			\mmp\{Y\ge n\}~\sim~\mmp\{Y=n\}\\~\sim~\Big(\frac{\beta}{2\pi n^{1-\beta}}\Big)^{1/2}
			\exp\Big(-\frac{1}{1+\beta}n^{1+\beta}+\frac{\beta n^{1+\beta}}{1+\beta}\sum_{i=2}^\ell
			\binom{1+1/\beta}{i}(\alpha_\beta(n))^i+n\log c-\frac{1}{2}n^\beta\\+\frac{2}{\beta}\sum_{\text{even }j=0}^{\lfloor 1/\beta\rfloor-1}\binom{1/\beta-1}{j}f_jn^{1-\beta(1+j)}\sum_{i_j=0}^{\ell-j/2-1}\binom{1/\beta-1-j}{i_j}(\alpha_\beta(n))^{i_j}-\zeta(-\beta)\Big).
		\end{multline*}
		Here,
		$f_j:=j!(1-2^{-(j+1)})\zeta(j+2)$ for $j\in\mn_0$, $\alpha_\beta$ is defined by $\alpha_\beta(n)=A_{1,\,\beta}n^{-2\beta}+A_{2,\,\beta}n^{-4\beta}+\ldots+A_{\ell-1,\,\beta}n^{-2(\ell-1)\beta}$, and $A_{i,\,\beta}$ are the constants which can be calculated explicitly.
		
		\item Let $c>0$. If $r_k=c\,\exp(-k)$ for $k\in\mn$, then
		$$
			\mmp\{Y\ge n\}~\sim~\mmp\{Y=n\}~\sim~c_0\exp\Big(-n^2/2+(\log c-1/2)n-1/8+c_1+h_2(1/2)\Big)
		$$
		with the function $h_2$ and the constant $c_1$ as
		defined  in Lemma \ref{lem:beta2}, and a constant $c_0:=\mmp\big\{\sum_{k\ge0}\theta_k=0\big\}$ satisfying $c_0\in(0,1)$.
		Here, the random variables $\theta_k$ are as defined
		in Theorem \ref{thm:caseC}, with $p_k:=\eee^{-k+1/2}$ and $q_k:=\eee^{-k-1/2}$.
		
		\item Let $\beta>1$ and $c>0$.
		If $r_k=c\,\exp(-k^\beta)$ for $k\in\mn$, then
		\begin{multline*}
			\mmp\{Y\ge n\}~\sim~\mmp\{Y=n\}~\sim~ \exp\Big(-\frac{1}{\beta+1}n^{\beta+1}-\frac{1}{\beta}n^{(\beta+1)/2}+\log c\Big(1-\frac{1}{\beta}\Big)n-\frac{1}{2\beta^2}\\-\frac{1}{\beta+1}\sum_{k=1}^{\lfloor \beta\rfloor+1}(-1)^k\binom{\beta+1}{k}B_k\Big(\frac{1}{\beta}n^{-(\beta-1)/2}+\frac{\log c}{\beta} n^{1-\beta}-\frac{\beta-1}{2\beta^2}n^{-\beta}\Big)\\\times\Big(n^{\beta-(k-1)}+\Big(1-\frac{k-1}{\beta}\Big)\1_{\{k\leq \lfloor (\beta+1)/2\rfloor\}}n^{(\beta+1)/2-k}\Big)\Big),
		\end{multline*}
		where $B_k$ are the Bernoulli polynomials.
	\end{enumerate}
\end{thm}
\begin{rem}[to Theorem \ref{thm:main}(b)]
	The paragraph containing formula \eqref{eq:eps_precise1} provides an explanation on how to calculate the coefficients $A_{i,\,\beta}$ appearing in Theorem~\ref{thm:main}(b).
	For instance, for $\beta\in(1/3,1)$, as $n\to\infty$,
	$$
		\mmp\{Y\ge n\}~\sim~\Big(\frac{\beta}{2\pi n^{1-\beta}}\Big)^{1/2}
		\exp\Big(-\frac{1}{1+\beta}n^{1+\beta}+n\log c-\frac{1}{2}n^\beta+\frac{2}{\beta}f_0n^{1-\beta}-\zeta(-\beta)\Big),
	$$
	where $f_0=\pi^2/12$. For $\beta\in(1/4,1/3]$, as $n\to\infty$,
	\begin{multline*}
		\mmp\{Y\ge n\}~\sim~\Big(\frac{\beta}{2\pi n^{1-\beta}}\Big)^{1/2}
		\exp\Big(-\frac{1}{1+\beta}n^{1+\beta}+n\log c-\frac{1}{2}n^\beta+\frac{2}{\beta}f_0n^{1-\beta}\\
		+\frac{1}{\beta}\Big(\frac{1}{\beta}-1\Big)\Big(2\Big(\frac{1}{\beta}-1\Big)c_1^2-4\Big(\frac{1}{\beta}-1\Big)c_1f_0+\Big(\frac{1}{\beta}-2\Big)f_2\Big)n^{1-3\beta}-\zeta(-\beta)\Big),
	\end{multline*}
	where $f_0=c_1=\pi^2/12$ and $f_2=7\pi^4/360$.
\end{rem}

\begin{rem}[to Theorem \ref{thm:main}(d)]
	For instance, for $\beta\in(1,2)$, the result reads: as $n\to\infty$
	$$
		\mmp\{Y\ge n\}~\sim~\mmp\{Y=n\}\sim \exp\Big(-\frac{1}{\beta+1}n^{\beta+1}
		-\frac{1}{2}n^{\beta}+(\log c)n-\frac{\beta}{12}
		n^{\beta-1}+\frac{1}{2}\log c\Big).
	$$
\end{rem}

\section{Some motivating examples}

Our interest in the random variable $Y$ was raised by its appearance in various guises. To be more precise, we now give some examples.
\begin{example} [The range of Poissonized samples and related variables]\label{ex:range}
	\sloppy Let $(\pi(t))_{t\geq 0}$ be a Poisson process of unit intensity, which is independent of $\xi_1$, $\xi_2,\ldots$ independent identically distributed random variables taking infinitely many values in $\mn$ and having the distribution $p_j:=\mmp\{\xi_1=j\}$ for $j\in\mn$. Fix $t>0$ and denote by ${\rm Ran}(\xi, \pi(t))$ the range of the sample $\xi_1,\ldots,\xi_{\pi(t)}$, that is, the number of distinct values attained by the sample. Also, for $k\in\mn$, denote by $\pi_k(t)$ the number of times that the value $k$ appears in the sample. By the thinning property of Poisson processes, the variables $\pi_1(t)$, $\pi_2(t),\ldots$ are independent, and $\pi_k(t)$ has the Poisson distribution of mean $tp_k$. Then ${\rm Ran}(\xi, \pi(t))=\sum_{k\geq 1}\1_{\{\pi_k(t)\geq 1\}}$ is a version of $Y$, with $r_k=1-\eee^{-tp_k}$ for $k\in\mn$.
	
	Analogously, the number of values appearing at least $j$ times, precisely $j$ times, and a positive
	even number of times, given by $\sum_{k\geq 1}\1_{\{\pi_k(t)\geq j\}}$, $\sum_{k\geq 1}\1_{\{\pi_k(t)=j\}}$ and $\sum_{k\geq 1}\1_{\{\pi_k(t)\in 2\mn
		\}}$ are also versions of $Y$, with $r_k=1-\eee^{-tp_k}(1+tp_k+\ldots+(tp_k)^{j-1}/(j-1)!)$, $r_k=\eee^{-tp_k}(tp_k)^j/j!$ and $r_k=\sum_{i\in 2\mn
	}\eee^{-tp_k}(tp_k)^i/i!=(1-\eee^{-tp_k})^2/2
	$, respectively.
	
	Actually, there is a one-to-one correspondence between the distributions of the ranges of Poissonized samples and the distributions of $Y$ for which $r_k<1$ for all $k\in\mn$ (if $r_k=1$ for at least one $k$, then $\mmp\{Y=0\}=0$, whereas a Poissonized sample has an empty range with positive probability).
	Indeed, {\it every} such random variable $Y=\sum_{k\ge 1} \1_{A_k}$ with $\sum_{k\ge1}r_k<\infty$ has the same distribution as ${\rm Ran}(\xi, \pi(t_0))$ with $t_0:=\sum_{k\ge1}|\log(1-r_k)|$ and $p_j=\mmp\{\xi_1=j\}=-\log(1-r_j)/t_0$ for $j\in\mn$.
	With this at hand, we can derive a universal upper bound for $\mmp\{Y\ge n\}$. Since ${\rm Ran}(\xi, \pi(t_0))\le\pi(t_0)$ almost surely,
	we conclude with the help of Stirling's formula that as $n\to\infty$
	$$
	\mmp\{Y\ge n\}\le\mmp\{\pi(t_0)\ge n\}\sim\mmp\{\pi(t_0)= n\}=\eee^{-t_0}\frac{t_0^n}{n!}\sim \frac{\eee^{-t_0}}{\sqrt{2\pi n}}(t_0\eee)^n\eee^{-n\log n}.
	$$
	In particular, this implies that $\me[\eee^{sY}]<\infty$ for all $s>0$.
\end{example}
\begin{example}[The infinite Ginibre point process and decoupled renewal processes]
	Let $\xi_1$, $\xi_2,\ldots$ be independent copies of a nonnegative random variable~$\xi$. Put $S_n=\xi_1+\ldots+\xi_n$ for $n\in\mn$. The random sequence $(S_n)_{n\ge1}$ is called {\it standard random walk} with nonnegative jumps. Let $\hat S_1$, $\hat S_2,\ldots$ be independent random variables such that, for each $n\in\mn$, $\hat S_n$ has the same distribution as $S_n$. Following \cite{Alsmeyer+Iksanov+Kabluchko:2024+}, we call the sequence $(\hat S_n)_{n\ge1}$ a decoupled standard random walk.
	
	Let $\Theta$ be the infinite Ginibre point process on the set of complex numbers $\mathbb{C}$, that is, a simple point process such that,
	for any $k\in\mn$ and any pairwise disjoint Borel subsets $B_1,\ldots, B_k$ of $\mathbb{C}$, $$\me\Big[\prod_{j=1}^k \Theta(B_j)\Big]=\int_{B_1\times\ldots\times B_k}{\rm det}\big(\eee^{z_i\bar z_j}\big)_{1\leq i,j\leq k}\rho({\rm d}z_1)\ldots\rho({\rm d}z_k),$$
	where $\bar{z}$ is the complex conjugate of $z\in \mathbb{C}$, {\rm det} denotes the determinant, and $\rho$ is a Gaussian distribution defined by $\rho({\rm d}z)=\pi^{-1}\eee^{-|z|^2}{\rm d}z$ for $z\in\mathbb{C}$.
	Fix any $t>0$. According to an infinite version of Kostlan’s result \cite{Kostlan:1992}, $\Theta\{z\in\mathbb{C}: |z|<t^{1/2}\}$ has the same distribution as $\sum_{n\geq 1}\1_{\{\hat S_n\leq t\}}$, which is an instance of $Y$. Here, $\hat S_1$ has an exponential distribution of unit mean. More generally, the variable $\sum_{n\geq 1}\1_{\{\hat S_n\leq t\}}$ which is the value of a decoupled renewal process at time $t$ is an instance of $Y$ even if the distribution of $\hat S_1$ is arbitrary nondegenerate.
\end{example}
\begin{example}[Records in the $F^\alpha$ scheme]
	Let $(\alpha_k)_{k\geq 1}$ be positive numbers satisfying $\sum_{k\geq 1}\alpha_k<\infty$ and $F$ a continuous distribution function on $\mr$. Let $\eta_1$, $\eta_2,\ldots$ be independent random variables such that $\mmp\{\eta_k\leq x\}=F^{\alpha_k}(x)$ for $x\in\mr$. Noting that $\eta_1$ is a record value, put, for $i\geq 2$,
	$$
		A_i:=\{\eta_i>\max(\eta_1,\ldots, \eta_{i-1})\}=\{\text{a record in the sequence}~ (\eta_k)_{k\geq 1}~\text{occurs at time}~ i\}.
	$$
	According to Lemma 25.4 on p.~109 in \cite{Nevzorov:2001}, $r_i=\mmp(A_i)=\alpha_i/(\alpha_1+\ldots+\alpha_i)$ for $i\geq 2$. Since $\sum_{i\geq 1}\alpha_i<\infty$ is equivalent to $\sum_{i\geq 1}r_i<\infty$, we infer that the variable $1+\sum_{i\geq 2}\1_{A_i}$ which is the number of records in the sequence $(\eta_k)_{k\geq 1}$ is a version of $Y$.
\end{example}
\begin{example}\label{ex:gnedin}
	This example was kindly supplied by Alexander Gnedin. For $\lambda>0$, let $Z$ be a random variable with the distribution $$\mmp\{Z=n\}=\frac{1}{{\rm sinh}(\lambda)}\frac{\lambda^{2n+1}}{(2n+1)!},\quad n\in\mn_0$$ (the right-hand side is a restriction of the Poisson distribution to odd numbers). An application of the Weierstrass factorization theorem to the generating function of $Z$
	yields
	$$
		\me [s^{Z}]=\frac{{\rm sinh}(\lambda s^{1/2} )}{{\rm sinh}(\lambda)s^{1/2}}=\frac{\lambda}{{\rm sinh}(\lambda)}\prod_{k\geq 1}\Big(1+\frac{\lambda^2s}{(\pi k)^2}\Big)=\prod_{k\geq 1}\Big(\frac{(\pi k)^2}{(\pi k)^2+\lambda^2}+\frac{\lambda^2}{(\pi k)^2+\lambda^2}s\Big),~~ s\in\mr.
	$$
	Thus, $Z$ is a version of $Y$ with $r_k=\lambda^2 ((\pi k)^2+\lambda^2)^{-1}$ for $k\in\mn$. We note in passing that an application of the Stirling formula yields
	\begin{equation}\label{eq:gnedin}
		\mmp\{Z=n\}~\sim~\frac{\lambda}{4\pi^{1/2}{\rm sinh}(\lambda)}\eee^{-2n\log(2n)}(\lambda \eee)^{2n}n^{-3/2},\quad n\to\infty.
	\end{equation}
	
	By a similar argument, a random variable $\vartheta$ with the distribution  $$\mmp\{\vartheta=n\}=\frac{1}{{\rm cosh}(\lambda)}\frac{\lambda^{2n}}{(2n)!},\quad n\in\mn_0$$ is a version of $Y$ with $r_k=4\lambda^2 ((\pi(2k-1))^2+4\lambda^2)^{-1}$ for $k\in\mn$.
	
	Since both distributions appearing in this example are close to the Poisson distribution, it seems curious that, for no $r_k$, does $Y$ have a Poisson distribution. This follows from Raikov's theorem, see, for instance, Theorem 5.1.2 on p.~132 in \cite{Linnik+Ostrovskii:1977}, which states that if the sum of two independent random variables has a (possibly shifted) Poisson distribution, then so does each summand. It is equally remarkable that the variable $\sum_{k\geq 1}\1_{A_k}\1_{A_{k+1}}$, with $r_k=1/k$ for $k\in\mn$, does have the Poisson distribution of unit mean. The partial history of this intriguing observation can be traced via the references given in \cite{Csorgo+Wu:2000}.
\end{example}

\section{Connection to other fields of mathematics}

\subsection{Connection to analytic function theory}
Assuming a stronger condition $\lim_{s\to\infty}\psi^{\prime\prime}(s)=\infty$, instead of $\lim_{n\to\infty}\psi^{\prime\prime}(s_n)=\infty$, the asymptotic relation for $\mmp\{Y=n\}$ in Theorem \ref{thm:caseB} follows from the theory of \textit{Hayman-admissible} (or simply \textit{admissible}) functions developed by Hayman~\cite{Hayman:1956} in 1956.

Let $\mathbb{C}$ denote the set of complex numbers. Consider a class of functions $f(z)=\sum_{n\ge0}a_nz^n$ for $z\in\mathbb{C}$ which are \textit{admissible}.
For the full definition of \textit{admissible} function we refer to \cite{Hayman:1956}. Here, we only state a sufficient condition (Theorem~11 in \cite{Hayman:1956}).

\begin{thm}\label{thm:hayman_f}
	A function $f$ is \textit{admissible} provided that
	\begin{itemize}
		\item $f$ is an entire function
		of genus zero, that is, $f$ admits a representation
		$$
		f(z)=\text{\rm const }z^m\prod_{k\ge1}\Big(1-\frac{z}{z_k}\Big),
		$$
		where $m$ is the order of the zero $z=0$ (if $f(0)\ne0$, then $m=0$), $z_k$ are nonzero roots of $f$ (the number of $z_k$ may be finite),
		and $\sum_{k\ge1}|z_k|^{-1}<\infty$;
		
		\item $f(x)>0$ for large $x>0$;
		
		\item $\displaystyle b(x):=x\frac{f^\prime(x)}{f(x)}+x^2\Big(\frac{f^\prime(x)}{f(x)}\Big)^\prime\to\infty$ as $x\to\infty$;
		
		\item for some $\delta>0$, $f$ has finitely many zeros in the angle $|\arg z|\le\pi/2+\delta$.
	\end{itemize}
\end{thm}

Define a function $a$ by $a(x)=x\frac{f^\prime(x)}{f(x)}$ for $x\in\mr$. According to Corollary I in~\cite{Hayman:1956}, $a$ is positive and increasing on $(x_0,\infty)$ for some $x_0>0$.  For large $n\in\mn$, let $x_n$
be a unique solution to $a(x)=n$. We omit the explanation, which can be found in \cite{Hayman:1956}, why the solution exists and is unique.
The main result of Hayman's paper is the following theorem.

\begin{thm}\label{thm:hayman}
	For an admissible function $f$,
	\begin{equation}\label{eq:hayman}
		a_n\sim \frac{f(x_n)}{x_n^n\sqrt{2\pi b(x_n)}},\quad n\to\infty.
	\end{equation}
\end{thm}

\begin{proof}[Analytic proof of the main relation of Theorem \ref{thm:caseB} under $\lim_{s\to\infty}\psi^{\prime\prime}(s)=\infty$]
	Put
	$$f(z):=\sum_{n\ge0} \mmp\{Y=n\}z^n=\prod_{k\ge1}(1-r_k+r_kz).$$ Since $a_n=\mmp\{Y=n\}$,
	the asymptotic behavior of $\mmp\{Y=n\}$ as $n\to\infty$ follows from Theorem \ref{thm:hayman}.
	
	Here are details. First, we point out representations for functions $a$ and $b$. Recall that the function $\psi$ was defined in \eqref{eq:me_e^sY}.
	We start with
	$$
	\frac{f^\prime(x)}{f(x)}=\Big(\sum_{k\ge1}\log(1-r_k+r_kx)\Big)^\prime=(\psi(\log x))^{\prime}=\sum_{k\ge1}\frac{r_k}{1-r_k+r_kx}
	$$
	for $x\ge1$, where we have used \eqref{eq:series} for the last equality.
	Hence, for $x\ge1=:x_0$, the function $a$ is defined by
	$$a(x)=\sum_{k\ge1}\frac{r_kx}{1-r_k+r_kx}.$$ 
	Invoking \eqref{eq:series} again,
	we obtain that
	$$
	\Big(\frac{f^\prime(x)}{f(x)}\Big)^\prime=(\psi(\log x))^{\prime\prime}=-\sum_{k\ge1}\frac{r_k^2}{(1-r_k+r_kx)^2}
	$$
	for $x\ge1$. Therefore, $$b(x)=\sum_{k\ge1}\frac{r_kx}{1-r_k+r_kx}-\sum_{k\ge1}\frac{r_k^2x^2}{(1-r_k+r_kx)^2}=\sum_{k\ge1}\frac{(1-r_k)r_kx}{(1-r_k+r_kx)^2}.$$ 
	Summarizing, we have, for $s\ge 0$,
	\begin{align}
		f(\eee^s)=&\exp(\psi(s));\label{eq:hayman_f}\\
		a(\eee^s)=&\,\psi^\prime(s);\label{eq:hayman_a}\\
		b(\eee^s)=&\,\psi^{\prime\prime}(s).\label{eq:hayman_b}
	\end{align}
	
	Now, we are ready to check the sufficient conditions for admissibility stated in Theorem~\ref{thm:hayman_f}:
	\begin{itemize}
		\item for all $z>0$, $f(z)>0$;
		
		\item the zeros of $f$ are $z_k=-(1-r_k)/r_k$, whence $\arg z_k=\pi$ for all $k\in\mn$;
		
		\item Hadamard's canonical representation for $f$ is
		$$
		f(z)=\prod_{j\ge1}(1-r_j)\prod_{k\ge1}\Big(1+\frac{r_kz}{1-r_k}\Big)=\prod_{j\ge1}(1-r_j)\prod_{k\ge1}\Big(1-\frac{z}{z_k}\Big),
		$$
		which implies that the genus of $f$ is zero;
		
		\item \sloppy in view of \eqref{eq:hayman_b}, the condition $\lim_{x\to\infty}b(x)=\infty$
		is equivalent to $\lim_{s\to\infty}\psi^{\prime\prime}(s)=\infty$.
	\end{itemize}
	
	In view of \eqref{eq:hayman_a}, the equation $\psi^\prime(s_n)=n$
	is equivalent to $a(x_n)=n$ with
	$x_n=\eee^{s_n}$. With this and \eqref{eq:hayman_f} at hand, \eqref{eq:hayman} reads
	$$
	\mmp\{Y=n\}\sim\frac{f(\eee^{s_n})}{(\eee^{s_n})^n\sqrt{2\pi b(\eee^{s_n})}}=\frac{\exp(\psi(s_n)-s_n\psi^\prime(s_n))}{\sqrt{2\pi \psi^{\prime\prime}(s_n)}}, \quad n\to\infty.
	$$
	A proof of $\mmp\{Y\ge n\}\sim\mmp\{Y=n\}$ does not follow directly from Theorem \ref{thm:hayman} and has to be worked out separately, for instance, by verifying relation \eqref{eq:inter1} given below.
\end{proof}

\sloppy The definition of \textit{admissible} function includes the condition $\lim_{x\to\infty}b(x)=\infty$. According to~\eqref{eq:hayman_b}, this limit relation is equivalent to $\lim_{s\to\infty}\psi^{\prime\prime}(s)=\infty$. Hence, Theorem~\ref{thm:hayman} only applies to the distributions of $Y$ satisfying $\lim_{s\to\infty}\psi^{\prime\prime}(s)=\infty$.

To the best of our knowledge, there is no counterpart of Theorem \ref{thm:hayman} in the case $\liminf_{x\to\infty}b(x)<\infty$.
However, there is a recent article \cite{Fernandez+Macia:2023+} investigating the asymptotic behavior of the $n$th order convolutions $\sum_{i_1+\ldots+i_n=k}\mmp\{V=i_1\}\mmp\{V=i_2\}\cdot\ldots\cdot \mmp\{V=i_n\}$ for $k\in\mn$ as $n\to\infty$, where
$V$ is a random variable taking values in $\mn_0$.
The condition $\lim_{x\to\infty}b(x)=\infty$ is not required in that work.

\subsection{Connection to total positivity}\label{sec:positive}

This connection was noticed by Alexander Gnedin. According to \cite{Aissen+Edrei+Schoenberg+Whitney:1951}, a real matrix, finite or infinite, is called \textit{totally positive} if its minors of all orders are nonnegative. A real sequence $a_0,a_1,\ldots (a_0=1)$ is called \textit{totally positive} if the corresponding infinite matrix $D$ is totally positive
\[
D=\begin{pmatrix}
	a_0 & 0 & 0 & \ldots\\
	a_1 & a_0 & 0 &\ldots\\
	a_2 & a_1 & a_0 &\ldots\\
	\vdots & \vdots & \vdots &\ddots
\end{pmatrix}.
\]
\noindent The same paper \cite{Aissen+Edrei+Schoenberg+Whitney:1951} provides necessary and sufficient conditions for total positivity of a sequence.
\begin{thm}\label{thm:totally_positive}
	The sequence $a_0,a_1,\ldots (a_0=1)$ is totally positive if, and only if, its generating function $f$ is of the form
	\[
	f(z)=\sum_{k\ge0}a_kz^k=\eee^{\gamma z}\prod_{k\ge1}\frac{1+\alpha_k z}{1-\beta_k z},
	\]
	where $\gamma\ge0$,
	$\alpha_k\ge0$,
	$\beta_k\ge 0$ for all $k\in\mn$, $\sum_{k\ge1}\alpha_k<\infty$ and $\sum_{k\ge1}\beta_k<\infty$.
\end{thm}

Assume that one is interested in a precise first-order asymptotic behavior of $a_n$, the $n$th element of a totally positive sequence,
as $n\to\infty$. Without loss of generality, we can and do assume that all $\beta_k<1$. Indeed, if this is not the case, put
$M:=\max_{k\ge1}\beta_k+1$ and consider
\[
f\Big(\frac{z}{M}\Big)=\sum_{k\ge0}\frac{a_k}{M^k}z^k=\eee^{\frac{\gamma}{M}
	z}\prod_{k\ge1}\frac{1+\frac{\alpha_k}{M} z}{1-\frac{\beta_k}{M} z}.
\]
Then, with $a_k^\prime:=\frac{a_k}{M^k}$ for $k\in\mn_0$, by Theorem \ref{thm:totally_positive}, $(a_k^\prime)_{k\ge0}$ is also a totally positive sequence but with all $\beta_k^\prime:=\frac{\beta_k}{M}<1$. Plainly, a precise first-order asymptotic behavior of $a_n$ immediately follows from that of $a_n^\prime$.

For $\lambda>0$ and $p\in (0,1)$, let $\textrm{Poiss}(\lambda)$, $\textrm{Bern}(p)$ and $\textrm{Geom}(p)$ denote random variables having the Poisson distribution with parameter $\lambda$, that is, $\mmp\{\textrm{Poiss}(\lambda)=j\}=\eee^{-\lambda}
\frac{\lambda^j}{j!}$ for $j\in\mn_0$, the Bernoulli distribution with success probability $p$, that is, $\mmp\{\textrm{Bern}(p)=1\}=p$ and $\mmp\{\textrm{Bern}(p)=0\}=1-p$, and a geometric distribution with parameter $p$, that is, $\mmp\{\textrm{Geom}(p)=j\}=(1-p)^j p$ for $j\in\mn_0$.

Let $f$ be a generating function as in Theorem \ref{thm:totally_positive}. Then $z\mapsto f(z)/f(1)$ is the generating function of a random variable $X$ say, taking nonnegative integer values and satisfying $\mmp\{X<\infty\}=1$, that is,
\[
\frac{f(z)}{f(1)}=\sum_{n\geq 0}\mmp\{X=n\}z^n.
\]
In other words, $a_n$ is now represented as $f(1)\mmp\{X=n\}$. The variable $X$ has the same distribution as $\textrm{Poiss}(\gamma)+\sum_{k\ge1}\textrm{Bern}(\frac{\alpha_k}{1+\alpha_k})+\sum_{\ell\ge1}\textrm{Geom}(1-\beta_\ell)$, where all the summands are independent.
In this setting, our paper is concerned with a precise first-order asymptotic behavior of $\mmp\{\sum_{k\ge1}\textrm{Bern}(\frac{\alpha_k}{1+\alpha_k})=n\}$ as $n\to\infty$. The remaining steps are
\begin{itemize}
	\item to derive a precise first-order asymptotic behavior of
	$\mmp\{\sum_{\ell\ge1}\textrm{Geom}(1-\beta_\ell)=n\}$ as $n\to\infty$,
	
	\item to deduce an asymptotic behavior of $\mmp\{X=n\}$ using the fact that
	$X$ is the sum of three independent
	random variables with known asymptotic behaviors of their point probabilities.
\end{itemize}

Our results are useful for deriving the precise asymptotic behavior
of $a_n$ when the geometric component is absent, that is, $\beta_\ell=0$ for all $\ell \ge 1$. Conversely, if at least one $\beta_\ell>0$, the geometric sum dominates. Denoting $\beta_{\rm max}:=\max_{\ell\ge 1}\beta_\ell$, one can show that $a_n\sim f(1)\me \beta_{\rm max}^{-\theta}\mmp\{\sum_{\ell\ge1}\textrm{Geom}(1-\beta_\ell)=n\}$, where $\theta$ denotes a random variable distributed as $\textrm{Poiss}(\gamma)+\sum_{k\ge1}\textrm{Bern}(\frac{\alpha_k}{1+\alpha_k})$.

\textit{Totally positive} sequences, and particularly Theorem \ref{thm:totally_positive}, ``reincarnate'' in other areas of mathematics under new circumstances:
\begin{itemize}
	\item Theorem \ref{thm:totally_positive} was independently obtained 13 years later in the theory of representations of the infinite symmetric group in \cite{Thoma:1964};
	
	\item in \cite{Vershik+Kerov:1981}, Theorem \ref{thm:totally_positive} is reproved via the ergodic method, and a connection with the Young graph is established;
	
	\item invoking Theorem \ref{thm:totally_positive} an implicit classification for Schur-positive specializations was given in \cite{Kerov+Vershik:1986};
	
	\item a consequence of Theorem \ref{thm:totally_positive} is used in combinatorics in various guises,
	see, for instance, \cite{Brenti:1996} for an explanation of the fact that many polynomials arising in combinatorics are known (or conjectured) to have only real zeros.
\end{itemize}

\subsection{Connection to Laguerre-P\'olya classes}

This connection was pointed out by Alexander Marynych.
By Theorem \ref{thm:totally_positive}, totally positive sequences with \textit{entire} generating functions are precisely those whose corresponding random series lack the geometric component, or equivalently, those with generating functions of the form
$$
f(z)=\eee^{\gamma z}\prod_{k\ge1}(1+\alpha_k z),
$$
where $\gamma\ge0$,
$\alpha_k\ge0$ for all $k\in\mn$ and $\sum_{k\ge1}\alpha_k<\infty$.

These generating functions belong to the so-called \textit{Laguerre-P\'olya class of type I}, which consists of real entire functions $g$ admitting the representation
$$
g(z)=cz^n\eee^{\gamma z}\prod_{k\ge1}\Big(1+\frac{z}{x_k}\Big),
$$
where $c\in\mr$, $\gamma\ge0$, $n\in\mn_0$, $x_k>0$ for all $k\in\mn$ and $\sum_{k\ge1}x_k^{-1}<\infty$.

The Laguerre-P\'olya class of type I is of particular interest in complex analysis, since it is the class of local uniform limits of polynomials with only real and non-positive roots, see Theorem A in \cite{Nguyen+Vishnyakova:2021}.

\section{Our approach}\label{sect:approach}

\noindent {\bf A change of measure}.
Put $\eta_k:=\1_{A_k}$ for $k\in\mn$, so that $Y=\sum_{k\ge1} \eta_k$.
For each $s>0$, define a new probability measure $\mmp^{(s)}$ by
\begin{equation}\label{eq:choice_of_measure(g)}
	\me^{(s)}[g(\eta_1,\eta_2,\ldots,\eta_k)]=\frac{\me [\eee^{sY}g(\eta_1,\eta_2,\ldots,\eta_k)]}{\me [\eee^{sY}]}
\end{equation}
for all $k\in\mn$ and all bounded measurable functions $g:\mr^k\to\mathbb{C}$. Here, $\me^{(s)}$ denotes expectation with respect to $\mmp^{(s)}$. The change of measure only affects the marginal distributions of $\eta_1$, $\eta_2,\ldots$ In particular, these random variables are still independent under $\mmp^{(s)}$. Moreover, the variables $\eta_k$ remain Bernoulli distributed,
with the respective new success probabilities 
$\frac{r_k\eee^s}{r_k\eee^s+1-r_k}$.

For any measurable bounded function $f:\mr\to\mr$,  put $g(x_1,x_2,\ldots,x_k)=f(\sum_{i=1}^k x_i)$ for $(x_1,x_2,\ldots,x_k)\in\mr^k$. Using \eqref{eq:choice_of_measure(g)} with the so defined $g$ and sending $k\to\infty$ we arrive at
\begin{equation}\label{eq:choice_of_measure}
	\me^{(s)}[f(Y)]=\frac{\me [\eee^{sY}f(Y)]}{\me [\eee^{sY}]}.
\end{equation}
A specialization of \eqref{eq:choice_of_measure}, with any $n\in\mn$ fixed and $f(z):=\eee^{-sz}\1_{\{z=n\}}$, yields
\begin{equation}\label{eq:Y>x1}
	\mmp\{Y=n\}=\me [\eee^{s(Y-n)}]\mmp^{(s)}\{Y=n\}.
\end{equation}

\noindent {\bf A formula for $\mmp\{Y=n\}$}. Using \eqref{eq:choice_of_measure}
we infer
\begin{equation}\label{eq:me(s)Y}
	\me^{(s)}[Y]=\frac{\me [\eee^{sY}Y]}{\me [\eee^{sY}]}=
	\psi^\prime(s).
\end{equation}
An alternative way to derive \eqref{eq:me(s)Y} is to recall the series representation \eqref{eq:series} for $\psi^\prime$:
$$
\me^{(s)}[Y]=\sum_{k\ge1}\mmp^{(s)}(A_k)=\sum_{k\ge1}\frac{r_k\eee^s}{r_k\eee^s+1-r_k}=\psi^\prime(s).
$$
We shall use \eqref{eq:Y>x1} with $s=s_n$ as defined in Section \ref{sect:intro}, that is, the unique solution to $\me^{(s)}[Y]=n$, see \eqref{eq:me(s)Y}. Invoking $$\eee^{-s_n n}=\eee^{-s_n\me^{(s_n)}[Y]}=\eee^{-s_n\psi^\prime(s_n)}$$
we obtain a formula
\begin{equation}\label{eq:Y>x2}
	\mmp\{Y=n\}=\exp (\psi(s)-s\psi^\prime(s))
	\mmp^{(s)}\{Y=n\}
\end{equation}
with $s=s_n$, which is a starting point of the subsequent analysis.

\begin{rem}\label{rem}
	Typically, it is hard to find an inverse function for $s\mapsto\me^{(s)}[Y]$. Solving the equation $\me^{(s)}[Y]=n$, even approximately in $s$, is then a non-trivial task.
	We suggest to deal with this complication in the following way. We solve the equation $\me^{(s)}[Y]=n$ approximately, that is, we find a solution $s_n$ to $\me^{(s)}[Y]+o(1/s)=n$. Then
	$$
	\eee^{-s_n n}=\eee^{-s_n(\me^{(s_n)}[Y]+o(1/s_n))}\sim\eee^{-s_n\psi^\prime(s_n)},\quad n\to\infty.
	$$
	Hence, for $s=s_n$,
	\begin{equation}\label{eq:alt}
		\mmp\{Y=n\}\sim\exp (\psi(s)-s\psi^\prime(s))
		\mmp^{(s)}\{Y=n\},\quad n\to\infty.
	\end{equation}
\end{rem}

The classification presented in Theorems \ref{thm:caseB}, \ref{thm:caseA} and \ref{thm:caseC} is essentially determined by the asymptotic behavior of the variance $\var^{(s_n)}[Y]:= \me^{(s_n)}[Y^2]-(\me^{(s_n)}[Y])^2$.

\noindent {\bf The asymptotic behavior of $\var^{(s_n)}Y$} as $n\to\infty$. Put
$Y_0(s):=Y-\me^{(s)}[Y]$. For notational simplicity, we often write $Y_0$ in place of $Y_0(s)$. According to \eqref{eq:choice_of_measure} and \eqref{eq:me(s)Y},
$$
\me^{(s)}[Y^2]=\frac{\me [\eee^{sY}Y^2]}{\me[\eee^{sY}]}=\frac{(\me[\eee^{sY}])^{\prime\prime}}{\me[\eee^{sY}]}\quad\text{and}\quad \me^{(s)}[Y]=\frac{(\me[\eee^{sY}])^{\prime}}{\me[\eee^{sY}]}.
$$
Invoking \eqref{eq:series}, we conclude that 
\begin{multline}\label{eq:var_Y0}
	\me^{(s)}[Y_0^2]=\me^{(s)}[Y^2]-(\me^{(s)}[Y])^2
	=\frac{(\me[\eee^{sY}])^{\prime\prime}}{\me[\eee^{sY}]}-\Big(\frac{(\me[\eee^{sY}])^{\prime}}{\me[\eee^{sY}]}\Big)^2=\Big(\frac{(\me[\eee^{sY}])^{\prime}}{\me[\eee^{sY}]}\Big)^\prime\\=\psi^{\prime\prime}(s)=\sum_{k\ge 1}\frac{(1-r_k)r_k\eee^s}{(r_k\eee^s+1-r_k)^2}.
\end{multline}
Analogously to the computation of $\me^{(s)}[Y]$, formula
\eqref{eq:var_Y0} can be obtained directly from the independence of the variables $\eta_k$: 
$$\var^{(s)}[Y]=\sum_{k\ge1}\var^{(s)}[\eta_k]=\sum_{k\ge 1}\frac{r_k\eee^s}{r_k\eee^s+1-r_k}\Big(1-\frac{r_k\eee^s}{r_k\eee^s+1-r_k}\Big)=\psi^{\prime\prime}(s).
$$

\section{An extension to asymptotically equivalent probabilities}
As of now, the application scope of Theorem \ref{thm:main} is limited just to $r_k=ck^{-\beta}$ for $\beta>1$ or $r_k=c\eee^{-k^\beta}$ for $\beta>0$. We aim at broadening the obtained results for $(u_k)_{k\ge 1}$, where $u_k\sim r_k$ as $k\to\infty$. First, we state Theorem \ref{thm:equiv}, a general result. Then
we illustrate Theorem \ref{thm:equiv} with specific examples.

\begin{thm}\label{thm:equiv}
	Let $(u_k)_{k\ge 1}$ and $(r_k)_{k\ge 1}$ be the success probabilities satisfying
	$\sum_{k\ge 1}u_k<\infty$ and $\sum_{k\ge 1}r_k<\infty$. Let $Y_u$ and $Y_r$ be the corresponding infinite sums of independent indicators and $\psi_u$ and $\psi_r$ the corresponding counterparts of $\psi$.
	
	Assume that $u_k=r_k(1+\varepsilon_k)$ for $k\in\mn$, and $\sum_{k\ge 1}|\varepsilon_k|<\infty$. Denote by $s_n$ the solution to $\psi_r^\prime(s)=n$.
	Suppose that $(r_k)_{k\ge 1}$ satisfies the assumptions of Theorem \ref{thm:caseB} or
	\ref{thm:caseA} or \ref{thm:caseC} and additionally that
	$$
	\eee^{-s_n}\sum_{k=1}^{n}|\varepsilon_k|/r_k=o(1/s_n)\quad\text{and}\quad \eee^{s_n}\sum_{k\ge n+1}|\varepsilon_k|r_k=o(1/s_n),\quad n\to\infty.
	$$
	Then $(u_k)_{k\ge 1}$ also satisfies the assumptions of the same theorem and
	$$
	\mmp\{Y_u\ge n\}~\sim~\mmp\{Y_u=n\}~\sim~A
	\mmp\{Y_r=n\}~\sim~A
	\mmp\{Y_r\ge n\},\quad n\to\infty
	$$
	with a constant $A=
	\prod_{k\ge1}(1+\varepsilon_k)\in (0,\infty)$.
\end{thm}

\noindent{\bf Application 1.} Recovering the result of Example \ref{ex:gnedin}.

Put $r_k=\lambda^2(\pi k)^{-2}$ and $u_k=\lambda^2 ((\pi k)^2+\lambda^2)^{-1}$ for $k\in\mn$. Then $\varepsilon_k=-\lambda^2((\pi k)^2+\lambda^2)^{-1}$ and $\sum_{k\ge 1}|\varepsilon_k|<\infty$. From equation \eqref{eq:s} (with $\beta=2$) in the proof of Theorem \ref{thm:main}(a) below it follows that $s_n=2\log \big(n+\frac{1}{2}\big)+2\log\big(\frac{2}{\lambda}\big)+o(1)$ as $n\to\infty$. With this at hand, we check that, as $n\to\infty$,
$$
\eee^{-s_n}\sum_{k=1}^{n}|\varepsilon_k|/r_k=\eee^{-s_n}\sum_{k=1}^{n}\frac{(\pi k)^2}{(\pi k)^2+\lambda^2}\le n\eee^{-s_n}=O(n^{-1})
=o((\log n)^{-1})
$$
and
$$
\eee^{s_n}\sum_{k\ge n+1}|\varepsilon_k|r_k\le\eee^{s_n}\sum_{k\ge n+1}\lambda^4(\pi k)^{-4}=\eee^{s_n} O(n^{-3})=
O(n^{-1})=o((\log n)^{-1}).
$$
Therefore, all the conditions of Theorem \ref{thm:equiv} are satisfied. To calculate the constant $A$, write
$$
A=\prod_{k\ge 1}\Big(1-\frac{\lambda^2}{(\pi k)^2+\lambda^2}\Big)=\frac{1}{\prod_{k\ge 1}(1+\lambda^2(\pi k)^{-2})}=\frac{\lambda}{\sinh (\lambda)}.
$$
Applying Theorem \ref{thm:equiv} together with the case $\beta=2$ of Theorem \ref{thm:main}(a) we obtain
$$
\mmp\{Z\ge n\}\sim\mmp\{Z=n\}\sim\frac{\lambda}{4\pi^{1/2}\sinh (\lambda)}\exp(-2n\log (2n)+2n)\lambda^{2n}n^{-3/2}
$$
as $n\to\infty$, which is in full agreement with \eqref{eq:gnedin}.

\noindent{\bf Application 2.} The range of Poissonized samples.

Consider the setting of Example \ref{ex:range} with $p_k=c\eee^{-k^\beta}$ for $k\in\mn$ and $\beta>1$. Fix any $t>0$. The variable ${\rm Ran}(\xi, \pi(t))$ is a version of $Y$ with $u_k=1-\eee^{-tp_k}$ for $k\in\mn$.
Put
$r_k=tp_k=tc\eee^{-k^\beta}$ for $k\in\mn$. Then $\sum_{k\ge 1}r_k<\infty$.
Further, in view of $$|\varepsilon_k|=\Big|\frac{u_k-r_k}{r_k}\Big|=\frac{tp_k-1+\eee^{-tp_k}}{tp_k}\leq \frac{tp_k}{2},$$ we infer $\sum_{k\ge1}|\varepsilon_k|\leq t/2$.
According to the proof of Theorem \ref{thm:main}(d) given in Section \ref{sect:proof d} below, $s_n=n^\beta+n^{(\beta-1)/2}$ for large $n$. With this at hand we obtain, as $n\to\infty$,
$$
\eee^{-s_n}\sum_{k=1}^{n}|\varepsilon_k|/r_k\le n\eee^{-s_n}/2=o(n^{-\beta})
$$
and
\begin{multline*}
	\eee^{s_n}\sum_{k\ge n+1}|\varepsilon_k|r_k\le 2^{-1}t^2\eee^{s_n}\sum_{k\geq n+1}p_k^2= 2^{-1} (tc)^2\eee^{s_n}\sum_{k\ge n+1}\eee^{-2k^\beta}\\\le 2^{-1}(tc)^2\eee^{s_n}\int_n^\infty\eee^{-2x^\beta}\dd x=(tc)^2\beta^{-1}2^{-1/\beta-1}\eee^{s_n}\int_{2n^\beta}^\infty\eee^{-y}y^{1/\beta-1}\dd y
	\\\sim 2^{-2}(tc)^2\beta^{-1}\eee^{s_n}\eee^{-2n^\beta}n^{1-\beta}=o(n^{-\beta}).
\end{multline*}
Here, for the last equivalence we have used
$\int_x^\infty y^{a-1}\eee^{-y}\dd y\sim x^{a-1}\eee^{-x}$ as $x\to\infty$ for $a>0$.

Summarizing, the conditions of Theorem \ref{thm:equiv} are satisfied. Hence,
the asymptotic relations stated in Theorem \ref{thm:main}(d) hold true for ${\rm Ran}(\xi, \pi(t))$ with $p_k=c\eee^{-k^\beta}$, up to the multiplicative constant $A=\prod_{k\ge1}(1+\varepsilon_k)=\prod_{k\ge1}\frac{1-\eee^{-tp_k}}{tp_k}$.

\section{Proofs for Section \ref{sec:classif}}

\subsection{Auxiliary results}\label{sec:aux}
Since $s_n$ is only defined 
for $n \ge \lfloor\psi^\prime(0)\rfloor + 1$, we tacitly assume that $n$ is large enough whenever needed.
First, we prove several inequalities for a general sequence $(r_k)_{k\geq 1}$. Recall that we assume $r_k>0$ for all $k\in\mn$ throughout the paper.

Moreover, by our earlier convention, the sequence $(r_k)_{k\geq 1}$ is nonincreasing. 
Recalling that $\lim_{n\to\infty}s_n=\infty$, we fix large enough $n$ satisfying
$r_1\eee^{s_n}\ge1$. Since $s_n<\infty$ and $r_k$ tends to $0$ monotonically,
there exists positive integer $K_n$ such that $r_k\eee^{s_n}\ge1$ for all $k\le K_n$ and $r_k\eee^{s_n}<1$ for all $k\ge K_n+1$.
\begin{lemma}\label{lem:ineq}
	The following inequalities hold:
	\begin{enumerate}[(a)]
		\item $\displaystyle \psi^{\prime\prime}(s_n)\le 2\sum_{k\ge n+1}\frac{r_k\eee^{s_n}}{r_k\eee^{s_n}+1-r_k}$;
		
		\item $\displaystyle
		\psi^{\prime\prime}(s_n)\ge\frac{1}{3}\sum_{k\ge n+1}\frac{r_k\eee^{s_n}}{r_k\eee^{s_n}+1-r_k}$ for $n$ so large that $r_k\le 1/2$ for all $k\ge n+1$;
		
		\item $\displaystyle
		\sum_{k=1}^n\frac{1-r_k}{r_k\eee^{s_n}+1-r_k}\le
		\sum_{k=1}^nr_k^{-1}\eee^{-s_n}$;
		
		\item $\displaystyle\sum_{k\ge n+1}\frac{r_k\eee^{s_n}}{r_k\eee^{s_n}+1-r_k}\le\sum_{k\ge n+1}r_k\eee^{s_n}$.
	\end{enumerate}
\end{lemma}
\begin{proof}
	(a) follows from \eqref{eq:core} and
	$$
	\psi^{\prime\prime}(s_n)=\sum_{k\ge1}\frac{(1-r_k)r_k\eee^{s_n}}{(r_k\eee^{s_n}+1-r_k)^2}\le \sum_{k=1}^n\frac{1-r_k}{r_k\eee^{s_n}+1-r_k}+\sum_{k\ge n+1}\frac{r_k\eee^{s_n}}{r_k\eee^{s_n}+1-r_k}.$$
	
	\noindent (b) Fix $n\in\mn$ so large that $r_k\leq 1/2$ for all $k\geq n+1$. Assume first that $K_n\le n$. Then $r_k\eee^{s_n}<1$ for all $k\ge n+1$. As a consequence,
	\[
	\psi^{\prime\prime}(s_n)\ge\sum_{k\ge n+1}\frac{(1-r_k)r_k\eee^{s_n}}{(r_k\eee^{s_n}+1-r_k)^2}\ge \frac{1}{3}\sum_{k\ge n+1}\frac{r_k\eee^{s_n}}{r_k\eee^{s_n}+1-r_k}.
	\]
	Let now $K_n>n$. Then $r_k\eee^{s_n}\ge 1$ for all $k\le n$ and thereupon
	\[
	\psi^{\prime\prime}(s_n)\ge\sum_{k=1}^n\frac{(1-r_k)r_k\eee^{s_n}}{(r_k\eee^{s_n}+1-r_k)^2}\ge \frac{1}{2}\sum_{k=1}^n\frac{1-r_k}{r_k\eee^{s_n}+1-r_k}
	\]
	which completes the proof of (b) by \eqref{eq:core}.
	
	\noindent (c) follows from $$\sum_{k=1}^n\frac{1-r_k}{r_k\eee^{s_n}+1-r_k}=\sum_{k=1}^n\frac{(r_k^{-1}-1)\eee^{-s_n}}{1+(r_k^{-1}-1)\eee^{-s_n}}\le \sum_{k=1}^nr_k^{-1}\eee^{-s_n}.$$
	\noindent (d) is trivial.
\end{proof}

Recall the notation $\eta_k=\1_{A_k}$ for $k\in\mn$. To prove Theorem \ref{thm:caseB}, we need the following lemma.
\begin{lemma}\label{lem:module}
	For $k\in\mn$, $s>0$ and $|u|\le \pi$
	\begin{equation}\label{eq:module}
		|\me^{(s)}[\eee^{\ii u\eta_k}]|\le
		\exp\Big(-\frac{2u^2}{\pi^2}\frac{(1-r_k)r_k\eee^s}{
			(r_k\eee^s+1-r_k)^2}\Big)=\exp\Big(-\frac{2u^2}{\pi^2}\psi_k^{\prime\prime}(s)\Big).
	\end{equation}
\end{lemma}
\begin{proof}
	We thank the referee for suggesting the following short proof.
	
	Recall that under $\mmp^{(s)}$, the variable $\eta_k$ has the Bernoulli distribution with the success probability $r_k\eee^s/(r_k\eee^s+1-r_k)$. 
	More generally, let $\eta = {\rm Bern}(p)$, where $p\in(0,1)$, and denote its variance by $\sigma^2:=\var[\eta]=p(1-p)$. Then invoking $(1-\cos u)/u^2\ge (1-\cos \pi)/\pi^2=2/\pi^2$ for $|u|\le\pi$, we obtain
	\begin{multline*}
		|\me[\eee^{\ii u\eta}]|^2=|1-p+p\eee^{\ii u}|^2=(1-p)^2+p^2+2p(1-p)\cos u=1-2\sigma^2(1-\cos u)\\
		\le \eee^{-2\sigma^2(1-\cos u)}\le \eee^{-\frac{4}{\pi^2}\sigma^2u^2}.
	\end{multline*}
	Applying this bound to $\eta_k$ under $\mmp^{(s)}$, we arrive at \eqref{eq:module}.
\end{proof}

To prove Theorem \ref{thm:caseC}, we need the following inequality.
\begin{lemma}\label{lem:log_ineq}
	For any $z\in(-\pi, \pi)$ and $w\in[0,1]$,
	$$
	|\log(1+(\eee^{\ii z}-1)w)|\le 2w|\tan (z/2)|,
	$$
	where $\log$ denotes the principal branch of the logarithm.
\end{lemma}
\begin{proof}
	Let $u:=(\eee^{\ii z}-1)w$. First, note that
	$$
	|u|^2=2w^2(1-\cos z)=4w^2\sin^2(z/2).
	$$
	Furthermore, for any $x\in[0,1]$, we have
	$$
	|1+ux|^2=1-2xw(1-xw)(1-\cos z)=1-4xw(1-xw)\sin^2(z/2)\ge \cos^2(z/2),
	$$
	since $p(1-p)\le 1/4$ for all $p\in[0,1]$. With this at hand, we arrive at the desired estimate
	$$
	|\log(1+u)|=\Big|\int_0^1\frac{u}{1+ux}\dd x\Big|\le \int_0^1\frac{|u|}{|1+ux|}\dd x\le \frac{2w|\sin(z/2)|}{|\cos(z/2)|}.
	$$
	
	Observe that the segment $\{1+ux:~x\in[0,1]\}$ does not intersect $(-\infty,0]$, so that the principal logarithm is analytic along this segment.
\end{proof}

\subsection{Proofs of Theorems \ref{thm:caseB}, \ref{thm:caseA} and \ref{thm:caseC}}

Unless stated otherwise, all the limits below are assumed to be taken as $n\to\infty$.

\begin{bluebar}
	\begin{proof}[Proof of Theorem \ref{thm:caseB}]
		
		\noindent (a)$\Rightarrow $(b) follows from Lemma \ref{lem:ineq}(b).
		
		\noindent (b)$\Rightarrow $(a) follows from Lemma \ref{lem:ineq}(a).
		
		\noindent (a)$\Rightarrow $(c) follows from Lemma \ref{lem:ineq}(c,d).
		
		\noindent (c)$\Rightarrow $(a).
		Fix $n\in\mn$ satisfying $n\geq k_1:=\min\{k\in\mn: r_k<1\}$. We use $K_n$ as defined in Section \ref{sec:aux}. Assume first that $K_n\le n$. Then $r_k\eee^{s_n}<1$ for  all $k\geq n+1$. As a consequence,
		\[
		\sum_{k\ge n+1}\frac{r_k\eee^{s_n}}{r_k\eee^{s_n}+1-r_k}\ge\frac{1}{2}\sum_{k\ge n+1}r_k\eee^{s_n}.
		\]
		Let now $K_n>n$. Then $r_k\eee^{s_n}\ge 1$ for all $k\le n$ and thereupon
		$$
			\sum_{k=1}^n\frac{1-r_k}{r_k\eee^{s_n}+1-r_k}=\sum_{k=1}^n\frac{(r_k^{-1}-1)\eee^{-s_n}}{1+(r_k^{-1}-1)\eee^{-s_n}}\ge\frac{1}{2}\sum_{k=1}^n(r_k^{-1}-1)\eee^{-s_n}\ge\frac{1-r_{k_1}}{2}\sum_{k=k_1}^nr_k^{-1}\eee^{-s_n}.
		$$
		
		Combining both inequalities together we obtain
		$$
			\sum_{k=1}^n\frac{1-r_k}{r_k\eee^{s_n}+1-r_k}\!=\!\sum_{k\ge n+1}\frac{r_k\eee^{s_n}}{r_k\eee^{s_n}+1-r_k}\!\ge\!\frac{1-r_{k_1}}{2
			}\min\Big\{\sum_{k\ge n+1}r_k\eee^{s_n},\sum_{k=k_1}^nr_k^{-1}\eee^{-s_n}\Big\}\!\to\!\infty.
		$$
		
		\noindent\textbf{Proof of $\displaystyle\mmp\{Y\ge n\}\sim\mmp\{Y=n\}\sim\frac{\exp(\psi(s_n)-s_n\psi^\prime(s_n))}{(2\pi \psi^{\prime\prime}(s_n))^{1/2}}$.}
		In view of~\eqref{eq:Y>x2}, the second asymptotic relation follows if we can show that
		\begin{equation}\label{eq:Ps_n}
			\mmp^{(s_n)}\{Y=n\}~\sim~\frac{1}{(2\pi \psi^{\prime\prime}(s_n))^{1/2}
			}.
		\end{equation}
		By the Stolz–Ces\`{a}ro theorem, $\mmp\{Y=n\}\sim \mmp\{Y\ge n\}$ follows if we can prove that
		\begin{equation}\label{eq:inter1}
			\lim_{n\to\infty}\frac{\mmp\{Y=n+1\}}{\mmp\{Y=n\}}=0.
		\end{equation}
		Fix any $j\in\mn_0$. A counterpart of \eqref{eq:Y>x1}, with $n+j$ replacing $n$, reads
		\begin{equation}\label{eq:Y=n+j}
			\mmp\{Y=n+j\}=\eee^{-js_n}\me [\eee^{s_n(Y-n)}]\mmp^{(s_n)}\{Y=n+j\}.
		\end{equation}
		Since $\lim_{n\to\infty}s_n=+\infty$, both \eqref{eq:Ps_n} and \eqref{eq:inter1} are secured by
		\begin{equation}\label{eq:Ps_n2}
			\mmp^{(s_n)}\{Y=n+j\}~\sim~\frac{1}{(2\pi \psi^{\prime\prime}(s_n))^{1/2}},
		\end{equation}
		with $j=0$ and $j=1$.
		
		\noindent {\sc Proof of \eqref{eq:Ps_n2}.} For convenience, we write $s$ in place of $s_n$.
		In view of
		$$
		Y_0=Y_0(s)=Y-\me^{(s)}[Y]=Y-n=\sum_{k\ge 1}(\1_{A_k}-\mmp^{(s)}(A_k)),
		$$
		the variable $Y_0$ under $\mmp^{(s)}$ is an infinite sum of independent bounded random variables with zero means. Hence, the relation
		\begin{equation}\label{eq:clt}
			\lim_{n\to\infty} \me^{(s_n)}
			[\exp(\ii u (\psi^{\prime\prime}(s_n)
			)^{-1/2}(Y_0(s_n)-j))]=\eee^{- u^2/2},\quad u\in\mr
		\end{equation}
		is an immediate consequence of the Lindeberg-Feller theorem and $\lim_{n\to\infty}\psi^{\prime\prime}(s_n)=\infty$, see, for instance, Theorem~3.4.10 on p.148 in \cite{Durrett:2019}.
		
		For any $n\geq \lfloor \psi^\prime(0)\rfloor+1
		$,
		$$
			\mmp^{(s)}\{Y=n+j\}=\frac{1}{2\pi}\int_{|z|\leq \pi}\eee^{-\ii z(n+j)} \me^{(s)}[\eee^{\ii zY}]\dd z=\frac{1}{2\pi}\int_{|z|\leq \pi}\me^{(s)}[\eee^{\ii z(Y_0-j)}]\dd z.
		$$
		Changing the variable $u=z (\psi^{\prime\prime}(s))^{1/2}$ yields
		\begin{multline}\label{eq:dom_int}
			\mmp^{(s)}\{Y=n+j\}=\frac{1}{2\pi (\psi^{\prime\prime}(s))^{1/2}}\int_{|u|\leq \pi (\psi^{\prime\prime}(s))^{1/2}} \me^{(s)}[\exp(\ii u(\psi^{\prime\prime}(s))^{-1/2}(Y_0-j))]\dd u\\=
			\frac{1}{2\pi (\psi^{\prime\prime}(s))^{1/2}}\int_\mr \me^{(s)}[\exp(\ii u(\psi^{\prime\prime}(s))^{-1/2}(Y_0-j))]\1_{|u|\leq \pi (\psi^{\prime\prime}(s))^{1/2}}\dd u
			.
		\end{multline}
		To derive \eqref{eq:Ps_n2} from~\eqref{eq:dom_int}, we apply the dominated convergence theorem.
		To this end, first note that by Lemma~\ref{lem:module}, for $u$ satisfying $|u|\le \pi (\psi^{\prime\prime}(s))^{1/2}$,
		\begin{multline*}
			|\me^{(s)}[\exp(\ii u(\psi^{\prime\prime}(s))^{-1/2}(Y_0-j))]|=\prod_{k\geq 1}|\me^{(s)}[\exp(\ii u(\psi^{\prime\prime}(s) )^{-1/2}\eta_k)]|\\\le 
			\exp\Big(-\frac{2u^2}{\pi^2
				\psi^{\prime\prime}(s)}\sum_{k\ge 1} \psi_k^{\prime\prime}(s)\Big)=
			\exp\Big(-\frac{2u^2}{\pi^2
			}\Big).
		\end{multline*}
		Therefore, $\exp(-2u^2/\pi^2)$ is a sought integrable majorant for the last integral in \eqref{eq:dom_int}. The pointwise limit of the integrand is described in \eqref{eq:clt}. To arrive at \eqref{eq:Ps_n2}, it remains to note that $\int_\mr \eee^{-u^2/2}\dd u=\sqrt{2\pi}$.
	\end{proof}
	
	\noindent {\sc Alternative proof of \eqref{eq:Ps_n2}.} For all sufficiently large $n$, put $a_n(k):=\mmp^{(s_n)}\{Y=k\}$ for $k\in\mn_0$. We first show that the sequence $(a_n(k))_{k\ge0}$ is \textit{log-concave}.
	
	If $r_k=1$ for some $k\in\mn$, then the corresponding indicator $\eta_k$ is equal to $1$. Since $\sum_{k\ge1}r_k<\infty$, there are only finitely many indices $k$ for which $r_k=1$. Removing the corresponding
	indicators merely shifts $Y$ by a constant, and such a shift does not affect log-concavity. Therefore, we may assume, without loss of generality, that $r_k<1$ for all $k\in\mn$. In particular, $$\mmp\{Y=0\}=\prod_{k\ge1}(1-r_k)>0.$$
	
	The generating function of the normalized sequence $\Big(\frac{\mmp\{Y=j\}}{\mmp\{Y=0\}}\Big)_{j\ge0}$ is given by
	$$\sum_{j\geq 0}\frac{\mmp\{Y=j\}}{\mmp\{Y=0\}}z^j=\frac{\me[z^Y]}{\mmp\{Y=0\}}=\prod_{k\geq 1}\frac{1-r_k+r_kz}{1-r_k}=\prod_{k\geq 1}\Big(1+\frac{r_k}{1-r_k}z\Big).$$
	Hence, by Theorem~\ref{thm:totally_positive}, the normalized sequence is totally positive and therefore log-concave. Consequently, $(\mmp\{Y=j\})_{j\ge0}$ is also log-concave.
	
	Invoking \eqref{eq:choice_of_measure}, we obtain, for any $s\in\mr$ and $k\in\mn$,
	\begin{multline*}
		(\mmp^{(s)}\{Y=k\})^2=\frac{\eee^{2sk}(\mmp\{Y=k\})^2}{(\me [\eee^{sY}])^2}
		\ge \frac{\eee^{s(k+1)}\mmp\{Y=k+1\}}{\me [\eee^{sY}]}\frac{\eee^{s(k-1)}\mmp\{Y=k-1\}}{\me[\eee^{sY}]}\\=\mmp^{(s)}\{Y=k+1\}\mmp^{(s)}\{Y=k-1\}.
	\end{multline*}
	Thus,
	$(a_n(k))_{k\ge0}$ is log-concave.
	
	As mentioned above, 
	under $\mmp^{(s_n)}$, $Y$ satisfies the Lindeberg-Feller central limit theorem, with $\me^{(s_n)}[Y]=n$ and $\var^{(s_n)}[Y]=\psi^{\prime\prime}(s_n)$, see~\eqref{eq:var_Y0}. Therefore, combining these two facts and applying Lemma 2 in \cite{Bender:1973}, we conclude that the array $(a_n(k))_{n,k}$
	satisfies a local central limit theorem. 
	This proves \eqref{eq:Ps_n2} and also shows that it holds for
	every fixed $j\in\mathbb{Z}$, rather than only for $j=0$ and $j=1$.
\end{bluebar}

\begin{greenbar}
	\begin{proof}[Proof of Theorem \ref{thm:caseA}]
		(a)$\Rightarrow $(b) follows from Lemma \ref{lem:ineq}(a).
		
		\noindent (b)$\Rightarrow $(a)
		follows from Lemma \ref{lem:ineq}(b).
		
		\sloppy\noindent (a)$\Rightarrow $(c) and (d).
		We use $K_n$ as defined in Section \ref{sec:aux}. First, we prove that part (a) entails $K_n=n$ as long as $n$ is so large that $\sum_{k=1}^n\frac{1-r_k}{r_k\eee^{s_n}+1-r_k}=\sum_{k\ge n+1}\frac{r_k\eee^{s_n}}{r_k\eee^{s_n}+1-r_k}<1/3$ and $r_n\le 1/2$.
		Indeed, assuming that $r_{n+1}\eee^{s_n}\ge 1$ we arrive at a contradiction
		$$
			\frac{1}{3}\!>\!\sum_{k\ge n+1}\frac{r_k\eee^{s_n}}{r_k\eee^{s_n}+1-r_k}\!\ge\!\frac{r_{n+1}\eee^{s_n}}{r_{n+1}\eee^{s_n}+1-r_{n+1}}=1-\frac{1-r_{n+1}}{r_{n+1}\eee^{s_n}+1-r_{n+1}}\ge \frac{1}{2-r_{n+1}}\ge\frac{1}{2}.
		$$
		Hence, $K_n\le n$. Assuming now that $r_{n}\eee^{s_n}<1$ we arrive at another contradiction
		\[
		\frac{1}{3}>\sum_{k=1}^n\frac{1-r_k}{r_k\eee^{s_n}+1-r_k}\ge\frac{1-r_n}{r_n\eee^{s_n}+1-r_n}>\frac{1-r_n}{2-r_n}\ge \frac{1}{3}
		\]
		since
		$r_n\le 1/2$. Thus, $K_n=n$ for large $n$.
		
		With this at hand, (a)$\Rightarrow$(c) follows from
		$$
			0\!\leftarrow\!\sum_{k=1}^n\frac{1-r_k}{r_k\eee^{s_n}+1-r_k}\!=\!\sum_{k=1}^n\frac{(r_k^{-1}-1)\eee^{-s_n}}{1+(r_k^{-1}-1)\eee^{-s_n}}\!\ge\! \frac{1}{2}\sum_{k=1}^n(r_k^{-1}-1)\eee^{-s_n}\!\ge\! \frac{1-r_{k_1}}{2}\sum_{k=k_1}^nr_k^{-1}\eee^{-s_n}.
		$$
		Here, we have used for the first inequality that $K_n=n$ implies $(r_k^{-1}-1)\eee^{-s_n}\le 1$ for $k\leq n$. Recall for the second inequality that $k_1$ denotes the first index $k$ such that $r_k<1$.
		
		The implication (a)$\Rightarrow$(d) is a consequence of
		\[
		0\leftarrow\sum_{k\ge n+1}\frac{r_k\eee^{s_n}}{r_k\eee^{s_n}+1-r_k}\geq\sum_{k\ge n+1}\frac{r_k\eee^{s_n}}{2-r_k}\ge \frac{1}{2}\sum_{k\ge n+1}r_k\eee^{s_n}\ge 0.
		\]
		The first inequality is secured by $K_n=n$ for large enough $n$ which entails $r_k\eee^{s_n}<1$ for all $k\ge n+1$.
		
		\noindent (c)$\Rightarrow $(a) follows from Lemma \ref{lem:ineq}(c).
		
		\noindent (d)$\Rightarrow $(a) follows from Lemma \ref{lem:ineq}(d).
		
		\noindent\textbf{Proof of $\mmp\{Y=n\}\sim \exp (\psi(s_n)-s_n\psi^\prime(s_n))$.} In view of \eqref{eq:alt}, it suffices to show that $\lim_{n\to\infty} \mmp^{(s_{n})}\{Y=n\}=1$.
		
		Recall that $\me^{(s_n)}[Y]=n$ and $\var^{(s_n)}[Y] =\psi^{\prime\prime}(s_n)$. Since $Y$ is integer-valued, by Chebyshev's inequality
		$$
		\mmp^{(s_{n})}\{Y\neq n\}=\mmp^{(s_{n})}\{|Y-n|\ge 1\}=\mmp^{(s_{n})}\{|Y-\me^{(s_n)}[Y]|\ge 1\}\le \var^{(s_n)}[Y].
		$$
		Assumption (b) ($\psi^{\prime\prime}(s_n)\to0$) implies $\mmp^{(s_{n})}\{Y\neq n\}\to0$, and consequently, $\mmp^{(s_{n})}\{Y=n\}\to1$.

		\noindent\textbf{Proof of $\mmp\{Y=n\}\sim \mmp\{Y\ge n\}$.} By the Stolz–Ces\`{a}ro theorem, $\mmp\{Y\ge n\}\sim\mmp\{Y=n\}$ follows if we can prove that
		$\lim_{n\to\infty}\frac{\mmp\{Y=n+1\}}{\mmp\{Y=n\}}=0$.
		An application of \eqref{eq:Y=n+j} yields
		$$
		\frac{\mmp\{Y=n+1\}}{\mmp\{Y=n\}}=\eee^{-s_n}\frac{\mmp^{(s_n)}\{Y=n+1\}}{\mmp^{(s_n)}\{Y=n\}}\to0,\quad n\to\infty,
		$$
		since $s_n\to+\infty$, and $\mmp^{(s_n)}\{Y=n\}\to1$ whereas $\mmp^{(s_n)}\{Y=n+1\}\le1$.
	\end{proof}
\end{greenbar}

\begin{orangebar}
	\begin{proof}[Proof of Theorem \ref{thm:caseC}]
		
		\noindent\textbf{(i)-(iv)} follow from the dominated convergence theorem. To find termwise limits and appropriate summable majorants, we use\\
		for (i): $
		\sum_{k=1}^nr_k^{-1}\eee^{-s_n}=\sum_{k=0}^{n-1}r_{n-k}^{-1}\eee^{-s_n}
		$
		and the assumption (b)
		;\\
		for (ii): $
		\sum_{k\ge n+1}r_k\eee^{s_n}=\sum_{k\ge1}r_{n+k}\eee^{s_n}
		$
		and the assumption (a)
		;\\
		for (iii): split $\psi^{\prime\prime}(s_n)$ into two sums $\sum_{k=1}^n$ and $\sum_{k\ge n+1}$ and treat them separately. The first sum is equal to
		$$
		\sum_{k=1}^{n}\frac{r_k^{-1}\eee^{-s_n}-\eee^{-s_n}}{(r_k^{-1}\eee^{-s_n}-\eee^{-s_n}+1)^2}=\sum_{k=0}^{n-1}\frac{r_{n-k}^{-1}\eee^{-s_n}-\eee^{-s_n}}{(r_{n-k}^{-1}\eee^{-s_n}-\eee^{-s_n}+1)^2}.
		$$
		According to the assumption (b), the termwise limit is $q_k/(1+q_k)^2$. Keeping in mind that all $r_k\le 1$, to obtain a summable majorant, we consider the estimate
		$$
		\frac{r_{n-k}^{-1}\eee^{-s_n}-\eee^{-s_n}}{(r_{n-k}^{-1}\eee^{-s_n}-\eee^{-s_n}+1)^2}\le r_{n-k}^{-1}\eee^{-s_n}-\eee^{-s_n}\le r_{n-k}^{-1}\eee^{-s_n}
		$$
		and invoke (b) once again.
		
		Similarly, for the second sum
		$$
		\sum_{k\ge n+1}\frac{(1-r_k)r_k\eee^{s_n}}{(1-r_k+r_k\eee^{s_n})^2}=\sum_{k\ge 1}\frac{(1-r_{n+k})r_{n+k}\eee^{s_n}}{(1-r_{n+k}+r_{n+k}\eee^{s_n})^2}
		$$
		the termwise limit is $p_k/(1+p_k)^2$ by the assumption (a). The summable majorant is obtained from the estimate
		$$
		\frac{(1-r_{n+k})r_{n+k}\eee^{s_n}}{(1-r_{n+k}+r_{n+k}\eee^{s_n})^2}\le \frac{r_{n+k}\eee^{s_n}}{1-r_{n+k}+r_{n+k}\eee^{s_n}}\le r_{n+k}\eee^{s_n}
		$$
		together with another application of (a).\\
		for (iv): proceed as in (iii) using the bounds (after changing the summation indices)
		$$
		\frac{1-r_{n-k}}{1-r_{n-k}+r_{n-k}\eee^{s_n}}=\frac{r_{n-k}^{-1}\eee^{-s_n}-\eee^{-s_n}}{r_{n-k}^{-1}\eee^{-s_n}-\eee^{-s_n}+1}\le r_{n-k}^{-1}\eee^{-s_n}
		$$
		and
		$$
		\frac{r_{n+k}\eee^{s_n}}{1-r_{n+k}+r_{n+k}\eee^{s_n}}\le r_{n+k}\eee^{s_n}.
		$$
		\color{black}

		\noindent\textbf{Proof of $\mmp\{Y=n\}\sim c_0\exp (\psi(s_n)-s_n\psi^\prime(s_n))$.} In view of \eqref{eq:alt}, it suffices to show that $\lim_{n\to\infty} \mmp^{(s_{n})}\{Y=n\}=c_0$. To this end,
		we shall prove that, for $z\in (-\pi, \pi)$, 
		$$\lim_{n\to\infty} \me^{(s_n)}[\eee^{\ii z(Y-n)}]=\me\big[\eee^{\ii z\sum_{k\ge0}\theta_k}\big].$$
		
		Invoking \eqref{eq:choice_of_measure}
		we obtain, for any $z\in(-\pi,\pi)$,
		$$
		\me^{(s_n)}[\eee^{\ii z(Y-n)}]=\frac{\me[\eee^{(s_n+\ii z)Y-\ii zn}]}{\me[\eee^{s_nY}]}=\eee^{-\ii zn}\prod_{k\ge1}\frac{1-r_k+r_k\eee^{s_n+\ii z}}{1-r_k+r_k\eee^{s_n}}.
		$$
		Therefore, with $\log$ denoting the principal branch of the logarithm,
		\begin{multline*}
			\me^{(s_n)}[\eee^{\ii z(Y-n)}]=\exp\Big(\sum_{k\ge1}\log\Big(\frac{1-r_k+r_k\eee^{s_n+\ii z}}{1-r_k+r_k\eee^{s_n}}\Big)-\ii zn\Big)\\=\exp\Big(\sum_{k\ge n+1}\ldots+\sum_{k=1}^n\ldots-\ii zn\Big)=\exp\Big(\sum_{k\ge n+1}\log\Big(1+\frac{(\eee^{\ii z}-1)r_k\eee^{s_n}}{1-r_k+r_k\eee^{s_n}}\Big)\\+\sum_{k=1}^n\log\Big(1+\frac{(\eee^{-\ii z}-1)(1-r_k)}{1-r_k+r_k\eee^{s_n}}\Big)\Big)=\exp\Big(\sum_{k\ge 1}\log\Big(1+\frac{(\eee^{\ii z}-1)r_{n+k}\eee^{s_n}}{1-r_{n+k}+r_{n+k}\eee^{s_n}}\Big)\\
			+\sum_{k=0}^{n-1}\log\Big(1+\frac{(\eee^{-\ii z}-1)(1-r_{n-k})}{1-r_{n-k}+r_{n-k}\eee^{s_n}}\Big)\Big)=:\exp(S_1(n,z)+S_2(n,z)).
		\end{multline*}
		By Lemma \ref{lem:log_ineq} and the assumption (a), the absolute values of the summands defining $S_1$ are dominated by a summable sequence:
		\begin{equation*}
			\Big|\log\Big(1+\frac{(\eee^{\ii z}-1)r_{n+k}\eee^{s_n}}{1-r_{n+k}+r_{n+k}\eee^{s_n}}\Big)\Big|\!\le\!\frac{2|\tan(z/2)|r_{n+k}\eee^{s_n}}{1-r_{n+k}+r_{n+k}\eee^{s_n}}\!\le\! 2|\tan(z/2)|r_{n+k}\eee^{s_n}.
		\end{equation*}
		Similarly, Lemma \ref{lem:log_ineq} and (b) entail that the absolute values of the summands defining $S_2$ are dominated by a summable sequence:
		\begin{multline*}
			\Big|\log\Big(1+\frac{(\eee^{-\ii z}-1)(1-r_{n-k})}{1-r_{n-k}+r_{n-k}\eee^{s_n}}\Big)\Big|\le \frac{2|\tan(z/2)|(1-r_{n-k})}{1-r_{n-k}+r_{n-k}\eee^{s_n}}\\\le2|\tan(z/2)|\,\frac{1/r_{n-k}-1}{1/r_{n-k}-1+\eee^{s_n}} \le2|\tan(z/2)|\,\frac{1/r_{n-k}}{\eee^{s_n}}.
		\end{multline*}
		Therefore, by the dominated convergence theorem, for $z\in(-\pi,\pi)$, as $n\to\infty$,
		$$
		S_1(n,z)=\sum_{k\ge 1}\log\Big(1+\frac{(\eee^{\ii z}-1)r_{n+k}\eee^{s_n}}{1-r_{n+k}+r_{n+k}\eee^{s_n}}\Big)\to\sum_{k\ge 1}\log\Big(\frac{1+\eee^{\ii z}p_k}{1+p_k}\Big)
		$$
		and
		$$
			S_2(n,z)=\sum_{k=0}^{n-1}\log\Big(1+\frac{(\eee^{-\ii z}-1)(r_{n-k}^{-1}\eee^{-s_n}-\eee^{-s_n})}{r_{n-k}^{-1}\eee^{-s_n}-\eee^{-s_n}+1}\Big)\to\sum_{k\ge0}\log\Big(\frac{1+\eee^{-\ii z}q_k}{1+q_k}\Big).
		$$
		Summarizing, for $z\in(-\pi,\pi)$, as $n\to\infty$,
		$$
		\me^{(s_n)}[\eee^{\ii z(Y-n)}]\to\frac{1+\eee^{-\ii z}q_0}{1+q_0} \prod_{k\ge1}\frac{1+p_kq_k+\eee^{\ii z}p_k+\eee^{-\ii z}q_k}{(1+p_k)(1+q_k)}=\me\big[\eee^{\ii z\sum_{k\ge0}\theta_k}\big].
		$$
		Noting that both $Y-n$ and $\sum_{k\ge0}\theta_k$ are integer-valued, the Fourier inversion formula and the dominated convergence theorem yield, for each $j\in\mathbb{Z}$,
		\begin{multline*}
			\mmp^{(s_n)}\{Y-n=j\}=\frac{1}{2\pi}\int_{-\pi}^\pi \eee^{-\ii zj}\me^{(s_n)}[\eee^{\ii z(Y-n)}]\dd z\\~\to~ \frac{1}{2\pi}\int_{-\pi}^\pi\eee^{-\ii zj}\me\Big[\eee^{\ii z\sum_{k\ge0}\theta_k}\Big]\dd z
			=\mmp\Big\{\sum_{k\geq 0}\theta_k=j\Big\},\quad n\to\infty.
		\end{multline*}
		In particular, 
		$\mmp^{(s_{n})}\{Y=n\}\to \mmp\{\sum_{k\ge0}\theta_k=0\}=c_0$ as $n\to\infty$.
		
		Finally, $$c_0\ge\mmp\{ \theta_k=0~\text{for all}~k\in\mn_0\}=\frac{1}{1+q_0}\prod_{m\geq 1}\frac{1+p_mq_m}{(1+p_m)(1+q_m)}>0.$$ Positivity of the infinite product is secured by $$\sum_{k\geq 0}\frac{q_k}{1+q_k}<\infty\quad\text{and}\quad  \sum_{k\geq 1}\frac{p_k}{1+p_k}<\infty,$$ see point (iv) of the theorem. Moreover, $c_0<1$. Indeed, by the assumption~(c), $q_0>0$, whence 
		$$\var\Big[\sum_{k\ge0}\theta_k\Big]=\sum_{k\ge0}\var[\theta_k]\geq \var[\theta_0]=\frac{q_0}{(1+q_0)^2}>0.$$ Thus, the distribution of $\sum_{k\ge0}\theta_k$ is not degenerate at $0$.
		
		\noindent\textbf{Proof of $\mmp\{Y=n\}\sim \mmp\{Y\ge n\}$} mimics the proof of this relation in Theorem \ref{thm:caseA}.
	\end{proof}
\end{orangebar}

\begin{rem}\label{rem:explanation}
	Here, we explain why the statement of Theorem \ref{thm:caseC} is not as clean as those of Theorems \ref{thm:caseB} and \ref{thm:caseA}.
	
	The condition $\lim_{n\to\infty}\psi^{\prime\prime}(s_n)=\infty$ appearing in Theorem \ref{thm:caseB}, without any further conditions, entails that the $\mmp^{(s_n)}$-distributions of $(\psi^{\prime\prime}(s_n))^{-1/2}(Y-n)$ satisfy a central limit theorem. The condition $\lim_{n\to\infty}\psi^{\prime\prime}(s_n)=0$ appearing in Theorem \ref{thm:caseA}, also without further conditions, ensures that the $\mmp^{(s_n)}$-distributions of $Y-n$ are asymptotically degenerate at $0$.
	
	\sloppy The remaining condition $\lim_{n\to\infty}\psi^{\prime\prime}(s_n)\in (0,\infty)$ alone (holding possibly along a subsequence) does not seem to secure a limit theorem  for the $\mmp^{(s_n)}$-distributions of $Y-n$. Sufficient conditions given in Theorem \ref{thm:caseC} ensure the convergence of the $\mmp^{(s_n)}$-distributions of the infinite sum of independent elements of the {\it non-infinitesimal\footnote{This means that for $\varepsilon\in (0,1)$, $\max_{k\geq 1}\mmp^{(s_n)}\{|\1_{A_k}-\mmp^{(s_n)}(A_k)|>\varepsilon\}$ does not tend to $0$ as $n\to\infty$.}} array $(\1_{A_k}-\mmp^{(s_n)}(A_k))_{k,n\geq 1}$ to the distribution of a random variable represented by a random series. We note in passing that, because of non-infinitesimality, we cannot exploit in the setting of Theorem~\ref{thm:caseC} sufficient conditions \`{a} la Gnedenko and Kolmogorov \cite{Gnedenko+Kolmogorov:1954} for distributional convergence to an infinitely divisible distribution.
\end{rem}

\section{Proof of Theorem \ref{thm:equiv}}

{\bf Step 1.} Proving that $s_n$, the solution to $\psi^\prime_r(s)=n$, is an approximate solution\footnote{As noted in Remark \ref{rem}, such an approximation is precise enough for the statements of Theorems~\ref{thm:caseB}, \ref{thm:caseA} and \ref{thm:caseC} to hold.} to $\psi^\prime_u(s)+o(1/s)=n$.

It suffices to show that $\psi^\prime_u(s_n)-\psi^\prime_r(s_n)=o(1/s_n)$ as $n\to\infty$. Consider
\begin{multline*}
	\psi^\prime_u(s_n)-\psi^\prime_r(s_n)=\sum_{k\ge 1}\frac{(u_k-r_k)\eee^{s_n}}{(u_k\eee^{s_n}+1-u_k)(r_k\eee^{s_n}+1-r_k)}\\=\sum_{k\ge 1}\frac{\varepsilon_kr_k\eee^{s_n}}{(u_k\eee^{s_n}+1-u_k)(r_k\eee^{s_n}+1-r_k)}=\sum_{k=1}^n\ldots+\sum_{k\ge n+1}\ldots
	=:S_1(n)+S_2(n).
\end{multline*}
There exists $k_0\in\mn$ such that $|\varepsilon_k|<1/2$ for $k\geq k_0$. By assumptions of the theorem, as $n\to\infty$,
\begin{align*}
	|S_1(n)|\le \sum_{k=1}^n\frac{|\varepsilon_k|}{u_k\eee^{s_n}+1-u_k}&\le \sum_{k=1}^{k_0-1}\frac{|\varepsilon_k|}{u_k\eee^{s_n}+1-u_k}+\sum_{k=k_0}^n\frac{|\varepsilon_k|}{r_k|1+\varepsilon_k|\eee^{s_n}}\\&=o(\eee^{-s_n})+2\eee^{-s_n}\sum_{k=k_0}^n |\varepsilon_k|/r_k
	=o(1/s_n)
\end{align*}
and
$$
|S_2(n)|\le \eee^{s_n}\sum_{k\ge n+1}|\varepsilon_k|r_k
=o(1/s_n).
$$

\noindent{\bf Step 2.} Showing that $(u_k)$ satisfies the assumptions of Theorem 1.i, $i=1,2,3$ if so does~$(r_k)$.

\noindent{\bf Case 1.} $(r_k)$ satisfies the assumptions of Theorem \ref{thm:caseB}.

For $s>0$, consider
\begin{align*}
	\psi_u^{\prime\prime}(s)-\psi_r^{\prime\prime}(s)&=\sum_{k\geq 1}\frac{(1-u_k)u_k\eee^s}{(u_k\eee^s+1-u_k)^2}-\sum_{k\geq 1}\frac{(1-r_k)r_k\eee^s}{(r_k\eee^s+1-r_k)^2}\\&=\sum_{k\ge 1}\frac{(u_k-r_k)\eee^{s}((1-u_k)(1-r_k)-u_kr_k\eee^{2s})}{(u_k\eee^s+1-u_k)^2(r_k\eee^s+1-r_k)^2}\\&=\sum_{k\ge 1}\frac{\varepsilon_kr_k\eee^{s}((1-u_k)(1-r_k)-u_kr_k\eee^{2s})}{(u_k\eee^s+1-u_k)^2(r_k\eee^s+1-r_k)^2}.
\end{align*}
Since
$$
\sum_{k\ge 1}\frac{|\varepsilon_k|r_k\eee^{s}(1-u_k)(1-r_k)}{(u_k\eee^s+1-u_k)^2(r_k\eee^s+1-r_k)^2}\le \sum_{k\ge 1}\frac{|\varepsilon_k|}{u_k\eee^s+1-u_k}\le \sum_{k\ge 1}|\varepsilon_k|<\infty
$$
and
$$
\sum_{k\ge 1}\frac{|\varepsilon_k|u_kr_k^2\eee^{3s}}{(u_k\eee^s+1-u_k)^2(r_k\eee^s+1-r_k)^2}\le\sum_{k\ge 1}\frac{|\varepsilon_k|}{u_k\eee^s+1-u_k}\le \sum_{k\ge 1}|\varepsilon_k|<\infty,
$$
we obtain $\psi_u^{\prime\prime}(s)-\psi_r^{\prime\prime}(s)=O(1)$ as $s\to\infty$ and thereupon
$\psi_u^{\prime\prime}(s_n)~\sim~\psi_r^{\prime\prime}(s_n)$ as $n\to\infty$.

Thus, $(u_k)$ satisfies the assumption of Theorem \ref{thm:caseB}, namely, $\lim_{n\to\infty}\psi_u^{\prime\prime}(s_n)=\infty$. As a consequence,
$$
	\mmp\{Y_u\ge n\}\sim\mmp\{Y_u= n\}\sim\frac{\exp(\psi_u(s_n)-s_nn)}{(2\pi\psi^{\prime\prime}_u(s_n))^{1/2}}\sim\exp(\psi_u(s_n)-\psi_r(s_n))\mmp\{Y_r\ge n\}.
$$

\noindent{\bf Case 2.} $(r_k)$ satisfies the assumptions of Theorem \ref{thm:caseA}.

Then
$$
\eee^{s_n}\sum_{k\ge n+1}u_k=\eee^{s_n}\sum_{k\ge n+1}r_k(1+\varepsilon_k)\le 2 \eee^{s_n}\sum_{k\ge n+1}r_k\to 0, \quad n\to\infty.
$$
Thus, $(u_k)$ satisfies the assumption (d) of Theorem \ref{thm:caseA}, whence
$$
	\mmp\{Y_u\ge n\}\sim\mmp\{Y_u= n\}\sim\exp(\psi_u(s_n)-s_nn)\sim\exp(\psi_u(s_n)-\psi_r(s_n))\mmp\{Y_r\ge n\}.
$$

\noindent{\bf Case 3.} $(r_k)$ satisfies the assumptions of Theorem \ref{thm:caseC}.
\begin{enumerate}[(a)]
	\item Fix $k\in\mn$. Then $u_{n+k}\eee^{s_n}=r_{n+k}(1+\varepsilon_{n+k})\eee^{s_n}\sim r_{n+k}\eee^{s_n}\to p_k$ and $\sup_{n\ge N_1}u_{n+k}\eee^{s_n}\le 2(p_k+\alpha_k)$.
	
	\item Fix $k\in\mn_0$. Then $u_{n-k}^{-1}\eee^{-s_n}=r_{n-k}^{-1}(1+\varepsilon_{n-k})^{-1}\eee^{-s_n} 
	\sim r_{n-k}^{-1}\eee^{-s_n}\to q_k$ and since $|\varepsilon_n|<1/2$ for large enough $n$, $\sup_{n\ge N_2}u_{n-k}^{-1}\eee^{-s_n}\le 2(q_k+\beta_k)$.
	
	\item Since $(p_k)$ and $(q_k)$ are the same for the success probabilities $(r_k)$ and $(u_k)$, the 
	inequality $q_0>0$ holds
	automatically.
\end{enumerate}
Thus, $(u_k)$ satisfies the assumptions of Theorem \ref{thm:caseC} with the same $(p_k)$ and $(q_k)$ as $(r_k)$ does. As a consequence,
$$
	\mmp\{Y_u\ge n\}\sim \mmp\{Y_u= n\}\sim c_0\exp(\psi_u(s_n)-s_nn)\sim\exp(\psi_u(s_n)-\psi_r(s_n))\mmp\{Y_r\ge n\}.
$$

\noindent{\bf Step 3.} End of the proof.

It remains to note that, by the dominated convergence theorem, as $s\to\infty$,
\begin{multline*}
	\psi_u(s)-\psi_r(s)=\sum_{k\ge 1}\log\Big(\frac{1-u_k+u_k\eee^s}{1-r_k+r_k\eee^s}\Big)=\sum_{k\ge 1}\log\Big(1+\frac{r_k\varepsilon_k(\eee^s-1)}{1-r_k+r_k\eee^s}\Big)\\\to\sum_{k\ge 1}\log(1+\varepsilon_k)=\log A.
\end{multline*}
Indeed, $r_k(\eee^s-1)/(1-r_k+r_k\eee^s)\in [0,1)$ for $s\geq 0$ and $k\in\mn$. Furthermore, there exists $k_1\in\mn$ such that $|\varepsilon_k|\leq 1/2$ for all $k\geq k_1$. Recall that $|\log(1+x)|\leq 2|x|$ whenever $|x|\leq 1/2$. For $s>0$ and $k\geq k_1$, put $x=r_k\varepsilon_k(\eee^s-1)/(1-r_k+r_k\eee^s)$. Then $|x|\leq 1/2$ and therefore
$$
\Big|\log\Big(1+\frac{r_k\varepsilon_k(\eee^s-1)}{1-r_k+r_k\eee^s}\Big)\Big|\leq \frac{2r_k|\varepsilon_k|(\eee^s-1)}{1-r_k+r_k\eee^s}\le 2|\varepsilon_k|.
$$
Since $\sum_{k\ge1}|\varepsilon_k|<\infty$ by assumption, the sequence $(2|\varepsilon_k|)_{k\ge k_1}$ is a summable majorant. The finitely many terms with $k\in\{1,2,\ldots, k_1-1\}$ can be treated separately.

\section{Proof of Theorem \ref{thm:main}}
\subsection{Auxiliary results in the case
	$r_k=ck^{-\beta}$}

We start by calculating the value of an integral, which repeatedly appears in our proofs.
\begin{lemma}\label{lem:integral}
	For $\theta>0$, $\alpha>-1$, $\beta>0$ and $\gamma>(\alpha+1)/\beta$,
	$$
	\int_0^\infty\frac{x^\alpha}{(1+(\theta x)^\beta)^\gamma}\dd x=\frac{1}{\beta \theta^{\alpha+1}}{\rm B}\Big(\frac{\alpha+1}{\beta},\gamma-\frac{\alpha+1}{\beta}\Big),
	$$
	where ${\rm B}$ denotes the Euler beta function.
\end{lemma}
\begin{proof}
	Denote the integral by  $I$. Changing the variable $z=(\theta x)^\beta$ and then $y=z/(z+1)$,
	we arrive at
	\begin{multline*}
		I=
		\frac{1}{\beta \theta^{\alpha+1}}\int_0^{\infty}\frac{z^{(\alpha+1)/\beta-1}}{(1+z)^\gamma}\dd z=\frac{1}{\beta \theta^{\alpha+1}}\int_0^1 y^{(\alpha+1)/\beta-1}(1-y)^{\gamma-(\alpha+1)/\beta-1}{\rm d}y\\=\frac{1}{\beta \theta^{\alpha+1}}{\rm B}\Big(\frac{\alpha+1}{\beta},\gamma-\frac{\alpha+1}{\beta}\Big).
	\end{multline*}
\end{proof}

We shall need versions of the Euler–Maclaurin summation formula, which can be derived from formula (9.67) on p.~455 in \cite{Graham+Knuth+Patashnik:1990}. Let $m\in\mn_0$, $n\in \mn\cup \{\infty\}$ and $f:[m,n]\to\mr$ be a twice continuously differentiable function. Then $$\sum_{j=m}^n f(j)=\int_m^n f(x){\rm d}x+\frac{f(m)+f(n)}{2}+\frac{f^\prime(n)-f^\prime(m)}{12}-\int_m^n f^{\prime\prime}(x)\frac{P_2(x)}{2}{\rm d}x.$$ Here, $P_2$ is the periodic Bernoulli polynomial defined by $P_2(x):=\{x\}^2-\{x\}+1/6$, $\{x\}$ is the fractional part of $x$, and $f(\infty):=\lim_{x\to\infty}f(x)$ and $f^\prime(\infty):=\lim_{x\to\infty}f^\prime(x)$. If $n=\infty$, we additionally assume that the integrals involved converge. Since $|P_2(x)|\le 1/6$ for all $x\in\mr$ we obtain a formula to be used below
\begin{equation}\label{eq:EM}
	\sum_{j=m}^n f(j)=\int_m^n f(x){\rm d}x+\frac{f(m)+f(n)}{2}+\frac{f^\prime(n)-f^\prime(m)}{12}+R_{m,n},
\end{equation}
where $|R_{m,n}|\leq (1/12)\int_m^n |f^{\prime\prime}(x)|{\rm d}x$. Analogously, if $f:[m,n]\to\mr$ is a continuously differentiable function, then
\begin{equation}\label{eq:EM2}
	\sum_{j=m}^n f(j)=\int_m^n f(x)\dd x+\frac{f(m)+f(n)}{2}+Q_{m,n},
\end{equation}
where $|Q_{m,n}|\leq (1/2)\int_m^n |f^\prime(x)|{\rm d}x$.

Lemmas \ref{lem:(a)series1}, \ref{lem:(a)series2} and \ref{lem:(a)series3} are designed to understand the asymptotic behavior of certain functional series parameterized by $a$ as $a\to 0+$.

\begin{lemma}\label{lem:(a)series1}
	For $\beta>1$,
	$$
	\sum_{k\ge 1}\frac{1}{1+(ak)^\beta}=\frac{\pi}{\beta a\sin(\pi/\beta)}-\frac{1}{2}+O(a),\quad a\to0+.
	$$
\end{lemma}
\begin{proof}
	For each $a>0$, put $f_a(x):=\frac{1}{1+(ax)^\beta}$ for $x\geq 0$. Since $f_a$ is twice continuously differentiable on $[1,\infty)$ with $f_a(\infty)=f_a^\prime(\infty)=0$, an application of formula~\eqref{eq:EM} yields
	\begin{multline*}
		\sum_{k\ge 1}f_a(k)=\int_1^\infty f_a(x)\dd x+\frac{1}{2(1+a^\beta)}+\frac{\beta a^\beta}{12(1+a^\beta)^2}+R_{1,\infty}(a)\\=\int_0^\infty f_a(x)\dd x-\frac{1}{2}+R_{1,\infty}(a)+O(a^\beta),
	\end{multline*}
	where
	$|R_{1,\infty}(a)| \leq (1/12)\int_1^\infty |f^{\prime\prime}_a(x)|{\rm d}x$.
	By Lemma \ref{lem:integral},
	$$
	\int_0^\infty f_a(x)\dd x=\frac{1}{\beta a}{\rm B}\Big(\frac{1}{\beta},1-\frac{1}{\beta}\Big)=\frac{\pi}{\beta a\sin(\pi/\beta)}.
	$$
	In view of
	$$
	f^{\prime\prime}_a(x)=\frac{2\beta^2a^{2\beta}x^{2\beta-2}}{(1+(ax)^\beta)^3}-\frac{\beta(\beta-1)a^\beta x^{\beta-2}}{(1+(ax)^\beta)^2},
	$$
	we conclude that
	$$
	\int_1^\infty
	|f^{\prime\prime}_a(x)|
	\dd x\leq \int_0^\infty\frac{2\beta^2a^{2\beta}x^{2\beta-2}}{(1+(ax)^\beta)^3}\dd x+ \int_0^\infty\frac{\beta(\beta-1)a^\beta x^{\beta-2}}{(1+(ax)^\beta)^2}\dd x=:I_1(a)+I_2(a).
	$$
	Invoking Lemma \ref{lem:integral} once again we obtain
	$$
	I_1(a)=2\beta a
	{\rm B}\Big(2-\frac{1}{\beta}, 1+\frac{1}{\beta}\Big)=O(a),\quad a\to0+
	$$
	and
	$$
	I_2(a)=(\beta-1)a
	{\rm B}\Big(1-\frac{1}{\beta},1+\frac{1}{\beta}\Big)=O(a),\quad a\to0+.
	$$
	Combining all fragments together completes the proof.
\end{proof}

\begin{lemma}\label{lem:(a)series2}
	For $\beta>1$,
	$$
	\sum_{k\ge 1}\frac{(ak)^\beta}{(1+(ak)^\beta)^2}=\frac{\pi}{\beta^2 a\sin(\pi/\beta)}+O(1),\quad a\to0+.
	$$
\end{lemma}
\begin{proof}
	For each $a>0$, put
	$f_a(x):=\frac{(ax)^\beta}{(1+(ax)^\beta)^2}$ for $x\geq 0$. Since $f_a$ is continuously differentiable on $[0,\infty)$, an application of formula \eqref{eq:EM2} yields
	$$\sum_{k\ge 0}f_a(k)=\int_0^\infty f_a(x)\dd x+Q_{0,\infty}(a),
	$$
	where $|Q_{0,\infty}(a)|\leq (1/2)\int_0^\infty |f^\prime_a(x)|{\rm d}x$.
	By Lemma \ref{lem:integral},
	$$
	\int_0^\infty f_a(x)\dd x=\frac{1}{\beta a}{\rm B}\Big(1+\frac{1}{\beta},1-\frac{1}{\beta}\Big)=\frac{\pi}{\beta^2 a\sin(\pi/\beta)}.
	$$
	Since
	$$	f^{\prime}_a(x)=\frac{\beta a^{\beta}x^{\beta-1}}{(1+(ax)^\beta)^2}-\frac{2\beta a^{2\beta} x^{2\beta-1}}{(1+(ax)^\beta)^3},
	$$
	we obtain with the help of Lemma \ref{lem:integral}
	$$
		\int_0^\infty |f_a^\prime(x)|{\rm d}x
		\le \int_0^\infty\frac{\beta a^{\beta}x^{\beta-1}}{(1+(ax)^\beta)^2}\dd x+\int_0^\infty\frac{2\beta a^{2\beta} x^{2\beta-1}}{(1+(ax)^\beta)^3}\dd x={\rm B}(1,1)+ 2{\rm B}(2,1)=O(1)
	$$
	as $a\to 0+$.
	Combining all fragments together completes the proof.
\end{proof}

\begin{lemma}\label{lem:(a)series3}
	For $\beta>1$, as $a\to0+$
	\begin{equation}\label{eq:logsum}
		\sum_{k\ge 1}\log\Big(1+\frac{1}{(ak)^\beta}\Big)=\frac{\pi}{a\sin(\pi/\beta)}-\frac{\beta}{2}\log\Big(\frac{1}{a}\Big)-\frac{\beta}{2}\log(2\pi)+O(a^{\min(1, (\beta-1)^2)}).
	\end{equation}
\end{lemma}
\begin{rem}
	If $\beta=2$, there is another way to obtain a counterpart of \eqref{eq:logsum} with a better accuracy. The equality
	$$
	\prod_{k\ge 1}\Big(1+\frac{1}{(ak)^2}\Big)=\frac{a\sinh(\pi/a)}{\pi},
	$$
	which holds for any $a>0$, implies that
	$$
	\sum_{k\ge 1}\log\Big(1+\frac{1}{(ak)^2
	}\Big)=\frac{\pi}{a}-\log\Big(\frac{1}{a}\Big)-\log(2\pi)+O(\eee^{-2\pi/a}),\quad a\to0+.
	$$
\end{rem}
\begin{proof}
	Put $m_a:=\lfloor a^{-\beta-(\beta-1)^{-1}}\rfloor$ and
	split the sum into two parts: one with $k\ge m_a+1$ and the other with $1\le k\le m_a$. We estimate the first sum as follows: as $a\to0+$,
	$$
		\sum_{k\ge m_a+1}\log\Big(1+\frac{1}{(ak)^\beta}\Big)\le \sum_{k\ge m_a+1} (ak)^{-\beta}\le a^{-\beta}\int_{a^{-\beta-(\beta-1)^{-1}}}^\infty x^{-\beta}\dd x=\frac{a^{(\beta-1)^2}}{\beta-1}=O(a^{(\beta-1)^2}).
	$$
	The second sum is equal to
	\begin{multline*}
		\sum_{k=1}^{ m_a}\log\Big(1+\frac{1}{(ak)^\beta}\Big)=\sum_{k=1}^{ m_a}\log\big(1+(ak)^\beta\big)-\beta\sum_{k=1}^{ m_a}\log k+\beta m_a\log\Big(\frac{1}{a}\Big)\\=:S_1(a)-S_2(a)+C_1(a).
	\end{multline*}
	By Stirling's approximation,
	\begin{multline*}
		-S_2(a)=-\beta\Big(m_a\log(m_a)-m_a+\frac{\log(m_a)}{2}+\frac{\log(2\pi)}{2}\Big)+O(a^{\beta+(\beta-1)^{-1}})\\=:-C_2(a)
		+C_3(a)-C_4(a)-C_5 +O(a),\quad a\to0+.
	\end{multline*}
	Note that $-C_5$ appears in the final formula \eqref{eq:logsum}.
	
	To analyze $S_1(a)$, put
	$f_a(x)=\log(1+(ax)^\beta)$ for $x\geq 0$. According to \eqref{eq:EM},
	$$
	S_1(a)=\int_1^{m_a} f_a(x)\dd x+\frac{f_a(1)+f_a(m_a)}{2}+\frac{f^\prime_a(m_a)-f^\prime_a(1)}{12}+ R_{1,m_a}(a).
	$$
	We first observe that as $a\to0+$
	\begin{equation*}
		f_a(1)=\log(1+a^\beta)=O(a^\beta)=O(a)
	\end{equation*}
	and
	$$\frac{f_a(m_a)}{2}=\frac{\log(1+(am_a)^\beta)}{2}=\frac{-\beta\log(1/a)}{2}+\frac{\beta \log (m_a)}{2}+O(a^{\beta(\beta-1+(\beta-1)^{-1})}).
	$$
	Note that the first summand on the right-hand side appears in \eqref{eq:logsum}, whereas the second is equal to $C_4(a)$. Also,
	$O(a^{\beta(\beta-1+(\beta-1)^{-1})}) =O(a) $ as $a\to0+$. Further, using
	$$
	f^\prime_a(x)=\frac{\beta a^\beta x^{\beta-1}}{1+(ax)^\beta},
	$$
	we conclude that
	$$
	f^\prime_a(1)=\frac{\beta a^\beta }{1+a^\beta}=O(a^\beta)=O(a)
	$$
	and, recalling
	$m_a=\lfloor a^{-\beta-(\beta-1)^{-1}}\rfloor$ which particularly entails $\lim_{a\to 0+}am_a=\infty$, that
	$$
	f^\prime_a(m_a)=\frac{\beta a^\beta m_a^{\beta-1} }{1+(am_a)^\beta}\sim \frac{\beta}{m_a}=O(a)
	$$
	as $a\to0+$. In view of
	$$
	f^{\prime\prime}_a(x)=\frac{\beta(\beta-1) a^\beta x^{\beta-2}}{1+(ax)^\beta}-\frac{\beta^2 a^{2\beta} x^{2\beta-2}}{(1+(ax)^\beta)^2},
	$$
	we obtain with the help of Lemma \ref{lem:integral}
	\begin{multline*}
		12 |R_{1,m_a}(a)|\leq \int_0^\infty |f^{\prime\prime}_a(x)|\dd x
		\le \beta(\beta-1) a^\beta \int_0^\infty\frac{ x^{\beta-2}}{1+(ax)^\beta}\dd x+\beta^2 a^{2\beta}\int_0^\infty\frac{ x^{2\beta-2}}{(1+(ax)^\beta)^2}\dd x\\=a\Big((\beta-1)B\Big(1-\frac{1}{\beta},\frac{1}{\beta}\Big)+\beta B\Big(2-\frac{1}{\beta},\frac{1}{\beta}\Big)\Big)=O(a),
		\quad a\to0+.
	\end{multline*}
	Changing the variable $z=(ax)^\beta$ and then integrating by parts yields
	\begin{multline*}
		\int_1^{m_a} \log(1+(ax)^\beta)\dd x=
		\frac{1}{\beta a}\int_{a^\beta}^{(am_a)^\beta} \frac{\log(1+z)}{z^{1-1/\beta}}\dd z\\=
		\big(m_a\log(1+(am_a)^\beta)-\log(1+a^\beta)\big)-\frac{1}{a}\int_{a^\beta}^{(am_a)^\beta} \frac{z^{1/\beta}}{1+z}\dd z=:C_6(a)-C_7(a).
	\end{multline*}
	Write
	$$
	C_6(a) =\beta m_a\Big(\log m_a-\log(1/a)\Big)+m_a\log(1+(am_a)^{-\beta})+O(a).
	$$
	The first summand is equal to $C_2(a)-C_1(a)$, whereas the second
	is $O(a^{(\beta-1)^2})$ as $a\to0+$.
	
	To deal with $C_7(a)$, write
	\begin{equation}\label{eq:H}
		-C_7(a)=-\frac{1}{a}\int_{a^\beta}^{(am_a)^\beta}\Big(z^{1/\beta-1}- \frac{z^{1/\beta-1}}{1+z}\Big)\dd z=-\beta m_a+\beta+\frac{1}{a}\int_{a^\beta}^{(am_a)^\beta}\frac{z^{1/\beta-1}}{1+z}\dd z.
	\end{equation}
	The first summand is equal to $-C_3(a)$.
	We represent the last integral as follows:
	$$
		\frac{1}{a}\int_{a^\beta}^{(am_a)^\beta}\frac{z^{1/\beta-1}}{1+z}\dd z=\frac{1}{a}\int_{0}^{\infty}\ldots-\frac{1}{a}\int_{0}^{a^\beta}\ldots-\frac{1}{a}\int_{(am_a)^\beta}^{\infty}\ldots=:I_1(a)-I_2(a)-I_3(a).
	$$
	Changing the variable $y=(1+z)^{-1}$ we obtain
	$I_1(a)={\rm B}(1-1/\beta, 1/\beta)/a=\frac{\pi}{a\sin(\pi/\beta)}$. This term appears in \eqref{eq:logsum}. By L'H\^opital's rule,
	$$
	\frac{aI_2(a)-\beta a}{a^{\beta+1}}\sim\frac{\beta(1+a^\beta)^{-1}-\beta}{(\beta+1)a^\beta}\to -\frac{\beta}{\beta+1},\quad a\to0+.
	$$
	\noindent Hence, $-I_2(a)=-\beta+O(a^\beta)=-\beta+O(a)$ as $a\to0+$. Note that $-\beta$ cancels out with $\beta$ in \eqref{eq:H}. Finally,
	$$
	I_3(a)\le \frac{1}{a}\int_{(am_a)^\beta}^{\infty} z^{1/\beta-2}\dd x=\frac{(am_a)^{1-\beta}}{(1-1/\beta)a}~\sim~\frac{a^{(\beta-1)^2}}{1-1/\beta}=O(a^{(\beta-1)^2}),
	\quad a\to0+.
	$$
	Combining fragments together we arrive at \eqref{eq:logsum}.
\end{proof}

\subsection{Proof of Theorem \ref{thm:main}(a)}

\noindent{\bf Step 1.} Determining $s_n$, the solution to $\psi^\prime(s)=n$.

Recalling that $r_k=ck^{-\beta}$ for $k\in\mn$, we start with
$$
\psi^\prime(s)=\sum_{k\ge 1}\frac{r_k\eee^s}{r_k\eee^s+1-r_k}=\frac{\eee^s}{\eee^s-1}\sum_{k\ge 1}\frac{1}{1+k^\beta(c(\eee^s-1))^{-1}},\quad s\in\mr.
$$
Using Lemma \ref{lem:(a)series1} with $a:=(c(\eee^s-1))^{-1/\beta}$ we infer, as $s\to\infty$,
$$
	\psi^\prime(s)
	=\frac{\pi c^{1/\beta}(\eee^s-1)^{1/\beta}}{\beta \sin(\pi/\beta)}-\frac{1}{2}+O(\eee^{-s/\beta})=\frac{\pi c^{1/\beta}\eee^{s/\beta}}{\beta \sin(\pi/\beta)}-\frac{1}{2}+O(\eee^{-s\min(1/\beta,1-1/\beta)}).
$$

For large enough $n$, $s=s_n$ is the
solution to the equation $\psi^\prime(s)=n$. Equivalently,
\begin{equation}\label{eq:x}
	\frac{\pi c^{1/\beta}\eee^{s/\beta}}{\beta \sin(\pi/\beta)}-\frac{1}{2}+O(\eee^{-s\min(1/\beta,1-1/\beta)})=n.
\end{equation}
Solving \eqref{eq:x} asymptotically we conclude that
\begin{equation}\label{eq:s}
	s_n=\beta\log \Big(n+\frac{1}{2}\Big)+\beta\log\Big(\frac{\beta\sin(\pi/\beta)}{\pi c^{1/\beta}}\Big)+\varepsilon(n),\quad n\to\infty,
\end{equation}
where $\varepsilon$ is a function satisfying $\lim_{n\to\infty}\varepsilon(n)=0$.
To determine the asymptotic behavior of $\varepsilon$, we substitute the right-hand side of~\eqref{eq:s} into \eqref{eq:x}, thereby obtaining
$$
n=\Big(n+\frac{1}{2}\Big)\eee^{\varepsilon(n)/\beta}-\frac{1}{2}+O\Big(\frac{1}{n^{\min(1,\beta-1)}}\Big),\quad n\to\infty.
$$
This proves that $\varepsilon(n)=O(n^{-\min(2,\beta)})$ as $n\to\infty$.

\noindent {\bf Step 2.} Asymptotic behavior of $\psi^{\prime\prime}$.

In view of \eqref{eq:var_Y0},
\begin{multline*}
	\psi^{\prime\prime}(s) =c\,\eee^s\sum_{k\ge 1}\frac{k^\beta-c}{(k^\beta+c(\eee^s-1))^2}=\frac{\eee^s}{\eee^s-1}\sum_{k\ge1}\frac{k^\beta c^{-1}(\eee^s-1)^{-1}}{(1+k^\beta c^{-1}(\eee^s-1)^{-1})^{2}}\\-\frac{\eee^s}{(\eee^s-1)^2}\sum_{k\ge1}\frac{1}{(1+k^\beta c^{-1}(\eee^s-1)^{-1})^{2}}=:T_1(s)-T_2(s).
\end{multline*}
According to Lemma \ref{lem:(a)series1},
$$
T_2(s)\le \frac{\eee^s}{(\eee^s-1)^2}\sum_{k\ge1}\frac{1}{1+k^\beta c^{-1}(\eee^s-1)^{-1}}=O(\eee^{-s(1-1/\beta)})=o(1),\quad s\to\infty,
$$
whereas by Lemma \ref{lem:(a)series2},
$$
T_1(s) =\frac{\eee^s}{\eee^s-1}\Big(\frac{\pi c^{1/\beta}(\eee^s-1)^{1/\beta}}{\beta^2 \sin(\pi/\beta)}+O(1)\Big)~\sim~ \frac{\pi c^{1/\beta}}{\beta^2 \sin(\pi/\beta)}\eee^{s/\beta},\quad s\to\infty.
$$
Hence,
\begin{equation}\label{eq:(a)variance}
	\psi^{\prime\prime}(s)
	~\sim~ \sigma^2\eee^{s/\beta},\quad s\to\infty,
\end{equation}
where $\sigma^2:=\frac{\pi c^{1/\beta}}{\beta^2 \sin(\pi/\beta)}$. In particular, \eqref{eq:(a)variance}
implies that we are in the setting of Theorem~\ref{thm:caseB}.

\noindent {\bf Step 3.} Asymptotic behavior of $\psi$.

By Lemma \ref{lem:(a)series3} with the same $a$ as before, that is,
$a=(c(\eee^s-1))^{-1/\beta}$,
\begin{multline}\label{eq:psi_k}
	\psi(s)
	=\sum_{k\ge 1}\log\Big(1+\frac{c(\eee^s-1)}{k^\beta}\Big)=\frac{\pi c^{1/\beta}(\eee^s-1)^{1/\beta}}{\sin(\pi/\beta)}-\frac{\beta}{2}\log\big(c^{1/\beta}(\eee^s-1)^{1/\beta}\big)\\-\frac{\beta}{2}\log(2\pi)+o(1)=\frac{\pi c^{1/\beta}\eee^{s/\beta}}{\sin(\pi/\beta)}-\frac{s}{2}-\frac{\log c}{2}-\frac{\beta}{2}\log(2\pi)+o(1),\quad s\to\infty.
\end{multline}

\noindent {\bf Step 4.} End of the proof.

By Theorem \ref{thm:caseB}, combining \eqref{eq:s}, \eqref{eq:(a)variance} and \eqref{eq:psi_k} we arrive at
\begin{multline*}
	\mmp\{Y\ge n\}~\sim~\mmp\{Y=n\}~\sim~\exp\Big(-\beta n\log n-\beta\Big(\log\Big(\frac{\beta\sin(\pi/\beta)}{\pi c^{1/\beta}}\Big)-1\Big)n\\-\frac{\beta+1}{2}\log n-\frac{\beta}{2}\log(2\beta\sin(\pi/\beta))-\frac{1}{2}\log(2\pi)+\frac{1}{2}\log\beta\Big),\quad
	n\to\infty.
\end{multline*}

\subsection{Auxiliary results for $r_k=c\exp(-k^{\beta})$, $\beta\in(0,1)$}

First, we calculate the values of integrals which arise
in the proofs of Lemmas~\ref{lem:exp} and \ref{lem:int2}.

\begin{lemma}\label{lem:coeff}
	For $n\in\mn$, $$\int_1^\infty\frac{(\log y)^n}{y(1+y)}\dd y=n!(1-2^{-n})\zeta(n+1),$$ where $\zeta$ is the Riemann zeta function. For $n\in\mn_0$, $$\int_1^\infty\frac{\log (1+y^{-1})}{y}(\log y)^n\dd y=n!(1-2^{-(n+1)})\zeta(n+2).$$
\end{lemma}
\begin{proof}
	As far as the first integral is concerned, we change the variable $x=\log y$ and then use formula (23.2.7) in \cite{Abramowitz+Stegun:1964}. Integrating the second integral by parts we conclude that it is equal to $(n+1)^{-1}$ times the first integral with $n+1$ replacing~$n$.
\end{proof}

Next, we derive the first-order asymptotic behavior of an integral which appears in subsequent proofs.

\begin{lemma}\label{lem:int}
	For $\beta>0$, $c>0$, $b>c$ and $d\in\mr$,
	$$
	\int_1^\infty\frac{(a\eee^{x^\beta})^cx^d}{(1+a\eee^{x^\beta})^b}\dd x~\sim~ \frac{{\rm B}\,(b-c,c)}{\beta}
	\Big(\log\frac{1}{a}\Big)^{(d+1)/\beta-1},\quad a\to0+,
	$$
	where ${\rm B}$ is the Euler beta function.
\end{lemma}
\begin{proof}
	Changing the variable $y=a\eee^{x^\beta}$ we infer that the integral is equal to
	$$
	\frac{1}{\beta} \int_{a\eee}^\infty \frac{y^{c-1}}{(1+y)^b}\Big(\log\frac{y}{a}\Big)^{(d+1)/\beta-1}\dd y.
	$$
	
	Consider a function $\ell$ defined by
	$\ell(x)=(\log x)^{(d+1)/\beta-1}$ for $x\ge 1$. It is slowly varying at $\infty$. Hence, by Potter's theorem (Theorem 1.5.6 in \cite{Bingham+Goldie+Teugels:1989}), for any $\delta\in(0,\min(b-c,c))$, there exists $X>0$ such that
	$$
	\frac{\ell(u)}{\ell(v)}\le 2\max\Big\{\Big(\frac{u}{v}\Big)^\delta, \Big(\frac{u}{v}\Big)^{-\delta}\Big\}\quad \text{for all } u\ge X, v\ge X.
	$$
	We can and do assume that $X>\eee$.
	Applying the latter inequality with $u=y/a$, $v= 1/a$ we have, for small enough~$a$ \and $y\ge aX$,
	\begin{equation*}
		\Big(\log\frac{1}{a}\Big)^{-(d+1)/\beta+1}
		\frac{y^{c-1}}{(1+y)^b}\Big(\log\frac{y}{a}\Big)^{(d+1)/\beta-1}
		\le2\, 
		\frac{y^{c-1}\max\{y^\delta,y^{-\delta}\}}{(1+y)^b}.
	\end{equation*}
	Since the right-hand side is integrable on 
	$(0,\infty)$,
	invoking the dominated
	convergence theorem we obtain
	\begin{multline*}
		\lim_{a\to0+}\Big(\log\frac{1}{a}\Big)^{-(d+1)/\beta+1}\int_{aX}^\infty \frac{y^{c-1}}{(1+y)^b}\Big(\log\frac{y}{a}\Big)^{(d+1)/\beta-1}\dd y=\int_{0}^\infty \frac{y^{c-1}}{(1+y)^b}\dd y\\={\rm B}\,(b-c,c)<\infty.
	\end{multline*}
	Since
	$X>\eee$, we are left with
	\begin{multline*}
		\Big(\log\frac{1}{a}\Big)^{-(d+1)/\beta+1}\int_{a\eee}^{aX} \frac{y^{c-1}}{(1+y)^b}\Big(\log\frac{y}{a}\Big)^{(d+1)/\beta-1}\dd y\\\le c^{-1} a^c(X^c-\eee^c)\max\big(1, (\log X)^{(d+1)/\beta-1}\big)\Big(\log\frac{1}{a}\Big)^{-(d+1)/\beta+1}\to 0,\quad a\to0+.
	\end{multline*}
\end{proof}

Next, we provide asymptotic expansions for integrals which appear in Lemmas \ref{lem:(b)series1} and \ref{lem:(b)series2} after applications of
the Euler-Maclaurin formula.

\begin{lemma}\label{lem:exp}
	For $\beta\in(0,1)$, as $a\to 0+$,
	\begin{equation}\label{eq:int_exp}
		\int_1^\infty\frac{\dd x}{1+a\eee^{x^\beta}}=\Big(\log\frac{1}{a}\Big)^{1/\beta}-1+\frac{2}{\beta}\sum_{\text{odd }n=1}^{\lfloor 1/\beta\rfloor-1 }\binom{1/\beta-1}{n}c_n\Big(\log\frac{1}{a}\Big)^{1/\beta-1-n}+g(a).
	\end{equation}
	where
	$c_n:=n!(1-2^{-n})\zeta(n+1)$ for $n\in\mn$; $g(a)=O\big((\log (1/a))^{\{1/\beta\}-1}\big)$ as $a\to 0+$ if $1/\beta$ is not integer, and $g(a)=O\big(a(\log (1/a))^{1/\beta-1}\big)$ as $a\to 0+$ if $1/\beta$ is integer.
	
\end{lemma}
\begin{proof}
	Write, for $a\in (0, 1/\eee)$,
	$$
	\int_1^\infty\frac{\dd x}{1+a\eee^{x^\beta}}=\int_1^{(\log(1/a))^{1/\beta} }\ldots+\int_{(\log(1/a))^{1/\beta} }^\infty\ldots=:I_1(a)+I_2(a).
	$$
	Changing the variable $y=a\eee^{x^\beta}$ yields
	\begin{multline*}
		I_1(a)=\frac{1}{\beta} \int_{\eee a}^1 \frac{(\log (y/a))^{1/\beta-1}}{y(1+y)}\dd y\\=\frac{1}{\beta}
		\Big(\int_{\eee a}^1 \frac{(\log (y/a))^{1/\beta-1}}{y}\dd y-\int_{\eee a}^1 \frac{(\log (y/a))^{1/\beta-1}}{1+y}\dd y\Big)=:I_{1,1}(a)-I_{1,2}(a).
	\end{multline*}
	Plainly, $I_{1,1}(a)=\big(\log(1/a) \big)^{1/\beta}-1$. Further, assuming that $1/\beta$ is not integer,
	\begin{multline*}
		I_{1,2}(a)=\frac{1}{\beta}\Big(\log\frac{1}{a}\Big)^{1/\beta-1}\int_{\eee a}^1\frac{\Big(1+\frac{\log y}{\log(1/a)}\Big)^{1/\beta-1}}{1+y}\,\dd y\\=\frac{1}{\beta} \Big(\log\frac{1}{a}\Big)^{1/\beta-1}\Big(\int_{\eee a}^1\frac{\sum_{n=0}^{\lfloor 1/\beta\rfloor-1} \binom{1/\beta-1}{n}\Big(\frac{\log y}{\log(1/a)}\Big)^n}{1+y}\,\dd y+\int_{\eee a}^1\frac{O\Big(\Big(\frac{\log y}{\log(1/a)}\Big)^{\lfloor 1/\beta\rfloor}\Big)}{1+y}{\rm d}y\Big)
	\end{multline*}
	having utilized the Taylor expansion for the function $x\mapsto (1+x)^{1/\beta-1}$,
	$x\in (-1,1)$. Changing the variable $y=1/z$ in the first integral appearing in Lemma~\ref{lem:coeff} we infer
	$\int_{0}^1\frac{(\log z)^n}{1+z}\,\dd z=(-1)^n c_n $ for $n\in\mn$. Put $c_0:=\log 2$, so that the latter equality also holds for $n=0$. By L'H\^{o}pital's rule, $$\int_0^{a\eee}\frac{(\log y)^n}{1+y}\,\dd y~\sim~(-1)^n\eee a \Big(\log \frac{1}{a}\Big)^n,\quad a\to 0+.$$
	\vspace{0pt}
	Thus,
	\begin{multline*}
		I_{1,2}(a)=\frac{1}{\beta}\sum_{n=0}^{\lfloor 1/\beta\rfloor-1}\binom{1/\beta-1}{n}\Big(\int_0^1 \frac{(\log y)^n}{1+y}\,\dd y-\int_0^{\eee a}\frac{(\log y)^n}{1+y}\,\dd y\Big)\Big(\log\frac{1}{a}\Big)^{1/\beta-1-n}\\+O\Big(\Big(\log \frac{1}{a}\Big)^{\{1/\beta\}-1}\Big)=\frac{1}{\beta}\sum_{n=0}^{\lfloor 1/\beta\rfloor-1}\binom{1/\beta-1}{n}(-1)^n c_n \Big(\log\frac{1}{a}\Big)^{1/\beta-1-n}+O\Big(a\Big(\log \frac{1}{a}\Big)^{1/\beta-1}\Big)\\\!+O\Big(\Big(\log \frac{1}{a}\Big)^{\{1/\beta\}-1}\Big)\!=\!\frac{1}{\beta}\sum_{n=0}^{\lfloor 1/\beta\rfloor-1}\binom{1/\beta-1}{n}(-1)^n c_n \Big(\log\frac{1}{a}\Big)^{1/\beta-1-n}\!+O\Big(\Big(\log \frac{1}{a}\Big)^{\{1/\beta\}-1}\Big).
	\end{multline*}
	If $1/\beta$ is integer, then the second big-oh term is absent, whence
	\begin{equation*}
		I_{1,2}(a)=\frac{1}{\beta}\sum_{n=0}^{1/\beta-1}\binom{1/\beta-1}{n}(-1)^n c_n \Big(\log\frac{1}{a}\Big)^{1/\beta-1-n}+O\Big(a\Big(\log \frac{1}{a}\Big)^{1/\beta-1}\Big).
	\end{equation*}
	Summarizing, as $a\to0+$
	$$I_1(a)=\Big(\log\frac{1}{a}\Big)^{1/\beta}-1-\frac{1}{\beta} \sum_{ n=0}^{\lfloor 1/\beta\rfloor-1}\binom{1/\beta-1}{n}(-1)^nc_n\Big(\log\frac{1}{a}\Big)^{1/\beta-1-n}+g(a).
	$$
	Arguing similarly we obtain $$I_2(a)=\frac{1}{\beta}\sum_{n=0}^{\lfloor 1/\beta\rfloor-1}\binom{1/\beta-1}{n}c_n\Big(\log\frac{1}{a}\Big)^{1/\beta-1-n}+g(a),\quad a\to 0+,$$
	thereby arriving at \eqref{eq:int_exp}.
\end{proof}

\begin{lemma}\label{lem:int2}
	For $\beta\in(0,1)$,
	\begin{multline}\label{eq:int2}
		\int_1^\infty\log\Big(1+\frac{1}{a\eee^{x^\beta}}\Big)\dd x=\frac{\beta}{\beta+1} \Big(\log\frac{1}{a}\Big)^{1+1/\beta}\\+\frac{2}{\beta}\sum_{\text{even }n=0}^{\lfloor 1/\beta\rfloor-1}\binom{1/\beta-1}{n}f_n\Big(\log\frac{1}{a}\Big)^{1/\beta-1-n}-\log\frac{1}{a}+\frac{1}{1+\beta}+h(a),
		\quad a\to0+,
	\end{multline}
	where $f_n=n!(1-2^{-(n+1)})\zeta(n+2)$ for $n\in \mn_0$; $h(a)=O\big((\log (1/a))^{\{1/\beta\}-1}\big)$ as $a\to 0+$ if $1/\beta$ is not integer, and $h(a)=O\big(a(\log (1/a))^{1/\beta-1}\big)$ as $a\to 0+$ if $1/\beta$ is integer.
\end{lemma}
\begin{proof}
	For $a\in (0, 1/\eee)$, we use a representation
	$$
	\int_1^\infty\log\Big(1+\frac{1}{a\eee^{x^\beta}}\Big)\dd x=\int_1^{(\log(1/a))^{1/\beta}}\ldots+\int_{(\log(1/a))^{1/\beta}}^\infty\ldots=:J_1(a)+J_2(a).
	$$
	\noindent Further,
	$$
	J_1(a)=\int_1^{(\log(1/a))^{1/\beta}}\Big(\log(1+a\eee^{x^\beta})-\log(a\eee^{x^\beta})\Big)\dd x=:J_{1,1}(a)+J_{1,2}(a).
	$$
	A direct calculation shows that
	$J_{1,2}(a)=\beta(1+\beta)^{-1}\big(\log(1/a)\big)^{1+1/\beta}-\log 1/a+(1+\beta)^{-1}$. As has been done in the proof of Lemma \ref{lem:exp} for $I_{1,2}(a)$ and $I_2(a)$,
	we obtain asymptotic expansions for $J_{1,1}(a)$ and $J_2(a)$:
	$$J_{1,1}(a)=\frac{1}{\beta}\sum_{n=0}^{\lfloor 1/\beta\rfloor-1}\binom{1/\beta-1}{n}(-1)^nf_n\Big(\log\frac{1}{a}\Big)^{1/\beta-1-n}+h(a),\quad a\to 0+,$$
	$$J_2(a)=\frac{1}{\beta}\sum_{n=0}^{\lfloor 1/\beta\rfloor-1}\binom{1/\beta-1}{n}f_n\Big(\log\frac{1}{a}\Big)^{1/\beta-1-n}+h(a),
	\quad a\to 0+.$$
	Summing up $J_{1,1}(a)$, $J_{1,2}(a)$ and $J_{2}(a)$ yields \eqref{eq:int2}.
\end{proof}

Now we derive asymptotics of three series which appear in the proof of Theorem \ref{thm:main}(b).
\begin{lemma}\label{lem:(b)series1}
	For $\beta\in (0,1)$, as $a\to 0+$,
	\begin{equation*}
		\sum_{k\ge1}\frac{1}{1+a\eee^{k^\beta}}=\Big(\log\frac{1}{a}\Big)^{1/\beta}+\frac{2}{\beta}\sum_{\text{odd }n=1}^{\lfloor 1/\beta\rfloor-1 }\binom{1/\beta-1}{n}c_n\Big(\log\frac{1}{a}\Big)^{1/\beta-1-n}-\frac{1}{2}+t(a),
	\end{equation*}
	where $t(a)=O\big((\log (1/a))^{\max(\{1/\beta\}-1,\,1-1/\beta)
	}\big)$ as $a\to 0+$ if $1/\beta$ is not integer, and $t(a)=O\big((\log (1/a))^{1-1/\beta}\big)$ as $a\to 0+$ if $1/\beta$ is integer, and, as before, $c_n:=n!(1-2^{-n})\zeta(n+1)$ for $n\in\mn$.
\end{lemma}
\begin{proof}
	For each $a>0$, put $f_a (x)=(1+a\eee^{x^\beta})^{-1}$ for $x\geq 0$. The first two derivatives of $f_a$ are
	$$f^\prime_a(x)=-\frac{\beta ax^{\beta-1}\eee^{x^\beta}}{(1+a\eee^{x^\beta})^2}$$ and $$f^{\prime\prime}_a(x)=\frac{2\beta^2 a^2x^{2(\beta-1)}\eee^{2x^\beta}}{(1+a\eee^{x^\beta})^3}-\frac{\beta^2 ax^{2(\beta-1)}\eee^{x^\beta}}{(1+a\eee^{x^\beta})^2}+\frac{\beta(1-\beta) ax^{\beta-2}\eee^{x^\beta}}{(1+a\eee^{x^\beta})^2}.$$
	
	Since $f_a$ is twice continuously differentiable on $[1,\infty)$ with $f_a(\infty)=f_a^\prime(\infty)=0$, $f_a(1)=(1+\eee a)^{-1}=1+O(a)$ and $f_a^\prime(1)=O(a)$ as $a\to 0+$, an application of formula~\eqref{eq:EM} with $m=1$ and $n=\infty$ yields
	$$
	\sum_{k\ge 1}f_a(k)=\int_1^\infty f_a(x)\dd x+\frac{1}{2}+R_{1,\infty}(a)+O(a),
	$$
	\noindent where $|R_{1,\infty}(a)|\leq (1/12)\int_1^\infty |f^{\prime\prime}_a(x)|{\rm d}x$. The asymptotic expansion for $\int_1^\infty f_a(x){\rm d}x$ can be found in Lemma \ref{lem:exp}. By Lemma \ref{lem:int},
	\begin{multline*}
		\int_1^\infty |f_a^{\prime\prime}(x)|\dd x\leq 2\beta^2\int_1^\infty \frac{ (a\eee^{x^\beta})^2 x^{2(\beta-1)}}{(1+a\eee^{x^\beta})^3}{\rm d}x+\beta^2\int_1^\infty \frac{a\eee^{x^\beta}x^{2(\beta-1)}}{(1+a\eee^{x^\beta})^2}{\rm d}x\\+\beta(1-\beta)\int_1^\infty \frac{a\eee^{x^\beta}x^{\beta-2}}{(1+a\eee^{x^\beta})^2}{\rm d}x=O\Big(\Big(\log \frac{1}{a}\Big)^{1-1/\beta}\Big), \quad a\to0+.
	\end{multline*}
	Combining fragments together completes the proof.
\end{proof}

\begin{lemma}\label{lem:(b)series3}
	For $\beta\in (0,1)$, $$\sum_{k\ge 1}\frac{\eee^{k^\beta}}{(1+a\eee^{k^\beta})^2}~\sim~\frac{1}{\beta a}\Big(\log\frac{1}{a}\Big)^{1/\beta-1},\quad a\to 0+.$$
\end{lemma}
\begin{proof}
	For each $a>0$, put $$f_a(x)=\frac{a\eee^{x^\beta}}{(1+a\eee^{x^\beta})^2},\quad x\geq 0.$$
	Since $f_a$ is continuously differentiable on $[1,\infty)$ with $f_a(1)=O(a)=o(1)$ as $a\to 0+$ and $\lim_{x\to\infty}f_a(x)=0$, an application of formula \eqref{eq:EM2} yields
	$$\sum_{k\ge 1}f_a(k)=\int_1^\infty f_a(x)\dd x+Q_{1,\infty}(a)+o(1),
	$$
	where $|Q_{1,\infty}(a)|\leq (1/2)\int_1^\infty |f^\prime_a(x)|{\rm d}x$.
	Invoking Lemma \ref{lem:int} we obtain, as $a\to0+$
	$$
		\int_1^\infty |f_a^\prime(x)|\dd x \!\leq\!
		\beta\Big(\int_1^\infty\frac{a\eee^{x^\beta} x^{\beta-1}}{(1+a\eee^{x^\beta})^2}\dd x+2\int_1^\infty\frac{(a\eee^{x^\beta})^2 x^{\beta-1}}{(1+a\eee^{x^\beta})^3}\dd x\Big)\!=\!O(1)=o\Big(\Big(\log\frac{1}{a}\Big)^{1/\beta-1}\Big)
	$$
	and also
	$$\int_1^\infty f_a(x)\dd x~\sim~ \frac{1}{\beta}\Big(\log\frac{1}{a}\Big)^{1/\beta-1},$$ which completes the proof.
\end{proof}

\begin{lemma}\label{lem:(b)series2}
	For $\beta\in (0,1)$,
	\begin{multline*}
		\sum_{k\ge 1}\log\Big(1+\frac{1}{a\eee^{k^\beta}}\Big)=\frac{\beta}{1+\beta}\Big(\log\frac{1}{a}\Big)^{1+1/\beta}\\
		+\frac{2}{\beta}\sum_{\text{even }n=0}^{\lfloor 1/\beta\rfloor -1}\binom{1/\beta-1}{n}f_n\Big(\log\frac{1}{a}\Big)^{1/\beta-1-n}-\frac{1}{2}\log\frac{1}{a}-\zeta(-\beta)+o(1),\quad a\to0+,
	\end{multline*}
	where, as before,
	$f_n=n!(1-2^{-(n+1)})\zeta(n+2)$ for $n\in \mn_0$, and $\zeta$ is the Riemann zeta function.
\end{lemma}
\begin{proof}
	For each $a>0$, put
	$m_a:=\lfloor (\log 1/a)^{1/\beta}\rfloor$ and $$f_{1,a}(x):=\log\Big(1+\frac{1}{a\eee^{x^\beta}}\Big),\quad x\geq 0,$$ and then write
	$$\sum_{k\ge 1}\log\Big(1+\frac{1}{a\eee^{k^\beta}}\Big)=\sum_{k=1}^{m_a}\ldots+\sum_{k\ge m_a+1}\ldots=:S_1(a)+S_2(a).$$ Since $f_{1,a}$ is twice continuously differentiable on $[1,\infty)$ with $\lim_{a\to 0+}f_{1,a}(m_a+1)=\log 2$ and  $\lim_{x\to\infty} f_{1,a}(x)=\lim_{x\to\infty}f_{1,a}^\prime(x)=0$, an application of formula~\eqref{eq:EM} with $m=m_a+1$ and $n=\infty$ yields
	$$
	S_2(a)=\int_{m_a+1}^\infty f_{1,a}(x)\dd x+\frac{\log 2}{2}-\frac{ f^\prime_{1,a}(m_a+1)}{12}+R_{m_a+1,\,\infty}+o(1)$$ as $a\to0+$, where $|R_{m_a+1,\,\infty}|\leq (1/12)\int_{m_a+1}^\infty |f^{\prime\prime}_{1,a}(x)|{\rm d}x$.
	Using
	$f_{1,a}^\prime(x)=-\frac{\beta x^{\beta-1}}{1+a\eee^{x^\beta}}$ we infer
	$|f_{1,a}^\prime(m_a+1)|=
	O((\log 1/a)^{1-1/\beta})=o(1)$ as $a\to0+$. Further,
	$$
	f_{1,a}^{\prime\prime}(x)=-\frac{\beta(\beta-1)x^{\beta-2}}{1+a\eee^{x^\beta}}+\beta^2 \frac{a\eee^{x^\beta}x^{2(\beta-1)}}{(1+a\eee^{x^\beta})^2}=:-I_a(x)+J_a(x).
	$$
	Since as $a\to0+$ $$\int_{m_a+1}^\infty I_a(x)\dd x\le \beta(1-\beta)\int_{(\log 1/a)^{1/\beta}}^\infty x^{\beta-2}\dd x=\beta(\log1/a)^{1-1/\beta}=o(1)$$ and, by Lemma \ref{lem:int}, $\int_{m_a+1}^\infty J_a(x){\rm d}x\leq \int_1^\infty J_a(x){\rm d}x=O((\log 1/a)^{1-1/\beta})=o(1)$, we conclude that $\int_{m_a+1}^\infty |f_{1,a}^{\prime\prime}(x)|\dd x=o(1)$ as $a\to 0+$. Summarizing,
	$$S_2(a)=\int_{m_a+1}^{\infty} f_{1,a}(x)\dd x+\frac{\log 2}{2}+o(1)=\int_{m_a}^{\infty} f_{1,a}(x)\dd x-\frac{\log 2}{2}+o(1),\quad a\to0+$$ because $\lim_{a\to 0+}\int_{m_a}^{m_a+1}f_{1,a}(x){\rm d}x=\log 2$.
	
	Passing to $S_1(a)$ we write
	$$S_1(a)=\sum_{k=1}^{m_a}\log(1+a\eee^{k^\beta})+m_a\log1/a-\sum_{k=1}^{m_a}k^\beta=:C(a)+D(a)-E(a).$$
	By Theorem 14 in \cite{Schumacher:2022}, as $a\to0+$
	\begin{equation*}
		E(a)=\frac{1}{1+\beta}(\log 1/a)^{1+1/\beta}-(\{(\log 1/a)^{1/\beta}\}-1/2)\log 1/a+\zeta(-\beta)+o(1).
	\end{equation*}
	To obtain an asymptotic expansion for $C(a)$,
	put, for each $a>0$,
	$$f_{2,a}(x)=\log\big(1+a\eee^{x^\beta}\big),\quad x\geq 0.$$
	\noindent Since $f_{2,a}$ is twice continuously differentiable on $[1,\infty)$ with $\lim_{a\to 0+}f_{2,a}(m_a)=\log 2$ and $\lim_{a\to 0+} f_{2,a}(1)=0$, an application of formula \eqref{eq:EM} with $m=1$ and $n=m_a$ yields
	$$
	C(a)=\int_1^{m_a}f_{2,a}(x)\dd x+\frac{\log 2}{2} +\frac{f^\prime_{2,a}(m_a)-f^\prime_{2,a}(1)}{12}+R_{1,\,m_a}+o(1)
	$$
	as $a\to0+$, where $|R_{1,\,m_a}|\leq (1/12)\int_1^{m_a}|f^{\prime\prime}_{2,a}(x)|{\rm d}x$. In view of
	$f_{2,a}^\prime(x)=\frac{a\beta x^{\beta-1}\eee^{x^\beta}}{1+a\eee^{x^\beta}}$ we infer $f^\prime_{2,a}(1)=o(1)$ and $f_{2,a}^\prime(m_a)=
	O((\log1/a)^{1-1/\beta})=o(1)$ as $a\to0+$. Further,
	$$
		f_{2,a}^{\prime\prime}(x)=\frac{a\beta^2x^{2(\beta-1)}\eee^{x^\beta}}{1+a\eee^{x^\beta}}-\frac{a\beta (1-\beta)x^{\beta-2}\eee^{x^\beta}}{1+a\eee^{x^\beta}}-\frac{a^2\beta^2x^{2(\beta-1)}\eee^{2x^\beta}}{(1+a\eee^{x^\beta})^2}=:F_a(x)-G_a(x)-H_a(x).
	$$
	By L'H\^opital's rule, as $a\to0+$
	$$
	\int_{1}^{m_a} F_a(x)\dd x\le a\beta^2 \int_{1}^{(\log 1/a)^{1/\beta}} x^{2(\beta-1)}\eee^{x^\beta}\dd x~\sim~\beta (\log 1/a)^{1-1/\beta}=o(1).
	$$
	Since $(1-\beta)^{-1}G_a(x)\le \beta^{-1} F_a(x)$ and $H_a(x)\le F_a(x)$ for $x\ge 1$, we conclude that $\int_{1}^{m_a} |f_{2,a}^{\prime\prime}(x)|\dd x=o(1)$ as $a\to0+$.
	Finally, we have to deal with
	$\int_1^{m_a}f_{2,a}(x)\dd x$.
	To this end, write
	$$
	\int_1^{m_a}f_{2,a}(x)\dd x=\int_1^{m_a}f_{1,a}(x)\dd x +\int_1^{m_a}\log(a\eee^{x^\beta})\dd x
	$$
	and observe that
	$$
	\int_1^{m_a} \log(a\eee^{x^\beta})\dd x= (m_a-1)\log a+\frac{1}{1+\beta}(m_a^{1+\beta}-1).
	$$
	Noting that
	\begin{multline*}
		\frac{1}{1+\beta}\big(m_a^{1+\beta}-(\log 1/a)^{1+1/\beta}\big)+\{(\log 1/a)^{1/\beta}\}\log 1/a\\~\sim~ \frac{\beta}{2} (\{(\log 1/a)^{1/\beta}\})^2 (\log 1/a)^{1-1/\beta}
		=o(1), \quad a\to 0+
	\end{multline*}
	and combining all fragments together we obtain
	$$
	S_1(a)=\int_1^{m_a}f_{1,a}(x)\dd x+ \frac{1}{2}(\log 1/a+\log 2) -\frac{1}{1+\beta}-\zeta(-\beta)+o(1),\quad a\to 0+
	$$
	and thereupon
	$$
		\sum_{k\ge 1}\log\Big(1+\frac{1}{a\eee^{k^\beta}}\Big)=S_1(a)+S_2(a)=\int_1^{\infty}f_{1,a}(x)\dd x+\frac{1}{2}\log 1/a-\frac{1}{1+\beta}-\zeta(-\beta)+o(1).
	$$
	An application of Lemma \ref{lem:int2} completes the proof.
\end{proof}

\subsection{Proof of Theorem \ref{thm:main}(b)}

{\bf Step 1.} Determining $s_n$, the solution to $\psi^{\prime}(s)=n$.

Recalling that $r_k=c\eee^{-k^\beta}$ for $\beta\in (0,1)$ and $k\in\mn$,
we start with
$$
\psi^\prime(s)=\sum_{k\ge 1}\frac{r_k\eee^s}{r_k\eee^s+1-r_k}=\frac{\eee^s}{\eee^s-1}\sum_{k\ge 1}\frac{1}{1+
	(c(\eee^s-1))^{-1}\eee^{k^\beta}},\quad s\in\mr.
$$
Using Lemma \ref{lem:(b)series1} with $a=(c(\eee^s-1))^{-1}$ yields as $s\to\infty$
\begin{multline*}
	\psi^\prime(s)=(\log c+s)^{1/\beta}+\frac{2}{\beta}\sum_{\text{odd }j=1}^{\lfloor 1/\beta\rfloor-1}\binom{1/\beta-1}{j}c_j(\log c+s)^{1/\beta-1-j}-\frac{1}{2}+t(\eee^{-s}/c).
\end{multline*}
For large enough $n$, $s=s_n$ is a solution to the equation $\psi^\prime(s)=n$, that is,
\begin{equation}\label{eq:inter}
	(\log c+s_n)^{1/\beta}+\frac{2}{\beta}\sum_{\text{odd }j=1}^{\lfloor 1/\beta\rfloor-1}\binom{1/\beta-1}{j}c_j(\log c+s_n)^{1/\beta-1-j}-\frac{1}{2}+t(\eee^{-s_n}/c)=n.
\end{equation}
Hence,
\begin{equation}\label{eq:s(b)}
	\log c+s_n =n^\beta(1+\varepsilon(n)),\quad n\to\infty,
\end{equation}
where $\varepsilon$ is a function satisfying $\lim_{n\to\infty}\varepsilon(n)=0$. To determine the asymptotic behavior of $\varepsilon$, we substitute the right-hand side of \eqref{eq:s(b)} into \eqref{eq:inter} and obtain
\begin{equation}\label{eq:eps}
	1=(1+\varepsilon(n))^{1/\beta}+\frac{2}{\beta}\sum_{\text{odd }j=1}^{\lfloor 1/\beta\rfloor-1}\binom{1/\beta-1}{j}c_j\frac{(1+\varepsilon(n))^{1/\beta-1-j}}{n^{\beta(1+j)}}-\frac{1}{2n}+O\big(n^{-\gamma}\big)
\end{equation}
as $n\to\infty$, where $\gamma=\min(1+\beta-\beta\{1/\beta\}, 2-\beta)>1$ if $1/\beta$ is not integer and $\gamma=2-\beta>1$ if $1/\beta$ is integer.

Now we have to find an asymptotic expansion of $\varepsilon$ of length $\ell$, where $\ell=\lfloor (1+\beta)/(2\beta)\rfloor$.
Assume first that $\ell=1$. Then $\varepsilon(n)=O(n^{-1})$ if $\beta\in [1/2, 1)$ and $\varepsilon(n)=O(n^{-2\beta})$ if $\beta\in (1/3, 1/2)$. It follows from \eqref{eq:expon(b)} given below that in both cases $\varepsilon$ gives a negligible contribution to the right-hand side of \eqref{eq:expon(b)}. Assume now that $\ell\geq 2$. Then
\begin{equation*}
	\varepsilon(n)=A_{1,\,\beta}n^{-2\beta}+A_{2,\,\beta}n^{-4\beta}+\ldots+A_{\ell-1,\,\beta}n^{-2(\ell-1)\beta}+O(n^{-1}), \quad n\to\infty
\end{equation*}
if $\beta\in(1/(2\ell), 1/(2\ell-1)]$, and
\begin{equation}\label{eq:eps_precise1}
	\varepsilon(n)=A_{1,\,\beta}n^{-2\beta}+A_{2,\,\beta}n^{-4\beta}+\ldots+A_{\ell-1,\beta}n^{-2(\ell-1)\beta}+O(n^{-2\ell\beta}), \quad n\to\infty
\end{equation}
if $\beta\in(1/(2\ell+1),1/(2\ell)]$. While each of the first $\ell-1$ summands gives a nonzero contribution to the right-hand side of \eqref{eq:expon(b)}, the big-oh terms do not because obviously $O(n^{-1})=o(n^{\beta-1})$ and $O(n^{-2\beta-1})=o(n^{-\beta-1})$ and, whenever $\beta\in((2\ell+1)^{-1},(2\ell)^{-1}]$, $O(n^{-2\ell\beta})=o(n^{\beta-1})$ and $O(n^{-2\beta(\ell+1)})=o(n^{-\beta-1})$. The coefficients~$A_{i,\,\beta}$ are uniquely determined and can be calculated as follows. Substitute in \eqref{eq:eps} the $\beta$-dependent asymptotic expansion for $\varepsilon$, take an appropriate number of terms in the Taylor's expansions of $x\mapsto (1+x)^{1/\beta}$ and $x\mapsto (1+x)^{1/\beta-1-j}$ and then equate to $0$ the coefficients in front of $n^{-2j\beta}$ for $j=1,\ldots, \ell-1$.
For instance, implementing this programme for $n^{-2\beta}$ alone we obtain $A_{1,\,\beta}/\beta+(2/\beta)\big(1/\beta-1\big)c_1=0$, whence
\begin{equation}\label{eq:inter2}
	A_{1,\,\beta}=-2(1/\beta-1)c_1.
\end{equation}
Equating to $0$ the coefficient in front of $n^{-4\beta}$ yields $$\frac{1}{\beta}A_{2,\,\beta}+\binom{1/\beta}{2}(A_{1,\,\beta})^2+\frac{2}{\beta}\Big(\frac{1}{\beta}-1\Big)\Big(\frac{1}{\beta}-2\Big)c_1A_{1,\,\beta}+\frac{2}{\beta}\binom{1/\beta-1}{3}c_3=0.$$ This in combination with \eqref{eq:inter2} enables us to compute $A_{2,\,\beta}$.

\noindent {\bf Step 2.} Asymptotic behavior of $\psi^{\prime\prime}$ and the denominator in \eqref{eq:caseB}.

According to \eqref{eq:var_Y0},
\begin{multline*}
	\psi^{\prime\prime}(s)
	=\frac{\eee^s}{c(\eee^s-1)^2}\sum_{k\ge1}\frac{\eee^{k^\beta}}{(1+
		(c(\eee^s-1))^{-1}\eee^{k^\beta})^2}\\-\frac{\eee^s}{(\eee^s-1)^2}\sum_{k\ge1}\frac{1}{(1+
		(c(\eee^s-1))^{-1}\eee^{k^\beta})^2}=:R_1(s)-R_2(s).
\end{multline*}
By Lemmas \ref{lem:(b)series1} and \ref{lem:(b)series3}, respectively, with $a=(c(\eee^s-1))^{-1}$,
$$R_2(s)
\le \frac{\eee^s}{(\eee^s-1)^2}\sum_{k\ge1}\frac{1}{1+(c(\eee^s-1))^{-1}\eee^{k^\beta} }=O
(\eee^{-s}s^{1/\beta})=o(1)$$
and $R_1(s) \sim \beta^{-1}s^{1/\beta-1}$ as $s\to\infty$. Hence
\begin{equation}\label{eq:psi}
	\psi^{\prime\prime}(s)
	~\sim~ \beta^{-1}s^{1/\beta-1},\quad s\to\infty.
\end{equation}
In particular, \eqref{eq:psi} ensures that we are again in the setting of Theorem \ref{thm:caseB}.

Using \eqref{eq:psi} for the first equivalence below and $s_n\sim n^\beta$, which is a consequence of \eqref{eq:s(b)}, for the second we obtain
\begin{equation}\label{eq:(b)denominator}
	\frac{1}{(2\pi\psi^{\prime\prime}(s_n))^{1/2}}~\sim~\Big(\frac{\beta}{2\pi s_n^{1/\beta-1}}\Big)^{1/2}~\sim~\Big(\frac{\beta}{2\pi n^{1-\beta}}\Big)^{1/2},
	\quad n\to\infty.
\end{equation}

\noindent {\bf Step 3.} Asymptotic behavior of $\psi$ and the numerator in \eqref{eq:caseB}.

By Lemma \ref{lem:(b)series2}, with the same $a=(c(\eee^{s}-1))^{-1}$ as before, as $s\to\infty$
\begin{multline}\label{eq:psi_k(b)}
	\psi(s)=\sum_{k\ge 1}\log\Big(1+\frac{c(\eee^s-1)}{\eee^{k^\beta}}\Big)=\frac{\beta}{1+\beta} (\log c+s)^{1+1/\beta}
	\\+\frac{2}{\beta}\sum_{\text{even }j=0}^{\lfloor 1/\beta\rfloor-1}\binom{1/\beta-1}{j}f_j (\log c+s)^{1/\beta-1-j}-\frac{1}{2} (\log c+s) -\zeta(-\beta)+o(1).
\end{multline}
In view of \eqref{eq:psi_k(b)}, \eqref{eq:s(b)} and $\psi^\prime(s_n)=n$, the numerator in \eqref{eq:caseB} is
\begin{multline*}
	\exp(\psi(s_n)-s_n\psi^\prime(s_n))
	=\exp\Big(\frac{\beta n^{1+\beta}}{1+\beta}(1+\varepsilon(n))^{1+1/\beta}\\+\frac{2}{\beta}\sum_{\text{even }j=0}^{\lfloor 1/\beta\rfloor -1}\binom{1/\beta-1}{j}f_jn^{1-\beta(1+j)}(1+\varepsilon(n))^{1/\beta-1-j}-\frac{n^\beta}{2}(1+\varepsilon(n))-\zeta(-\beta)\\
	-n^{1+\beta}(1+\varepsilon(n))+n\log c+o(1)\Big),\quad n\to\infty.
\end{multline*}
Using $$\frac{\beta n^{1+\beta}}{1+\beta}(1+(1+1/\beta)\varepsilon(n))-n^{1+\beta}(1+\varepsilon(n))=-\frac{1}{1+\beta}n^{1+\beta}$$ and $$n^{\beta}\varepsilon(n)=n^{\beta}O(n^{-\min(1,2\beta)})=o(1),\quad n\to\infty$$
we obtain
\begin{multline}\label{eq:expon(b)}
	\exp(\psi(s_n)-s_n\psi^\prime(s_n))
	=\exp\Big(-\frac{1}{1+\beta}n^{1+\beta}+\frac{\beta n^{1+\beta}}{1+\beta}\sum_{i\ge2}\binom{1+1/\beta}{i}(\varepsilon(n))^i\\+\frac{2}{\beta}\sum_{\text{even }j=0}^{\lfloor 1/\beta\rfloor -1}\binom{1/\beta-1}{j}f_jn^{1-\beta(1+j)}(1+\varepsilon(n))^{1/\beta-1-j}-\frac{n^\beta}{2}-\zeta(-\beta)
	+n\log c+o(1)\Big)
\end{multline}
as $n\to\infty$ (if $1/\beta$ is integer, the summation in the first sum extends over $i\leq 1+1/\beta$).
Put
\begin{equation*}
	\alpha_{\beta}(u):=A_{1,\,\beta}u^{-2\beta}+A_{2,\,\beta}u^{-4\beta}+\ldots+A_{\ell-1,\,\beta}u^{-2(\ell-1)\beta},\quad u>0
\end{equation*}
whenever $\ell\geq 2$ and $\alpha_{\beta}(u):=0$ if $\ell=1$. Then
\begin{multline}\label{eq:expon1(b)}
	\exp(\psi(s_n)-s_n\psi^\prime(s_n))\!=\!
	\exp\Big(-\frac{1}{1+\beta}n^{1+\beta}+\frac{\beta n^{1+\beta}}{1+\beta}\sum_{i=2}^{\ell}\binom{1+1/\beta}{i}(\alpha_\beta(n))^i+n\log c-\frac{1}{2}n^\beta\\+\frac{2}{\beta}\sum_{\text{even }j=0}^{\lfloor 1/\beta\rfloor-1}\binom{1/\beta-1}{j}f_jn^{1-\beta(1+j)}\sum_{i_j=0}^{\ell-j/2-1}\binom{1/\beta-1-j}{i_j}(\alpha_\beta(n))^{i_j}-\zeta(-\beta)
	+o(1)\Big)
\end{multline}
as $n\to\infty$. Indeed, $$\lim_{n\to\infty}n^{1+\beta}\sum_{i\geq \ell+1}\binom{1+1/\beta}{i}(\alpha_\beta(n))^i=0$$
\noindent in view of $2(\ell+1)\beta>1+\beta$ which is a consequence of the equality $\ell=\lfloor (1+\beta)/(2\beta) \rfloor$. Since $2\beta(\ell-j/2)>1-\beta(1+j)$ for all $j\in\mn_0$, we also infer $$\lim_{n\to\infty}n^{1-\beta(1+j)}\sum_{i_j\geq \ell-j/2}\binom{1/\beta-1-j}{i_j}(\alpha_\beta(n))^{i_j}=0$$ for even $j\in[0,\lfloor 1/\beta\rfloor-1]$.

For instance, let $\beta\in(1/3,1)$. Then $\ell=1$,
$\alpha_\beta(u)=0$ for $u>0$ and thereupon
$$
	\exp(\psi(s_n)-s_n\psi^\prime(s_n))=
	\exp\Big(-\frac{1}{1+\beta}n^{1+\beta}+n\log c-\frac{1}{2}n^\beta +\frac{2}{\beta}f_0n^{1-\beta}-\zeta(-\beta)
	+o(1)\Big)
$$
as $n\to\infty$. Now, let $\beta\in(1/4,1/3]$. Then $\ell=2$, $\alpha_\beta(n)=A_{1,\,\beta}n^{-2\beta}=-2(1/\beta-1)c_1n^{-2\beta}$ (the value of $A_{1,\,\beta}$ is given in \eqref{eq:inter2}) and
\begin{multline*}
	\exp(\psi(s_n)-s_n\psi^\prime(s_n))=
	\exp\Big(-\frac{1}{1+\beta}n^{1+\beta}+n\log c-\frac{1}{2}n^\beta+\frac{2}{\beta}f_0n^{1-\beta}\\
	+\frac{1}{\beta}\Big(\frac{1}{\beta}-1\Big)\Big(2\Big(\frac{1}{\beta}-1\Big)c_1^2-4\Big(\frac{1}{\beta}-1\Big)c_1f_0+\Big(\frac{1}{\beta}-2\Big)f_2
	\Big)n^{1-3\beta}-\zeta(-\beta)
	+o(1)\Big)
\end{multline*}
as $n\to\infty$.

\noindent {\bf Step 4.} End of the proof.

By Theorem \ref{thm:caseB}, \eqref{eq:(b)denominator} together with \eqref{eq:expon1(b)} yields as $n\to\infty$
\begin{multline*}
	\mmp\{Y\ge n\}~\sim~\mmp\{Y=n\}~\sim~\Big(\frac{\beta}{2\pi n^{1-\beta}}\Big)^{1/2}
	\exp\Big(-\frac{1}{1+\beta}n^{1+\beta}\\+\frac{\beta n^{1+\beta}}{1+\beta}\sum_{i=2}^\ell\binom{1+1/\beta}{i}(\alpha_\beta(n))^i+n\log c-\frac{1}{2}n^\beta\\+\frac{2}{\beta}\sum_{\text{even }j=0}^{\lfloor 1/\beta\rfloor-1}\binom{1/\beta-1}{j}f_jn^{1-\beta(1+j)}\sum_{i_j=0}^{\ell-j/2-1}\binom{1/\beta-1-j}{i_j}(\alpha_\beta(n))^{i_j}-\zeta(-\beta)
	\Big).
\end{multline*}

\subsection{Auxiliary results for $r_k=c\exp(-k)$}

We start by pointing out some properties of a function that plays a role in Lemmas \ref{lem:beta1} and \ref{lem:beta2} below.

\begin{lemma}\label{lem:h}
	Define a function $h:[0,1]\mapsto\mr$ via
	\begin{equation}\label{eq:h}
		h(b):=\sum_{k\ge0}\frac{1}{1+\eee^{k-b}}-\sum_{k\ge1}\frac{1}{1+\eee^{k+b}}-b-1/2.
	\end{equation}
	Then (i) $h$ is not a constant; (ii) $|h(b)|<2\cdot10^{-7}$ for all $b\in[0,1]$.
\end{lemma}
\begin{proof}
	We first prove (i). The series $\sum_{k\ge0}\frac{\eee^{k-b}}{(1+\eee^{k-b})^2}$ and $\sum_{k\ge1}\frac{\eee^{k+b}}{(1+\eee^{k+b})^2}$ converge uniformly on $[0,1]$. Hence, to obtain a formula for $h^\prime$ we can differentiate termwise the series representing $h$:
	$$h^\prime(b)=\sum_{k\ge0}\frac{\eee^{k-b}}{(1+\eee^{k-b})^2}+\sum_{k\ge1}\frac{\eee^{k+b}}{(1+\eee^{k+b})^2}-1,\quad b\in [0,1].$$ It suffices to show that $h^\prime(b_0)\neq 0$ for some $b_0\in [0,1]$.
	Take $b_0=\log 2$. Then
	$$
	h^\prime(\log 2)=\sum_{k\ge0}\frac{\eee^k/2}{(1+\eee^k/2)^2}+\sum_{k\ge1}\frac{2\eee^k}{(1+2\eee^k)^2}-1.
	$$
	Since
	$$
	\sum_{k\ge19}\Big(\frac{\eee^k/2}{(1+\eee^k/2)^2}+\frac{2\eee^k}{(1+2\eee^k)^2}\Big)\le 5/2\sum_{k\ge19}\eee^{-k}<5\eee^{-19}<3\cdot10^{-8}
	$$
	and
	$$
	\sum_{k=1}^{18}\Big(\frac{\eee^k/2}{(1+\eee^k/2)^2}+\frac{2\eee^k}{(1+2\eee^k)^2}\Big)<0.\underbrace{7\ldots7}_{7 \text{ times}},
	$$
	we infer
	$$
	h^\prime(\log 2)<3\cdot10^{-8}+0.\underbrace{7\ldots7}_{7 \text{ times}}-7/9<3\cdot10^{-8}-7\cdot10^{-8}<0.
	$$
	
	Now we prove (ii). Noting that $h(0)=0$ we shall use an inequality: for $x\in[0,1]$
	\begin{equation}\label{eq:der_ineq}
		|h(x)|^2=|h(x)-h(0)|^2=\Big|\int_0^xh^\prime(b)\dd b\Big|^2\le x\int_0^x\big(h^\prime(b)\big)^2\dd b\le \int_0^1\big(h^\prime(b)\big)^2\dd b,
	\end{equation}
	where the first inequality follows by an application of the Cauchy–Schwarz inequality. Calculating $\int_0^1\big(h^\prime(b)\big)^2\dd b$ is mostly routine. The only nontrivial action is to notice
	that all sums (including double sums) of integrals are telescopic. The obtained result is
	$$
	\int_0^1\big(h^\prime(b)\big)^2\dd b=-\frac{5}{6}+2\sum_{m\ge1}\frac{m\eee^m(\eee^m+1)}{(\eee^m-1)^3}-4\sum_{m\ge1}\frac{\eee^m}{(\eee^m-1)^2}.
	$$
	With this at hand, we proceed as follows:
	$$
	2\sum_{m\ge100}\frac{m\eee^m(\eee^m+1)}{(\eee^m-1)^3}\le2\sum_{m\ge100}\frac{\eee^m\eee^{11m/10}}{(\eee^{9m/10})^3}=2\frac{\eee^{-60}}{1-\eee^{-0.6}}<3\cdot10^{-26}
	$$
	and
	$$
	2\sum_{m=1}^{99}\frac{m\eee^m(\eee^m+1)}{(\eee^m-1)^3}-4\sum_{m=1}^{99}\frac{\eee^m}{(\eee^m-1)^2}<0.8\underbrace{3\ldots3}_{12 \text{ times}}6,
	$$
	\noindent whence
	$$
	\int_0^1\big(h^\prime(b)\big)^2\dd b<3\cdot10^{-26}+0.8\underbrace{3\ldots3}_{12 \text{ times}}6-5/6<3\cdot10^{-14}.
	$$
	Invoking inequality \eqref{eq:der_ineq} completes the proof.
\end{proof}

Lemmas \ref{lem:beta1} and \ref{lem:beta2} provide asymptotic expansions for functional series which appear in the proof of Theorem \ref{thm:main}(c).

\begin{lemma}\label{lem:beta1}
	With the function $h$ as defined in \eqref{eq:h},
	$$
	\sum_{k\ge1} \frac{1}{1+a\eee^k}=\log(1/a)-1/2+h(\{\log(1/a)\})+O(a),\quad a\to0+.
	$$
\end{lemma}
\begin{proof}
	For fixed $a\in(0,1)$, put
	$n:=\lfloor\log(1/a)\rfloor$ and $b:=\{\log(1/a)\}$. Then
	\begin{align*}
		\sum_{k\ge1} \frac{1}{1+a\eee^k}&=\sum_{k\ge1} \frac{1}{1+\eee^{k-n-b}}=\sum_{k\ge0} \frac{1}{1+\eee^{k-b}}+\sum_{k=1}^{n-1} \frac{1}{1+\eee^{k-n-b}}\\&=\sum_{k\ge0} \frac{1}{1+\eee^{k-b}}+n-1-\sum_{k=1}^{n-1} \frac{\eee^{-k-b}}{1+\eee^{-k-b}}\\&=\sum_{k\ge0} \frac{1}{1+\eee^{k-b}}+n-1-\sum_{k\ge1} \frac{1}{1+\eee^{k+b}}+\sum_{k\ge n} \frac{1}{1+\eee^{k+b}}\\&=n-1+h(b)+b+1/2+\sum_{k\ge n} \frac{1}{1+\eee^{k+b}}.
	\end{align*}
	It remains to note that $\sum_{k\ge n} \frac{1}{1+\eee^{k+b}}\le\sum_{k\ge n} \eee^{-k-b}<2\eee^{-n-b}=O(a)$ as $a\to0+$.
\end{proof}

\begin{lemma}\label{lem:beta2}
	With $h_2$ given by $h_2(b):=\int_0^bh(x)\dd x$ for $b\in [0,1]$, where $h$ is as defined in \eqref{eq:h}, as $a\to0+$
	$$
	\sum_{k\ge1} \log\Big(1+\frac{1}{a\eee^k}\Big)=\big(\log(1/a)\big)^2/2-\log(1/a)/2+c_1+h_2(\{\log(1/a)\})+O(a),
	$$
	where $c_1:=\log 2+2\sum_{m\ge1}\log(1+\eee^{-m})$.
\end{lemma}
\begin{proof}
	As in Lemma \ref{lem:beta1}, fix $a\in(0,1)$ and put
	$n:=\lfloor\log(1/a)\rfloor$ and $b:=\{\log(1/a)\}$. Then
	\begin{multline*}
		S(a):=\sum_{k\ge1} \log\Big(1+\frac{1}{a\eee^k}\Big)=\sum_{k\ge1} \log(1+\eee^{n+b-k})=\sum_{k\ge0} \log(1+\eee^{-k+b})\\+\sum_{k=1}^{n-1} \log(1+\eee^{k+b})
		=\sum_{k\ge0} \log(1+\eee^{-k+b})+\sum_{k=1}^{n-1}\Big( \log(1+\eee^{-k-b})+k+b\Big)\\
		=\!\sum_{k\ge0} \log(1+\eee^{-k+b})+\sum_{k\ge1} \log(1+\eee^{-k-b})+n(n-1)/2+b(n-1)-\sum_{k\ge n}\log(1+\eee^{-k-b}).
	\end{multline*}
	Note that $\sum_{k\ge n}\log(1+\eee^{-k-b})\le\sum_{k\ge n}\eee^{-k-b}<2\eee^{-n-b}=O(a)$ as $a\to0+$. Next, consider a function $f:[0,1]\to\mr$ defined via
	$$
	f(b):=\sum_{k\ge0} \log(1+\eee^{-k+b})+\sum_{k\ge1} \log(1+\eee^{-k-b}).
	$$
	The series $\sum_{k\ge0}\frac{\eee^{-k+b}}{1+\eee^{-k+b}}$ and $\sum_{k\ge1}\frac{\eee^{-k-b}}{1+\eee^{-k-b}}$ converge uniformly on $[0,1]$. Hence, differentiating termwise the series representing $f$ we obtain
	$$
		f^\prime(b)=\sum_{k\ge0}\frac{\eee^{-k+b}}{1+\eee^{-k+b}}-\sum_{k\ge1}\frac{\eee^{-k-b}}{1+\eee^{-k-b}}=\sum_{k\ge0}\frac{1}{1+\eee^{k-b}}-\sum_{k\ge1}\frac{1}{1+\eee^{k+b}}=h(b)+b+1/2
	$$
	and thereupon
	$$
	f(b)=\int_0^bf^\prime(x)\dd x+f(0)=h_2(b)+b^2/2+b/2+c_1.
	$$
	Summarizing,
	$$
	S(a)=(n+b)^2/2-(n+b)/2+c_1+h_2(b)+O(a),\quad a\to0+.
	$$
	
\end{proof}

\subsection{Proof of Theorem \ref{thm:main}(c)}

{\bf Step 1.} Determining $s_n$, the solution to $\psi^{\prime}(s)=n$.

Recalling that $r_k=c\eee^{-k}$ for $k\in\mn$,
we start with
$$
\psi^\prime(s)=\sum_{k\ge 1}\frac{r_k\eee^s}{r_k\eee^s+1-r_k}=\frac{\eee^s}{\eee^s-1}\sum_{k\ge 1}\frac{1}{1+
	(c(\eee^s-1))^{-1}\eee^k},\quad s\in\mr.
$$
An application of Lemma \ref{lem:beta1} with $a=(c(\eee^s-1))^{-1}$ yields
\begin{equation*}
	\psi^\prime(s)=s+\log c-1/2+h(\{\log c+\log(\eee^s-1)\})+O(s\eee^{-s}), \quad s\to\infty,
\end{equation*}
where $h$ is as defined in \eqref{eq:h} and, by Lemma \ref{lem:h}, $h(b)$ is of order $10^{-7}$ for $b\in[0,1]$.

For large enough $n$, $s=s_n$ is a solution to the equation $\psi^\prime(s)=n$, that is,
\begin{equation*}
	s+\log c-1/2+h(\{\log c+\log(\eee^s-1)\})+O(s\eee^{-s})=n.
\end{equation*}
Hence,
\begin{equation}\label{eq:s(c)}
	s_n =n+1/2-\log c-\varepsilon_n+O(n\eee^{-n}),\quad n\to\infty,
\end{equation}
where $\varepsilon_n:=h(\{\log c+s_n+\log(1-\eee^{-s_n})\})$. We intend to prove that $\varepsilon_n=o(1)$ as $n\to\infty$. To this end, consider any limit point $\varepsilon$, that is $\lim_{k\to\infty}\varepsilon_{n_k}=\varepsilon$ for some subsequence $(n_k)$. By Lemma \ref{lem:h}, $|\varepsilon_{n_k}|<1/4$, whence $|\varepsilon|\le 1/4$. With this at hand, since $\varepsilon_{n_k}=h(\{n_k+1/2-\varepsilon_{n_k}+o(1)\})=h(1/2-\varepsilon_{n_k}+o(1))$, we conclude that the limit $\varepsilon$ should satisfy the equation $\varepsilon=h(1/2-\varepsilon)$. Recalling~\eqref{eq:h} the latter is equivalent to
$$
\sum_{k\ge0}\frac{1}{1+\eee^{k+\varepsilon-1/2}}-\sum_{k\ge1}\frac{1}{1+\eee^{k+1/2-\varepsilon}}=1.
$$
The function on the left-hand side is strictly decreasing. Therefore,
$\varepsilon=0$ is the unique solution to $\varepsilon=h(1/2-\varepsilon)$. Summarizing,
\begin{equation}\label{eq:s(c)1}
	s_n =n+1/2-\log c+o(1),\quad n\to\infty,
\end{equation}

\noindent {\bf Step 2.}
Verifying the conditions of Theorem \ref{thm:caseC}.

\noindent (a) Fix $k\in\mn$. Invoking \eqref{eq:s(c)1}, as $n\to\infty$
$$
r_{n+k}\eee^{s_n}=\eee^{-k+1/2+o(1)}\to \eee^{-k+1/2}=:p_k$$ and $$\sup_{n\ge N}r_{n+k}\eee^{s_n}\le \eee^{-k+1/2}+\eee^{-k+1/2}
$$
for $N\in\mn$ such that $\eee^{o(1)}\le 2$ for all $n\ge N$. Here, $\alpha_k=\eee^{-k+1/2}$. Plainly, $$\sum_{k\geq 1}(p_k+\alpha_k)=2\sum_{k\ge1}\eee^{-k+1/2}<\infty.$$

\noindent (b) Fix $k\in\mn_0$. Again, invoking \eqref{eq:s(c)1}, as $n\to\infty$
$$
r_{n-k}^{-1}\eee^{-s_n}=\eee^{-k-1/2+o(1)}\to \eee^{-k-1/2}=:q_k$$ and $$\sup_{n\ge N}r_{n-k}^{-1}\eee^{-s_n}\le \eee^{-k-1/2}+\eee^{-k-1/2}
$$
for the same $N\in\mn$ as in (a). Here, $\beta_k=\eee^{-k-1/2}$, and
$$\sum_{k\geq 0}(q_k+\beta_k)=2\sum_{k\ge0
}\eee^{-k-1/2}<\infty.$$

\noindent (c) $q_0=\eee^{-1/2}>0$.

\noindent {\bf Step 3.} Asymptotic behavior of $\psi$.

By Lemma \ref{lem:beta2}, with the same $a=(c(\eee^s-1))^{-1}$ as before,
\begin{multline}\label{eq:psi_k(c)}
	\psi(s)=\sum_{k\ge 1}\log(1+c(\eee^s-1)\eee^{-k})=s^2/2+s(\log c-1/2)\\+(\log^2c)/2-(\log c)/2+c_1+h_2(\{\log c+\log(\eee^s-1)\})+O(s\eee^{-s}),\quad s\to\infty,
\end{multline}
where $h_2$ is as defined in Lemma \ref{lem:beta2} and, similarly to $h$, $h_2(b)$ is of order $10^{-7}$ for $b\in[0,1]$.

\noindent {\bf Step 4.} End of the proof.

In view of \eqref{eq:psi_k(c)},
\eqref{eq:s(c)} and $\psi^\prime(s_n)=n$,
\begin{multline*}
	\exp(\psi(s_n)-s_n\psi^\prime(s_n))
	=\exp\Big(-n^2/2+(\log c-1/2)n-1/8+\varepsilon^2_n/2\\+c_1+h_2\big(1/2-\varepsilon_n+O(n\eee^{-n})\big)+O(n^2\eee^{-n})\Big),\quad n\to\infty.
\end{multline*}
Recalling that $\lim_{n\to\infty}\varepsilon_n=0$ and $h_2$ is continuous, we arrive at as $n\to\infty$
\begin{multline*}
	\exp(\psi(s_n)-s_n\psi^\prime(s_n))
	\sim\exp\Big(-n^2/2+(\log c-1/2)n-1/8+c_1+h_2(1/2)\Big).
\end{multline*}
With this at hand, Theorem \ref{thm:caseC} yields
\begin{multline*}
	\mmp\{Y\ge n\}~\sim~\mmp\{Y=n\}~\sim~ c_0\exp(\psi(s_n)-s_n\psi^\prime(s_n)) \\\sim c_0\exp\Big(-n^2/2+(\log c-1/2)n-1/8+c_1+h_2(1/2)\Big),\quad n\to\infty,
\end{multline*}
where $c_0\in(0,1)$ is a constant defined in Theorem \ref{thm:caseC} with $p_k=\eee^{-k+1/2}$ for $k\in\mn$ and $q_k=\eee^{-k-1/2}$ for $k\in\mn_0$.

\subsection{Auxiliary results for $r_k=c\exp(-k^\beta)$, $\beta>1$}

Lemmas \ref{lem:d1} and \ref{lem:d2} provide asymptotic expansions for functional series which appear in the proof of Theorem \ref{thm:main}(d).

\begin{lemma}\label{lem:d1}
	For $\beta>1$,
	$$
	\sum_{k\ge1}\frac{1}{1+a\eee^{k^\beta}}=c_a-1+\frac{1}{1+a\eee^{c_a^\beta}}+\frac{1}{1+a\eee^{(c_a+1)^\beta}}+O(\eee^{-\frac{\beta}{2}(\log 1/a)^{1-1/\beta}}),\,\, a\to0+,
	$$
	where $c_a=\lfloor(\log1/a)^{1/\beta}\rfloor$.
\end{lemma}
\begin{proof}
	It suffices to show that
	\begin{equation}\label{eq:d1a}
		\sum_{k\ge c_a+2}\frac{1}{1+a\eee^{k^\beta}}=O(\eee^{-\frac{\beta}{2}(\log 1/a)^{1-1/\beta}}),\quad a\to0+
	\end{equation}
	and
	\begin{equation}\label{eq:d1b}
		\sum_{k=1}^{ c_a-1}\frac{a\eee^{k^\beta}}{1+a\eee^{k^\beta}}=O(\eee^{-\frac{\beta}{2}(\log 1/a)^{1-1/\beta}}),\quad a\to0+.
	\end{equation}
	{\sc Proof of \eqref{eq:d1a}.} As a preparation, note that
	\begin{multline*}
		\eee^{(c_a+2)^\beta}\sum_{k\ge c_a+3}\eee^{-k^\beta}\le \eee^{(c_a+2)^\beta}\int_{c_a+2}^\infty\eee^{-x^\beta}\dd x\\=\beta^{-1}\eee^{(c_a+2)^\beta}\int_{(c_a+2)^\beta}^\infty y^{1/\beta-1}\eee^{-y}\dd y\leq \beta^{-1} (c_a+2)^{1-\beta}~\to~ 0, \quad a\to 0+.
	\end{multline*}
	With this at hand, we obtain as $a\to0+$
	\begin{equation}\label{eq:dlema1}
		\sum_{k\ge c_a+2}\frac{1}{1+a\eee^{k^\beta}}\le a^{-1}\sum_{k\ge c_a+2}\eee^{-k^\beta}~\sim~ a^{-1}\eee^{-(c_a+2)^\beta}=O\big(\eee^{-\beta(\log 1/a)^{1-1/\beta}}\big).
	\end{equation}
	To justify the equality, write $$a^{-1}\eee^{-(c_a+2)^\beta}\leq a^{-1}\eee^{-(1+(\log 1/a)^{-1/\beta})^\beta\log (1/a)}\le \eee^{-\beta(\log 1/a)^{1-1/\beta}}$$ having utilized for the last inequality $(1+x)^\beta\geq 1+\beta x$ for $x\geq 0$.
	
	\noindent{\sc Proof of \eqref{eq:d1b}.} Arguing similarly we infer as $a\to0+$
	\begin{multline*}
		\eee^{-(c_a-1)^\beta}\sum_{k=1}^{c_a-2}\eee^{k^\beta}\le(c_a-2)\eee^{-(c_a-1)^\beta+(c_a-2)^\beta}\\=(c_a-2)\eee^{-(c_a-1)^\beta(1-(1-(c_a-1)^{-1})^\beta)}\sim~ c_a \eee^{-(c_a-1)^\beta(\beta(c_a-1)^{-1}+O((c_a-1)^{-2}))}\to0.
	\end{multline*}
	As a consequence,
	\begin{equation}\label{eq:dlemma2}
		\sum_{k=1}^{ c_a-1}\frac{a\eee^{k^\beta}}{1+a\eee^{k^\beta}}\le a\sum_{k=1}^{ c_a-1}\eee^{k^\beta}~\sim~ a\eee^{(c_a-1)^\beta}=O\big(\eee^{-\frac{\beta}{2}(\log 1/a)^{1-1/\beta}}\big),\quad a\to 0+,
	\end{equation}
	where the equality is ensured by
	\begin{multline*}
		a\eee^{(c_a-1)^\beta}\le a\eee^{((\log1/a)^{1/\beta}-1)^\beta}= a\eee^{\log(1/a)(1-(\log1/a)^{-1/\beta})^\beta}\\=a\eee^{\log(1/a)(1-\beta(\log1/a)^{-1/\beta}+O(\log1/a)^{-2/\beta})}\le \eee^{-\frac{\beta}{2}(\log 1/a)^{1-1/\beta}}
	\end{multline*}
	for small $a>0$.
\end{proof}

\begin{lemma}\label{lem:d2}
	For $\beta>1$, as $a\to0+$,
	\begin{multline*}
		\sum_{k\ge1}\log\Big(1+\frac{1}{a\eee^{k^\beta}}\Big)=\log\Big(1+\frac{1}{a\eee^{(c_a+1)^\beta}}\Big)+\log(1+a\eee^{c_a^\beta})+c_a\log(1/a)-\frac{(\log1/a)^{1+1/\beta}}{\beta+1}\\-\zeta(-\beta)-\frac{1}{\beta+1}\sum_{k=1}^{1+\lfloor\beta\rfloor}(-1)^k\binom{\beta+1}{k}B_k(\{(\log1/a)^{1/\beta}\})(\log1/a)^{1-(k-1)/\beta}+o(1),
	\end{multline*}
	where $c_a=\lfloor (\log 1/a)^{1/\beta}\rfloor$ and, for $k\in\mn$, $B_k$ is the $k$th Bernoulli polynomial.
\end{lemma}
\begin{proof}
	In view of \eqref{eq:dlema1},
	\[
	\sum_{k\ge c_a+2}\log\Big(1+\frac{1}{a\eee^{k^\beta}}\Big)\le a^{-1}\sum_{k\ge c_a+2}\eee^{-k^\beta}\to0, \quad a\to0+.
	\]
	\noindent Further,
	$$
	\sum_{k=1}^{ c_a}\log\Big(1+\frac{1}{a\eee^{k^\beta}}\Big)=\sum_{k=1}^{c_a}\log(1+a\eee^{k^\beta})-\sum_{k=1}^ {c_a}\log(a\eee^{k^\beta})=:S_1(a)-S_2(a).
	$$
	Note that the summand of $S_1(a)$ with $k=c_a$ appears in the final result. The rest of $S_1(a)$ is negligible:
	$$
	\sum_{k=1}^{c_a-1}\log(1+a\eee^{k^\beta})\le a\sum_{k=1}^{c_a-1}\eee^{k^\beta}\to0,\quad a\to0+
	$$
	as follows from \eqref{eq:dlemma2}. Finally, as $a\to0+$
	\begin{multline*}
		-S_2(a)=c_a\log(1/a)-\sum_{k=1}^ {c_a}k^{\beta}=c_a\log(1/a)-\frac{(\log1/a)^{1+1/\beta}}{\beta+1}-\zeta(-\beta)\\-\frac{1}{\beta+1}\sum_{k=1}^{1+\lfloor\beta\rfloor}(-1)^k\binom{\beta+1}{k}B_k(\{(\log1/a)^{1/\beta}\})(\log1/a)^{1-(k-1)/\beta}+o(1),
	\end{multline*}
	where the last equality follows directly from Theorem 14 in \cite{Schumacher:2022}.
\end{proof}
\begin{rem}\label{rem:d}
	Neither in Lemma \ref{lem:d1}, nor in Lemma \ref{lem:d2} were we able to analyze the behavior of terms with $k=c_a$ and $k=c_a+1$. Here is the reason of this failure. For instance, if along a chosen subsequence~$(a_n)$ the values of $(\log 1/a)^{1/\beta}$ are integer, then $c_a=(\log 1/a)^{1/\beta}$, whence $a\eee^{c_a^\beta}=1$. But if one chooses a subsequence $(a_n)$ along which the values of $(\log1/a)^{1/\beta}-1/2$ are integer, then $c_a=(\log 1/a)^{1/\beta}-1/2$, whence ${\lim_{a\to 0+}a\eee^{c_a^\beta}=0}$.
\end{rem}

\subsection{Proof of Theorem \ref{thm:main}(d)}\label{sect:proof d}

\noindent{\bf Step 1.} Determining $s_n$, the solution to $\psi^{\prime}(s)=n$.

Under the present assumption $r_k=c\eee^{-k^\beta}$ for $\beta>1$ and $k\in\mn$
we adopt the approach suggested in Remark \ref{rem}. The reason is discussed in Remark \ref{rem:d}.
\noindent Recall that
$$
\psi^\prime(s)=\sum_{k\ge 1}\frac{r_k\eee^s}{r_k\eee^s+1-r_k}=\frac{\eee^s}{\eee^s-1}\sum_{k\ge 1}\frac{1}{1+
	(c(\eee^s-1))^{-1}\eee^{k^\beta}},\quad s\in\mr.
$$
In view of Lemma \ref{lem:d1},
$$
\sum_{k\ge 1}\frac{1}{1+a\eee^{k^\beta}}=O\big(\log (1/a)^{1/\beta}\big),\quad a\to0+.
$$
Using this relation with $a=(c(\eee^s-1))^{-1}$ yields
$$
\psi^\prime(s)=\sum_{k\ge 1}\frac{1}{1+
	(c(\eee^s-1))^{-1}\eee^{k^\beta}}+O(\eee^{-s}s^{1/\beta}),\quad s\to\infty.
$$
\noindent We claim that $s_n=n^{\beta}+n^{(\beta-1)/2}$ is a solution to $\psi^\prime(s)=n+O(\eee^{-n^{(\beta-1)/2}})$.
Indeed, with $a=(c(\eee^{s_n}-1))^{-1}$, as $n\to\infty$,
\begin{multline}\label{eq:c_a=n}
	c_a=\lfloor (\log1/a)^{1/\beta}\rfloor=\lfloor (\log c+n^\beta+n^{(\beta-1)/2}+\log(1-\eee^{-n^\beta- n^{(\beta-1)/2}}))^{1/\beta}\rfloor\\=\Big\lfloor n\Big(1+\beta^{-1}(n^{-(\beta+1)/2}+n^{-\beta}(\log c+\log (1-\eee^{-n^\beta-n^{(\beta-1)/2}})))+O(n^{-(\beta+1)})\Big)\Big\rfloor=n,
\end{multline}

\begin{equation}\label{eq:c_a}
	a\eee^{c_a^\beta}=c^{-1}\eee^{n^\beta-s_n}(1-\eee^{-s_n})^{-1}\le 2c^{-1}\eee^{n^\beta-s_n}=O(\eee^{-n^{(\beta-1)/2}})
\end{equation}
and
\begin{equation}\label{eq:c_a+1}
	a^{-1}\eee^{-(c_a+1)^\beta}=c\eee^{-(n+1)^\beta+s_n}(1-\eee^{-s_n})\le c\eee^{-\beta n^{\beta-1}+O(n^{\beta-2})+n^{(\beta-1)/2}}=o(\eee^{-\beta n^{\beta-1}/2}).
\end{equation}
With this at hand, invoking Lemma \ref{lem:d1} we obtain, as $n\to\infty$,
$$
	\psi^\prime(s_n)=n-1+O(\eee^{-n^{(\beta-1)/2}})+1+o(\eee^{-\beta n^{\beta-1}/2})+O(\eee^{-\beta n^{\beta-1}/4})+o(\eee^{-n^\beta})=n+O(\eee^{-n^{(\beta-1)/2}}).
$$

\noindent {\bf Step 2.}
Verifying the condition (c) of Theorem \ref{thm:caseA}.

\noindent Recalling that $s_n=n^\beta+n^{(\beta-1)/2}$ we infer
$$
\eee^{-s_n}\sum_{k=1}^nr_k^{-1}\le c^{-1}\eee^{-s_n}n\eee^{n^\beta}\to 0,\quad n\to\infty.
$$

\noindent {\bf Step 3.} End of the proof.

To determine the asymptotic behavior of $\psi(s_n)=\sum_{k\ge1}\log(1+c\eee^{-k^\beta}(\eee^{s_n}-1))$, we apply Lemma \ref{lem:d2} with $a=(c(\eee^{s_n}-1))^{-1}$ and derive the asymptotic behavior of each term on the right-hand side of the equality in Lemma \ref{lem:d2} (excluding, of course, $-\zeta(-\beta)$). As $n\to\infty$,
\begin{itemize}
	\item $\log\big(1+\frac{1}{a\eee^{(c_a+1)^\beta}}\big)=O(\eee^{-\beta n^{\beta-1}/2})$ by \eqref{eq:c_a+1};
	\item $\log(1+a\eee^{c_a^\beta})=O(\eee^{-n^{(\beta-1)/2}})$ by \eqref{eq:c_a};
	\item $c_a\log(1/a)=n(\log c+ n^{\beta}+n^{(\beta-1)/2})+o(1)$ by \eqref{eq:c_a=n};
	\item $\displaystyle-\frac{(\log 1/a)^{1+1/\beta}}{\beta+1}\\=-\frac{1}{\beta+1}n^{\beta+1}\Big(1+n^{-(\beta+1)/2}+n^{-\beta}\big(\log c+\log(1-\eee^{-n^\beta- n^{(\beta-1)/2}})\big)\Big)^{1+1/\beta}\\=-\frac{1}{\beta+1}n^{\beta+1}\Big(1+\Big(1+\frac{1}{\beta}\Big)(n^{-(\beta+1)/2}+n^{-\beta}\log c)+
	\frac{\beta+1}{2\beta^2} n^{-(\beta+1)}\Big)+o(1)$;
	\item
	To analyze the last principal term, we first note that $(\log 1/a)^{1-(k-1)/\beta}= n^{\beta-(k-1)}+(1-(k-1)/\beta)n^{(\beta+1)/2-k}+o(1)$. Also, $\{(\log1/a)^{1/\beta}\}=\beta^{-1}(n^{-(\beta-1)/2}+(\log c) n^{1-\beta})-(2\beta)^{-1}(1-\beta^{-1})n^{-\beta}+o(n^{-\beta})$.
	Summarizing,
	\begin{align*}
		-&\frac{1}{\beta+1}\sum_{k=1}^{\lfloor \beta\rfloor+1}(-1)^k \binom{\beta+1}{k}B_k(\{(\log1/a)^{1/\beta}\})(\log1/a)^{1-(k-1)/\beta}\\=&-\frac{1}{\beta+1}\sum_{k=1}^{1+\lfloor \beta\rfloor}(-1)^k\binom{\beta+1}{k} B_k\Big(\frac{1}{\beta}n^{(1-\beta)/2}
		+\frac{\log c}{\beta}n^{1-\beta}
		-\frac{\beta-1}{2\beta^2}n^{-\beta}
		\Big)\\&\times(n^{\beta-(k-1)}+\Big(1-\frac{k-1}{\beta}\Big)\1_{\{k\leq \lfloor (1+\beta)/2\rfloor\}}n^{(\beta+1)/2-k})+o(1).
	\end{align*}
\end{itemize}
Combining all fragments together we obtain
\begin{multline*}
	\exp (\psi(s_n)-s_n
	\psi^\prime(s_n))~\sim~ \exp\Big(-\frac{1}{\beta+1}n^{\beta+1}-\frac{1}{\beta}n^{(\beta+1)/2}+\log c\Big(1-\frac{1}{\beta}\Big)n-\frac{1}{2\beta^2}\\-\frac{1}{\beta+1}\sum_{k=1}^{\lfloor \beta\rfloor+1}(-1)^k\binom{\beta+1}{k}B_k\Big(\frac{1}{\beta}n^{-(\beta-1)/2}+\frac{\log c}{\beta} n^{1-\beta}-\frac{\beta-1}{2\beta^2}n^{-\beta}\Big)\\\times\Big(n^{\beta-(k-1)}+\Big(1-\frac{k-1}{\beta}\Big)\1_{\{k\leq \lfloor (\beta+1)/2\rfloor\}}n^{(\beta+1)/2-k}\Big)\Big),\quad n\to\infty.
\end{multline*}
Now the statement of Theorem \ref{thm:main}(d) is a direct consequence of Theorem \ref{thm:caseA}.

\noindent {\bf Acknowledgement}. We thank the anonymous referee for a thorough review and constructive suggestions. In particular, these led to significant improvements in the proofs of Theorem~\ref{thm:caseA} and Lemma~\ref{lem:module}. The work of A. Iksanov was supported by the National Research Foundation of Ukraine (project
2023.03/0059 ‘Contribution to modern theory of random series’). The work of V. Kotelnikova was supported by a grant from the IMU, supported by the Friends of IMU (FIMU), under the IMU Breakout Fellowship Program.

\end{document}